\pdfoutput=1
\RequirePackage{ifpdf}
\ifpdf 
\documentclass[pdftex]{sigma}
\else
\documentclass{sigma}
\fi

\usepackage{tikz}
\usetikzlibrary{shapes.geometric,arrows,decorations.pathmorphing,decorations.markings,patterns}

\def\R{{\mathbb R}}
\def\C{{\mathbb C}}
\def\Z{{\mathbb Z}}
\def\F{{\mathbb F}}
\def\Q{{\mathbb Q}}
\def\A{{\mathcal{A}}}
\def\Gr{{\rm Gr}}
\def\bF{{\bar \F}}
\def\oPi{{\mathring{\Pi}}}
\def\oGr{{\mathring{\Gr}}}

\def\Xo{{\mathring{X}}}
\def\Zo{{\mathring{Z}}}
\def\Ao{{\mathring{\A}}}
\def\oV{{\mathring{V}}}
\def\u{{\mathbf{u}}}
\def\Spec{\operatorname{Spec}}
\def\Hom{\operatorname{Hom}}
\def\Ext{\operatorname{Ext}}
\def\prin{{\operatorname{prin}}}
\def\univ{{\operatorname{univ}}}
\def\tB{{\tilde B}}
\def\sp{\operatorname{span}}
\def\wt{\operatorname{wt}}
\def\E{{\mathbb E}}
\def\I{{\mathcal{I}}}
\def\U{{\mathbb U}}
\def\M{{\mathcal{M}}}
\def\tM{{\widetilde {\mathcal{M}}}}
\def\bM{{\overline {\mathcal{M}}}}
\def\tD{{\tilde D}}
\def\tB{{\tilde B}}
\def\tc{{\tilde \c}}
\def\tI{{\tilde I}}
\def\z{{\mathbf z}}
\def\x{{\mathbf x}}
\def\y{{\mathbf y}}

\def\v{{\mathbf v}}
\def\X{{\mathbf X}}
\def\N{{\mathcal{N}}}
\def\one{{\mathbf 1}}

\def\Trop{\operatorname{Trop}}
\def\SL{\operatorname{SL}}
\def\GL{\operatorname{GL}}
\def\dlog{\operatorname{dlog}}
\def\g{{\mathbf g}}
\def\Y{{\mathbf Y}}

\def\tomega{{\tilde \omega}}
\def\tgamma{{\tilde \gamma}}

\def\tPi{{\tilde \Pi}}
\def\c{{\mathbf c}}

\renewcommand{\P}{{\mathbb P}}
\newcommand{\sslash}{\mathbin{/\mkern-6mu/}}

\numberwithin{equation}{section}

\newtheorem{Theorem}{Theorem}[section]
\newtheorem{Corollary}[Theorem]{Corollary}
\newtheorem{Lemma}[Theorem]{Lemma}
\newtheorem{Proposition}[Theorem]{Proposition}
\newtheorem{Conjecture}[Theorem]{Conjecture}
 { \theoremstyle{definition}
\newtheorem{Definition}[Theorem]{Definition}
\newtheorem{Example}[Theorem]{Example}
\newtheorem{Remark}[Theorem]{Remark}
\newtheorem{question}[Theorem]{Question}
\newtheorem{problem}[Theorem]{Problem}}

\begin{document}
\allowdisplaybreaks

\newcommand{\arXivNumber}{2005.11419}

\renewcommand{\PaperNumber}{092}

\FirstPageHeading

\ShortArticleName{Cluster Configuration Spaces of Finite Type}

\ArticleName{Cluster Configuration Spaces of Finite Type}

\Author{Nima ARKANI-HAMED~$^{\rm a}$, Song HE~$^{\rm bcde}$ and Thomas LAM~$^{\rm f}$}
\AuthorNameForHeading{N.~Arkani-Hamed, S.~He and T.~Lam}

\Address{$^{\rm a)}$~School of Natural Sciences, Institute for Advanced Studies, Princeton, NJ, 08540, USA}
\EmailD{\href{mailto:arkani@ias.edu}{arkani@ias.edu}}

\Address{$^{\rm b)}$~CAS Key Laboratory of Theoretical Physics, Institute of Theoretical Physics,\\
\hphantom{$^{\rm b)}$}~Chinese Academy of Sciences, Beijing, 100190, China}
\EmailD{\href{mailto:songhe@itp.ac.cn}{songhe@itp.ac.cn}}
\Address{$^{\rm c)}$~School of Fundamental Physics and Mathematical Sciences,\\
\hphantom{$^{\rm c)}$}~Hangzhou Institute for Advanced Study, UCAS, Hangzhou 310024, China}
\Address{$^{\rm d)}$~ICTP-AP International Centre for Theoretical Physics Asia-Pacific,\\
\hphantom{$^{\rm d)}$}~Beijing/Hangzhou, China}
\Address{$^{\rm e)}$~School of Physical Sciences, University of Chinese Academy of Sciences,\\
\hphantom{$^{\rm e)}$}~No.19A Yuquan Road, Beijing 100049, China}

\Address{$^{\rm f)}$~Department of Mathematics, University of Michigan,\\
\hphantom{$^{\rm f)}$}~530 Church St, Ann Arbor, MI 48109, USA}
\EmailD{\href{mailto:tfylam@umich.edu}{tfylam@umich.edu}}
\URLaddressD{\url{http://math.lsa.umich.edu/~tfylam/}}

\ArticleDates{Received January 05, 2021, in final form October 04, 2021; Published online October 16, 2021}

\Abstract{For each Dynkin diagram $D$, we define a ``cluster configuration space'' ${\mathcal{M}}_D$ and a partial compactification ${\widetilde {\mathcal{M}}}_D$. For $D = A_{n-3}$, we have ${\mathcal{M}}_{A_{n-3}} = {\mathcal{M}}_{0,n}$, the configuration space of $n$ points on ${\mathbb P}^1$, and the partial compactification ${\widetilde {\mathcal{M}}}_{A_{n-3}}$ was studied in this case by Brown. The~space ${\widetilde {\mathcal{M}}}_D$ is a smooth affine algebraic variety with a stratification in bijection with the faces of the Chapoton--Fomin--Zelevinsky generalized associahedron. The~regular functions on ${\widetilde {\mathcal{M}}}_D$ are generated by coordinates $u_\gamma$, in bijection with the cluster variables of type~$D$, and the relations are described completely in terms of the compatibility degree function of the cluster algebra. As an application, we define and study cluster algebra analogues of tree-level open string amplitudes.}

\Keywords{configuration space; cluster algebras; generalized associahedron; string amplitudes}
\Classification{05E14; 13F60; 14N99; 81T30}

\section{Introduction}

\medskip\noindent{\large\bf 1.1.} The configuration space $\M_{0,n}$ of $n$ distinct points on $\P^1$ is a smooth affine algebraic variety of dimension $n{-}3$, and it has a very well-studied Deligne--Knudsen--Mumford compactification $\bM_{0,n}$, which is a smooth projective algebraic variety. The~boundary of $\bM_{0,n}$ consists of $2^{n{-}1}{-}n{-}1$ divisors, satisfying factorization: each divisor is itself a product $\bM_{0,n_1} \times \bM_{0,n_2}$, where $n_1{+}n_2 =n{+}2$.

The real points $\M_{0,n}(\R)$ have the structure of a smooth real manifold with $(n{-}1)!/2$ connected components. Fixing once and for all a (dihedral) ordering on $n$ points, we let $(\M_{0,n})_{>0} \subset \M_{0,n}(\R)$ denote the connected component where the $n$ points are ordered on $\P^1(\R) = S^1$. The~closure $(\M_{0,n})_{\geq 0}$ of $(\M_{0,n})_{>0}$ in $\bM_{0,n}(\R)$ is a stratified space that is homeomorphic to the face stratification of the associahedron polytope.

Let $W$ denote the union of those boundary divisors in $\bM_{0,n}$ whose intersection with $(\M_{0,n})_{\geq 0}$ is empty, and let $\tM_{0,n}:= \bM_{0,n} \setminus W$. The~divisors that do intersect $(\M_{0,n})_{\geq 0}$ correspond to ways to divide $\{1,2,\dots,n\}$ into two cyclic intervals, each of size greater than or equal to two. For example, $\bM_{0,5}$ has ten boundary divisors, and $\tM_{0,5}$ includes five of them, corresponding to the five sides of the pentagon (the associahedron of dimension two). Somewhat surprisingly, the partial compactification $\tM_{0,n}$ is an affine algebraic variety, and its ring of regular functions has the following description. Let~$u_{ij}$ be variables labeled by the diagonals $(i,j)$ (not including sides) of a $n$-gon $P_n$. Then $\C\big[\tM_{0,n}\big]$ is isomorphic to the polynomial ring $\C[u_{ij}]$ modulo the relations
\begin{gather}\label{eq:Rij}
R_{ij} := u_{ij} + \Bigg(\prod_{(k,\ell) \text{ crossing } (i,j)} u_{k\ell} \Bigg) - 1 , \qquad(i,j) \ \text{varying over all diagonals},
\end{gather}
and $\M_{0,n} \subset \tM_{0,n}$ is the locus where $u_{ij} \neq 0$. The~$u_{ij}$ are called \emph{dihedral coordinates}. Brown~\cite{Brown} describes the same space using a presentation with more relations (see Section~10.1); the extra relations are implied by our smaller set. The~$u_{ij}$ are cross-ratios (see \eqref{eq:uij}) on $\M_{0,n}$ and appeared in the study of scattering amplitudes in string theory and for the bi-adjoint $\phi^3$-theory~\cite{ABHY}.

\bigskip\noindent{\large\bf 1.2.} In this paper, we construct in an analogous manner two affine algebraic varieties \mbox{$\M_D \!\subset\! \tM_{D}$} for each Dynkin diagram $D$ of finite type by considering the relations
\begin{gather}\label{eq:Rgamma}
R_\gamma := u_\gamma + \prod_{\omega} u_\omega^{(\omega || \gamma)} - 1.
\end{gather}
Here, $\gamma$ and $\omega$ denote mutable cluster variables of a cluster algebra $\A$ of type $D$ \cite{CA2}, and $(\omega || \gamma)$ denotes the compatibility degree. We~call $\M_D$ the \emph{cluster configuration space of type $D$}. In~the case $D = A_{n-3}$, we have $\M_{A_{n-3}} = \M_{0,n}$ and $\tM_{A_{n-3}} = \tM_{0,n}$. Amongst many remarkable properties of these relations, let us immediately note that $u_\gamma = 0$ forces $u_\omega = 1$ for all $\omega$ such that $(\gamma||\omega) \neq 0$ (or equivalently, $(\omega||\gamma)\neq0$). Thus, factorization is manifest in \eqref{eq:Rgamma}. Some of the results of this work were reported in~\cite{AHLT}, and $\tM_D$ is an example of the notion of ``binary geometry" discussed therein.

Whereas $\tM_{0,n}$ has a stratification indexed by the faces of the associahedron, the space $\tM_D$ has a stratification (Proposition~\ref{prop:MDstrat}) indexed by the faces of the Chapoton--Fomin--Zelevinsky generalized associahedron for $D^\vee$ \cite{CFZ,FZY}. We~show (Theorem~\ref{thm:geom}) that $\M_D$ and $\tM_D$ are smooth affine algebraic varieties and that the boundary stratification of $\tM_D$ is simple normal-crossing. These geometric properties depend on integrality properties of the normal fan $\N(D^\vee)$ of the generalized associahedron, and an isomorphism (Theorem~\ref{thm:toric}) between $\tM_D$ and an affine open subset of the projective toric variety $X_{\N(D^\vee)}$ associated to $\N(D^\vee)$. Like $(\M_{0,n})_{\geq 0}$, the variety~$\tM_D$ contains a distinguished nonnegative part $\M_{D,\geq 0}$, which is a stratified space homeomorphic to the face stratification of the generalized associahedron (Theorem~\ref{thm:pospart}). The~positive part $\M_{D, > 0} \subset \M_D(\R)$ is a distinguished connected component in $\M_D(\R)$, and is cut out by the conditions $u_\gamma >0$. Though $\tM_D$ is not compact, $\big(\tM_D,(\M_{D, \geq 0})_{\geq 0}\big)$ satisfies the other properties of a positive geometry in the sense of \cite{ABL}.

\bigskip\noindent{\large\bf 1.3.} The configuration space $\M_{0,n}$ is isomorphic to the quotient of an open subset $\oGr(2,n) \subset \Gr(2,n)$ of the Grassmannian of 2-planes by the diagonal torus $T \subset \SL_n$ acting on $\Gr(2,n)$. Let~$\tB$ be a full rank acyclic extended exchange matrix of type $D$. Let~$\A\big(\tB\big)$ be the corresponding cluster algebra, $X\big(\tB\big) = \Spec\big(\A\big(\tB\big)\big)$ be the cluster variety, and let $\Xo\big(\tB\big) \subset X\big(\tB\big)$ denote the locus where all cluster variables are non-vanishing. We~show (Theorem~\ref{thm:fullrank}) that $\M_D$ is isomorphic to the (free) quotient of $\Xo\big(\tB\big)$ by the cluster automorphism group $T\big(\tB\big)$, generalizing the construction of $\M_{0,n}$ from $\Gr(2,n)$. The~functions $u_\gamma$ are particular $T\big(\tB\big)$-invariant rational functions on~$X\big(\tB\big)$. The~$u_\gamma$ are related to some of the ``cluster $X$-coordinates" in the sense of Fock and Goncharov~\cite{FG} by the equation $u = X/(1+X)$. The~cluster $X$-coordinates appearing here are exactly those encountered in the Auslander--Reiten walk through cluster variables, beginning from an acyclic quiver and mutating only on sources. It~is important to note that while the $u_\gamma$ are simply related to the cluster $X$-variables in this way, they are actually in bijection with the cluster $A$-variables, while in general there are more cluster $X$-variables than cluster $A$-variables.

We do not have a good understanding of the relationship between $\M_D$ and cluster ${\mathcal{X}}$-varieties; for example, $\M_D$ does not contain a collection of (cluster) torus charts.

Our approach depends crucially on the flexibility in the choice of $\tB$. When $\tB = \tB^\univ$ is the extended exchange matrix for the universal coefficient cluster algebra \cite{CA4,Reading}, the relation~\eqref{eq:Rgamma} is obtained from the primitive exchange relations of $\A\big(\tB^\univ\big)$ by setting all mutable cluster variables to 1, and sending the universal frozen variables $z_\gamma$ to $u_\gamma$. The~non-primitive exchange relations give rise to other relations of the form $U + U' = 1$, where $U$ and $U'$ are monomials in the $u_\gamma$-s.

When $\tB = \tB^\prin$ is the extended exchange matrix for the principal coefficient cluster algebra, the functions $u_\gamma$ become identified with certain ratios of the $F$-polynomials $F_\gamma(\y)$. Bazier-Matte, Douville, Mousavand, Thomas, and Yildrim have shown \cite{BDMTY} in the case that $D$ is simply-laced that the Newton polytope of $F_\gamma(\y)$ has normal fan a coarsening of the $\g$-vector fan $\N\big(D^\vee\big)$ of~$D^\vee$, and this result was extended to skew-symmetric cluster algebras by Fei~\cite{Fei2}. We~extend via folding this description to the case that $D$ is multiply-laced finite type Dynkin diagram. The~identification of $\tM_D$ with an open subset of the toric variety $X_{\N(D^\vee)}$ depends crucially on this analysis.

As an application of our results on quotients of cluster varieties and on $F$-polynomials, we~identify (Theorem~\ref{thm:tropMD}) the positive tropicalization $\Trop_{>0} \M_D$ of the cluster configuration space with the cluster fan $\N\big(D^\vee\big)$. In~particular, we resolve a conjecture of Speyer and Williams \cite[Conjecture~8.1]{SW} on positive tropicalizations of cluster varieties of finite type; see also~\cite{JLS}.

\bigskip\noindent{\large\bf 1.4.}
Inspired by similar questions for $\M_{0,n}$, we proceed with studying the topology of $\M_{D}(\C)$ and $\M_{D}(\R)$. We~identify $\M_{B_n}$ with the complement to the Shi-hyperplane arrangement and thereby compute point counts over finite fields, and the Euler characteristics of $\M_{B_n}(\R)$ and~$\M_{B_n}(\C)$. We~give a configuration space style description of $\M_{C_n}$ (Proposition~\ref{prop:Cnconfig}) but were not able to determine whether $\M_{C_n}$ is a hyperplane arrangement complement. Nevertheless, we~were able to compute the point count for $\M_{C_n}(\F_q)$, and the number of connected components of $\M_{C_n}(\R)$. We~found numerically the point counts for types $D_4$, $D_5$ and $G_2$, and obtained numerically that the point count of $\M_{D_4}(\F_q)$ over a finite field $\F_q$ is not a polynomial in $q$ but a quasi-polynomial.

\bigskip\noindent{\large\bf 1.5.} One of the main motivations for us are scattering amplitudes in string theory. In~\cite{AHL}, we introduced integral functions, called {\it stringy canonical forms},
\begin{gather}\label{eq:stringy}
\I = \int_{\R_{>0}^n} \prod_i \frac{{\rm d}x_i}{x_i} x_i^{\alpha' X_i} \prod_j p_j(\x)^{-\alpha' c_j},
\end{gather}
{\sloppy
where $p_j(\x)$ is a positive Laurent polynomial. We~showed in~\cite{AHL} that the leading order
 $\lim_{\alpha' \to 0} (\alpha')^n \I$ is a rational function that for fixed $c_j$-s coincides with the canonical rational function \cite{ABL} of the Minkowski sum of the Newton polytopes of $p_j(\x)$. Tree-level $n$-point open superstring amplitudes are integrals on $\M_{0,n}$. It~turns out that for a suitable parametrization of $\M_{0,n}$, these amplitudes can be written as an integral $\I_{A_{n-3}}$ in the form \eqref{eq:stringy}, where the $p_j(\x)$ are the $F$-polynomials for the type $A_{n-3}$ cluster algebra. The~importance of the $u_{ij}$-variables appears in the rewriting (see \cite[Section~9]{AHL} or \cite[Section~3]{BD})
\begin{gather*}
\I_{A_{n-3}} = \int_{(\M_{0,n})_{>0}} \Omega((\M_{0,n})_{>0}) \prod_{(i,j)} u_{ij}^{\alpha' X_{ij}}
\end{gather*}}\noindent
of the open-string amplitude. The~poles of $\I_{A_{n-3}}$ are given by $X_{ij} = 0$, and at this pole, the factorization of $\I_{A_{n-3}}$ mimics the factorization of the equations \eqref{eq:Rij}. We~define the \emph{cluster string amplitude} (where $x_\gamma$, $x_i$ are cluster variables of $\A\big(\tB\big)$ of type $D$)
\begin{gather*}
\I_D:= \int_{\M_{D,>0}} \Omega(\M_{D,>0}) \prod_{\gamma \in \Pi} x_\gamma^{\alpha' s_\gamma} \prod_{i=n+1}^{n+m} x_i^{\alpha' s_i}.
\end{gather*}
The poles of $\I_D$ are made manifest by rewriting in terms of the $u_\gamma$-s, and the leading order of~$\I_D$ is controlled by the combinatorics of the generalized associahedron of $D^\vee$.

\section{Background on cluster algebras and generalized associahedra}
In this section we review basic facts concerning cluster algebras. The~most important cluster algebra references for us are \cite{BDMTY,YZ}. For cluster varieties, our conventions follow \cite{LS}.

\bigskip\noindent{\large\bf 2.1.}
Let $D$ be a finite Dynkin diagram with vertex set $I$, and let $A=(a_{ij})$ denote the $n \times n$ Cartan matrix of $D$, where $n = |I|$. Let~$B$ be a skew-symmetrizable exchange matrix, i.e., there exists a matrix $Z$ with positive diagonal entries such that $ZB$ is skew-symmetric. We~say that $B = (B_{ij})$ has type $D$ if
\begin{gather*}
a_{ij} = \begin{cases} 2 & \mbox{if}\quad i = j,
\\
-|B_{ij}| & \mbox{if}\quad i \neq j.
\end{cases}
\end{gather*}
In standard cluster algebra language, $B$ corresponds to an \emph{acyclic} initial seed of a cluster algebra of finite type $D$.
Given $D$, the possible exchange matrices $B$ of type $D$ are in bijection with orientations of the underlying tree of $D$: writing $i \to j$ for the directed edges of this orientation,
we have
\begin{gather*}
B_{ij} = \begin{cases} -a_{ij} & \mbox{if}\quad i \to j,
\\
a_{ij} & \mbox{if}\quad j \to i,
\\
0 & \mbox{otherwise.}
\end{cases}
\end{gather*}
For an $(n+m) \times n$ extended exchange matrix $\tilde B$ extending $B$, we let $\A\big(\tB\big)$ denote the correspon\-ding cluster algebra of geometric type \cite{CA2}. By convention, $\A\big(\tB\big)$ is the $\C$-algebra generated by all mutable cluster variables, all frozen variables, and the inverses of all frozen variables. We~let $X\big(\tB\big) = \Spec\A\big(\tB\big)$ denote the \emph{cluster variety} \cite{LS}. This is a complex affine algebraic variety, and in general it differs from the union of cluster tori, which is sometimes called a cluster manifold.

\bigskip\noindent{\large\bf 2.2.}
We say that $\tB$ (or $\A$ or $X$) has \emph{full rank} if $\tB$ has rank $n$. We~say that $\tB$ (or $\A$ or $X$) has \emph{really full rank} if the rows of $\tB$ span $\Z^n$. If~$\tB$ has full rank, then $X\big(\tB\big)$ is a smooth affine algebraic variety \cite[Theorem~7.7]{Mul}.

\bigskip\noindent{\large\bf 2.3.}
Let $\Pi = \Pi(B)$ be the indexing set for cluster variables, which depends only on $B$. Set $r:=|\Pi|$. For $\gamma \in \Pi$, we let $x_\gamma \in \A\big(\tB\big)$ denote the corresponding cluster variable. (Abusing terminology, sometimes we will refer to elements of $\Pi$ as cluster variables.) We give $\Pi$ the structure of a simplicial complex, called the \emph{cluster complex}, by declaring the maximal faces to be the clusters $\{\gamma_1,\dots,\gamma_n\}$.

The set $\Pi$ can be identified with the following set of pairs of integers:
\begin{gather}\label{eq:Pidef}
\Pi = \bigsqcup_{i \in I} \{(s,i) \,|\, 0 \leq s \leq r_i\},
\end{gather}
where $r_i$, $i \in I$ are some positive integers. The~\emph{initial cluster} is $\{(0,i) \,|\, i \in I\}$. We~let $\Pi^+ \subset \Pi$ denote the subset of non-initial cluster variables, i.e., those $\gamma = (t,j)$ with $t \neq 0$.

\begin{Remark}
In \cite{YZ}, the set $\Pi$ is identified with the set of weights $\{c^s \omega_i \,|\, 0 \leq s \leq r_i\}$. We~have chosen to index using the pairs of integers $(s,i)$ instead. The~choice $c$ of a Coxeter element in~\cite{YZ} corresponds to our choice of an orientation of $D$ in determining the exchange matrix $B$.
\end{Remark}

\begin{Remark}
Starting from the initial cluster $\{x_{(0,i)} \,|\, i \in I\}$, the cluster $\{x_{(1,i)} \,|\, i \in I\}$ is obtained by mutating each vertex of $I$ once, always mutating at sources. This process is repeated to obtain all the cluster variables. In~particular, the cluster variable $x_{(t,j)}$ is obtained by mutation from $x_{(t-1,j)}$; see Proposition~\ref{prop:YZ} for the exchange relation. We~refer the reader to~\cite{BDMTY} for an explanation of this \emph{Auslander--Reiten walk}, the relation to quiver representations, and many examples.
\end{Remark}

\bigskip\noindent{\large\bf 2.4.}
There is an involution ${}^*\colon I \to I$ sending $i$ to $i^*$ induced by the longest element of the Weyl group of the root system of $D$. This involution is the identity in all types except for~$A_n$,~$D_{2n+1}$,~$E_6$, and in these types ${}^*\colon I \to I$ is the non-trivial automorphism of $D$ (as a graph). We~shall use the notation $(-1,i):= (r_{i^*}, i^*)$; see~\cite[Proposition~1.3]{YZ}.

\bigskip\noindent{\large\bf 2.5.} 
Each Dynkin diagram $D$ has a dual denoted $D^\vee$ defined by requiring that the Cartan matrix of $D^\vee$ be transpose to that of $D$. Note that $D$ and $D^\vee$ have the same underlying tree. If~$B$ is an exchange matrix of type $D$, then we let $B^\vee$ be the exchange matrix of type $D^\vee$ associated to the same orientation of the underlying tree of $D$ and $D^\vee$.
For dual exchange matrices $B$ and $B^\vee$, the cluster variables $\Pi(B)$ and $\Pi\big(B^\vee\big)$ are naturally in bijection and under this bijection the cluster complexes are isomorphic.

\bigskip\noindent{\large\bf 2.6.} 
For $\gamma,\omega \in \Pi$, we let $(\omega||\gamma)$ denote the \emph{compatibility degree}, defined for example in~\cite[Proposition~5.1]{YZ}. In~\cite{YZ}, the dependence of the compatibility degree on the choice of $c$ (equivalent to our choice of~$B$) is made explicit, but we have suppressed this dependence in our notation. By \cite[Section~5]{YZ}, the compatibility degrees for different choices of~$B$ are equivalent under an~appropriate renaming of $\Pi$. Examples of the compatibility degree are given in Section~3.2.

We have $(\omega||\gamma) = 0$ if and only if $(\gamma||\omega) = 0$ and in this case we say that $\omega$ and $\gamma$ are compatible. Otherwise, we call $\omega$ and $\gamma$ incompatible. If~$(\omega||\gamma) = (\gamma||\omega) = 1$, we say that $\omega$ and~$\gamma$ are {\it exchangeable}. The~faces of the cluster complex consist of sets of cluster variables that are pairwise compatible.

\bigskip\noindent{\large\bf 2.7.}
Let $\A^{\prin} = \A\big(\tB^\prin\big)$ denote the cluster algebra with principal coefficients \cite{CA4}. Thus $\tB^{\prin}$ is a $2n \times n$ matrix whose top half is equal to $B$ and bottom half is equal to the identity matrix. In~this case, the initial mutable variables are denoted $x_1,x_2,\dots,x_n$ and the principal frozen variables are denoted $y_1,y_2,\dots,y_n$. We~have a $\Z^n$-grading on
the principal coefficient cluster algebra $\A^\prin$ given by
\begin{gather}\label{eq:prindeg}
\deg(x_i) = e_i \qquad \text{and} \qquad \deg(y_i) = -B e_i.
\end{gather}
Each mutable cluster variable is homogeneous with respect to this grading, and we define the \emph{$\g$-vector} by $\g_\gamma := \deg(x_\gamma)$ for $\gamma \in \Pi$.

\bigskip\noindent{\large\bf 2.8.}
For $\gamma \in \Pi^+$, define the $F$-polynomial $F_\gamma(\y)$ by setting the initial cluster variables to 1 in the Laurent expansion of the cluster variable $x^\prin_\gamma$ in the cluster algebra $\A^\prin$ with principal coefficients:
\begin{gather*}
F_\gamma(\y) := x^\prin_\gamma(x_i = 1, y_1,y_2,\dots,y_n).
\end{gather*}
By convention, we have $F_\gamma(\y) = 1$ if $\gamma$ is initial. Computations of $\g$-vectors and $F$-polynomials are given in Examples~\ref{ex:A3} and~\ref{ex:B3}. Further examples can be found in~\cite{BDMTY,CA4}.

\bigskip\noindent{\large\bf 2.9.}
The {\it cluster fan} $\N(B)$ is the collection of cones spanned by $\{\g_{\gamma_1},\dots,\g_{\gamma_s}\}$ as $\{\gamma_1,\dots,\gamma_s\}$ varies over collections of cluster variables that belong to the same cluster, called {\it compatible} cluster variables. Recall that a cone $C$ is called {\it simplicial} if $\dim(C)$ is equal to the number of extremal rays of $C$, and a fan is called simplicial if all its cones are. A fan $\N$ in $\R^n$ is called {\it smooth} if it is simplicial and for each maximal cone $C \in \N$ the primitive integer vectors $\v_1,\dots,\v_n$ spanning $C$ form an integral basis for $\Z^n$.

\begin{Theorem}[{\cite{CFZ,FZY,HLT}}]\label{thm:clusterfan}
The collection of cones $\N(B)$ is a smooth, complete polyhedral fan.
\end{Theorem}

A \emph{generalized associahedron of type $B$} is any polytope whose normal fan is equal to $\N(B)$.
Often, we will say ``generalized associahedron of type $D$", with the choice of $B$ of type $D$ understood.

\bigskip\noindent{\large\bf 2.10.}
For $\gamma \in \Pi$, let $P_\gamma$ denote the Newton polytope of the $F$-polynomial $F_\gamma(\y)$. By convention, if $\gamma$ is initial, we have set $F_\gamma(\y) = 1$ and $P_\gamma = \{0\}$. Let~$F(\y) = \prod_{\gamma \in \Pi} F_\gamma(\y)$. Then the Newton polytope $P$ of $F(\y)$ is the Minkowski sum $\sum_{\gamma \in \Pi} P_\gamma$. The~following result is established in~\cite{BDMTY} when $D$ is simply-laced (and extended to not necessarily acyclic initial seeds in~\cite{Fei2}), and in~Theorem~\ref{thm:Newt} we extend the result to multiply-laced finite type $D$ with acyclic initial seed.

\begin{Theorem}\label{thm:FNewton}
The $($outer$)$ normal fan of the Minkowski sum $\sum_{\gamma \in \Pi} P_\gamma$ is equal to $\N(B^\vee)$.
\end{Theorem}

\medskip\noindent{\large\bf 2.11.} 
Let $\tau\colon \Pi \to \Pi$ be the bijection defined by $\tau(t,j) = (t-1,j)$ for $0 \leq t \leq r_j$, denoted $\tau_c$ in~\cite{YZ}. Then $(\tau \gamma|| \tau \omega) = (\gamma|| \omega)$ and $\tau$ induces an automorphism of the cluster complex of $D$.

An exchange relation for $\A\big(\tB\big)$ is called \emph{primitive} if it is of the form $x_\gamma x_{\omega} = M + M'$, where one of the two monomials $M$, $M'$ does not contain any mutable cluster variables. The~primitive exchange relations are exactly the ones of the form $x_{\tau \gamma} x_{\gamma} = M + M'$.

\bigskip\noindent{\large\bf 2.12.}
Let $\A^{\univ} = \A\big(\tB^\univ\big)$ denote the cluster algebra with universal coefficients, from \cite[Theorem~12.4]{CA4}, \cite[Section~5]{YZ}, and \cite[Theorem~10.12 and Remark 10.13]{Reading}. Thus $\tB^{\univ}$ is a~$(n+r) \times n$ matrix whose top part is equal to $B$ and whose bottom part has rows given by the $\g$-vectors of the cluster algebra with exchange matrix $B^T$, see \cite{Reading}. The~bottom $r$ rows of $\tB^\univ$ are again indexed by $\Pi$, and we denote the corresponding frozen variables by $z_\gamma$, for $\gamma \in \Pi$.

\begin{Proposition}[{\cite[Proposition~5.6]{YZ}}] \label{prop:YZ}
The primitive exchange relations of $\A\big(\tB^\univ\big)$ are given~by
\begin{gather}\label{eq:exchange}
x_{(t-1,j)} x_{(t,j)} = z_{(t,j)} \prod_{i \to j} x_{(t,i)}^{-a_{ij}} \prod_{j \to i} x_{(t-1,i)}^{-a_{ij}} + \prod_\omega z_\omega^{(\omega|| (t,j))}
\end{gather}
for $j \in I$ and $0 \leq t \leq r_j$.
\end{Proposition}

\bigskip\noindent{\large\bf 2.13.}
For $\gamma \in \Pi^+$, define the universal $F$-polynomial $F^\univ_\gamma(\z)$ by setting the initial cluster variables to 1 in the Laurent expansion of $x^\univ_\gamma$:
\begin{gather*}
F^\univ_\gamma(\z) := x^\univ_\gamma(x_i = 1, z_\omega).
\end{gather*}
By convention, we have $F^\univ_\gamma(\z) = 1$ if $\gamma$ is initial.

\bigskip\noindent{\large\bf 2.14.} 
Suppose $D$ is a multiply-laced Dynkin diagram whose underlying tree is oriented. Then there exists a simply-laced Dynkin diagram $\tD$ such that $D$ is obtained from $\tD$ by \emph{folding}~\cite{Dup}, and the orientation of $D$ is induced by the orientation of $\tD$. In~this situation, there is a finite group~$\Gamma$ acting on $\tilde I$ and $\tPi$ such that $I$ and $\Pi$ are identified with the $\Gamma$-orbits on $\tilde I$ and $\tPi$. We~obtain surjective quotient maps $\nu\colon \tI \to I$ and $\nu\colon \tPi \to \Pi$.
Abusing notation, let~$\nu\colon \R^{|\tPi|} \to \R^{|\Pi|}$ be given by $\nu(e_{\tgamma}) = e_{\nu(\tgamma)}$, and $\nu\colon \R^{|\tilde I|} \to \R^{|I|}$ be given by~$\nu(e_{\tilde i}) = e_{\nu(\tilde i)}$. Similarly, define $\nu\colon \Z\big[y_{\tilde i} \,|\, \tilde i \in \tilde I\big] \to \Z[y_i \,|\, i \in I]$ by $\nu(y_{\tilde i}) = y_{\nu(\tilde i)}$ and $\nu\colon \Z\big[z_{\tgamma} \,|\, \tgamma \in \tPi\big] \to \Z[z_\gamma \,|\, \gamma \in \Pi]$ by~$\nu(z_{\tgamma})= z_{\nu(\tgamma)}$. The~following results are a consequence of the definitions.

\begin{Proposition} \label{prop:folding}\
\begin{enumerate}\itemsep=0pt
\item[$1.$] For $\tgamma,\tilde \omega \in \tPi$ and $g\in \Gamma$, we have $(g \cdot \tgamma||g \cdot \tilde \omega) = (\tgamma||\tilde \omega)$.
\item[$2.$] For $\gamma,\omega \in \Pi$, we have $(\gamma || \omega)_D = \sum_{\tgamma \in \nu^{-1}(\gamma) } (\tgamma || \tilde \omega)_{\tD}$ for any $\tilde \omega \in \bar \omega$.
\item[$3.$] For any $\tgamma \in \tPi$, we have $\nu(F_{\tgamma}(\tilde \y)) = F_{\nu(\tgamma)}(\y)$ and $\nu\big(F^\univ_{\tgamma}(\tilde \z)\big) = F^\univ_{\nu(\tgamma)}(\z)$.
\item[$4.$] For any $\tgamma \in \tPi$, we have $\nu(\g_{\tgamma}) = \g_{\nu(\tgamma)}$.
\end{enumerate}
\end{Proposition}
Examples of foldings are given in Section~3.2.

\section[The cluster configuration space MD]{The cluster configuration space $\boldsymbol{\M_D}$}
In this section, we define the cluster configuration space $\M_D$ and its partial compactification~$\tM_D$, and we state some geometric properties of these spaces. We~also give examples of the cluster compatibility degree appearing in the defining relations.

\bigskip\noindent{\large\bf 3.1.}
Let $\C[u]:=\C[u_\gamma \,|\, \gamma \in\Pi]$ be the polynomial ring with generators $u_\gamma$ and let $\C[u^{\pm 1}]$ denote the Laurent polynomial ring with the same generators.
\begin{Definition}\label{def:ID}
Let $I_D$ denote the ideal \big(in $\C[u]$ or $\C\big[u^{\pm 1}\big]$\big) generated by the elements
\begin{gather}\label{eq:I}
R_\gamma := u_\gamma + \prod_{\omega} u_{\omega}^{(\omega||\gamma)} - 1
\end{gather}
for $\gamma \in \Pi$, and $(\omega||\gamma)$ denotes the compatibility degree.
\end{Definition}

\begin{Definition}
Define the \emph{cluster configuration space} $\M_D$ and its partial compactification $\tM_D$ by
\begin{gather*}
\M_D := \Spec\big(\C\big[u^{\pm 1}\big]/I_D\big) \qquad \text{and} \qquad
\tM_D:=\Spec\big(\C[u]/I_D\big).
\end{gather*}
\end{Definition}

The following result will be proved in Section~5.4.

\begin{Theorem}\label{thm:geom}
The two schemes $\M_D$ and $\tM_D$ are smooth, irreducible, affine algebraic varieties of dimension $n$. The~boundary divisor $\partial:= \tM_D \setminus \M_D$ is a simple normal-crossing divisor in~$\tM_D$.
\end{Theorem}

If $D = \varnothing$, we define $\M_D = \tM_D = \Spec(\C)$ to be a point. If~$D = A_1$, then $\C[u]/I_D = \C[u,u']/(u+u'=1)$ so $\tM_{A_1} = \C$ and $\M_{A_1} = \C \setminus \{0,1\}$.

\begin{Remark}
The definition of $\M_D$ and $\tM_D$ depends on the choice of exchange matrix $B$ of type $D$ only in the indexing of the generators $u_\gamma$ by $\Pi$. For two different orientations of $D$, there is a natural bijection between the two indexing sets $\Pi(B)$ that arise, and a natural isomorphism between the resulting schemes $\M_D(B)$.
\end{Remark}

\medskip\noindent{\large\bf 3.2.} 
Let us give the relations $R_\gamma$ explicitly in types $A$, $B$, $C$, $D$, $G$. In~the following discussion, we use models for $\Pi$ involving diagonals of a polygon; see \cite{CA2} for further details. The~precise correspondence with \eqref{eq:Pidef} depends on the choice of initial cluster (for example, a choice of triangulation of the polygon in type $A$), or equivalently the choice of $B$, or equivalently the choice of orientation of $D$.

\medskip\noindent{\bf 3.2.1. Type $\boldsymbol{A_{n-3}}$.} In this case, the set $\Pi$ can be identified with the diagonals (not including sides!) $\{(i,j)\}$ of an $n$-gon. The~$R_\gamma$ are the equations \eqref{eq:Rij}. The~compatibility degree is given by the formula $((i,j) || (k,\ell)) = 1$ if $(i,j)$ and $(k,\ell)$ cross (in the interior of the polygon) and $((i,j) || (k,\ell)) = 0$ if $(i,j)$ and $(k,\ell)$ do not cross. The~automorphism $\tau$ in this case corresponds to the order $n$ rotation of the polygon. The~clusters are exactly the maximal sets of pairwise compatible diagonals. In~other words, clusters are in bijection with triangulations of the $n$-gon.

\medskip\noindent{\bf 3.2.2. Type $\boldsymbol{C_{n-1}}$.} Let $P_{2n}$ be the $2n$-gon with vertices cyclically labeled $1,2,\dots,n,\bar 1,\bar 2,\dots, \bar n$. The~set $\Pi$ is identified with the union
\begin{align}
\Pi &= \Big\{\big[i,\bar i\big]:= \big(i,\bar i\big) \,|\, 1 \leq i \leq n\Big\}
\cup \Big\{[i, j] := \big((i,j),\big(\bar i, \bar j\big)\big) \,|\, 1 \leq i <j-1 < n\Big\}\nonumber
\\
&{}\phantom{=\ }\cup \Big\{\big[i,\bar j\big] := \big(\big(i,\bar j\big),\big(j, \bar i\big)\big) \,|\, 1 \leq i < j \leq n, \, (i,j) \neq (1,n)\Big\}
\label{eq:PiC}
\end{align}
of the long diagonals, and pairs of centrally symmetric diagonals in $P_{2n}$. In~total we have $|\Pi| = n^2-n$.

The compatibility degree $(\gamma||\omega)$ is equal to the number of crossings of one of the diagonals representing $\omega$ with the diagonals representing $\gamma$. Thus for example $\big([1, \bar 1]||[2, \bar 3]\big) =1$ but $\big([2,\bar 3] || [1,\bar 1]\big) = 2$. For example, for $C_2$, Definition~\ref{def:ID} gives the two equations
\begin{gather}
1 = u_{[1 \bar 1]} + u_{[2 \bar 2]} u_{[3 \bar 3]} u_{[2 \bar 3]}^2,\nonumber
\\
1 = u_{[2 \bar 3]} + u_{[1 \bar 1]} u_{[13]} u_{[1 \bar 2]}\label{eq:C2}
\end{gather}
and the three cyclic rotations of each.

This case is obtained from $A_{2n-3}$ by folding. There is an action of the two-element group $\Gamma$ on $P_{2n}$ mapping $i \leftrightarrow \bar i$. This induces the natural map $\nu\colon \tilde \Pi \to \Pi$ sending diagonals of $P_{2n}$ to $\Gamma$-orbits on the diagonals of $P_{2n}$. We~may verify Proposition~\ref{prop:folding}(2): for example $\big([2,\bar 3]||[1, \bar 1]\big)_{C_2} = \big((2,3)||(1,\bar 1)\big)_{A_3} + \big((\bar 2, \bar 3)||(1,\bar 1)\big)_{A_3} = 1+ 1 = 2$. The~automorphism $\tau$ is inherited from the rotation of the $2n$-gon $P_{2n}$.

\medskip\noindent{\bf 3.2.3. Type $\boldsymbol{D_n}$.} 
Let $PP_n$ denote an $n$-gon $P_n$ with vertices $1,2,\dots,n$ (in clockwise order) and an additional marked point $0$ in the middle. The~set $\Pi$ consists of certain arcs in $PP_n$ connecting vertices and $0$: ($a$) for $1 \leq i \neq j \leq n$ and $i \neq j+1 \mod n$ we have an arc $(i,j)$ connecting $i$ to $j$ going counterclockwise around $0$, and $(b)$ for each $1\leq i \leq n$ we have two arcs~$[i]$ and $[\tilde i]$ connecting $i$ to $0$. We~denote the corresponding $u$-variables by $u_{ij}$, and $u_i$ and~$u_{\tilde i}$. See Figure~\ref{d4}. (We caution the reader that the notation $\tilde i$ here is unrelated to the notation $\tgamma$ used for foldings.) In this case, the automorphism $\tau$ is the composition of the rotation of $P_n$ with ``changing the tagging at 0" (i.e., switching from $[i]$ to $[\tilde i]$ if the arc is incident to $0$).

The compatibility degree $(\gamma || \omega)$ is equal to the (minimal) number of intersection points bet\-ween the arc $\gamma$ and the arc $\omega$, if at least one of $\gamma$ and $\omega$ do not connect to $0$. If~both $\gamma$ and~$\omega$ connect to $0$, then we have $([i] || [j]) = \big([\tilde i] || [\tilde j]\big) = \big([i] || [\tilde i]\big) = 0$ but $\big([i] || [\tilde j]\big) = \big([\tilde j] || [i]\big) = 1$ if $i \neq j$. For $D_4$, representative equations from Definition~\ref{def:ID} are
\begin{gather}
1 = u_{12} + u_3 u_{\tilde3} u_4 u_{\tilde 4} u_{34}^2 u_{23} u_{24} u_{41} u_{31},\nonumber
\\
1 =u_{13} + u_4u_{\tilde 4} u_{41}u_{42} u_{24} u_{34},\nonumber
\\
1 = u_1+ u_{\tilde 2} u_{\tilde 3} u_{\tilde 4} u_{23} u_{34} u_{24}
\label{eq:D4}
\end{gather}
and we have $4$, $4$, $8$ equations of these types respectively, for a total of $16$ equations.

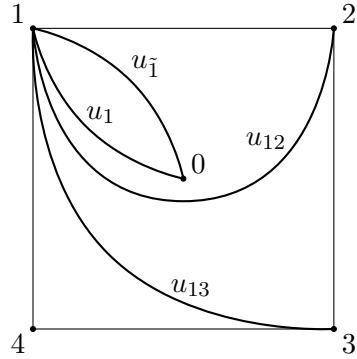
\begin{figure}
\centering
\begin{tikzpicture}
\draw[fill=black] (0,0) circle (.2ex);
\draw[fill=black] (2,2) circle (.2ex);
\draw[fill=black] (2,-2) circle (.2ex);
\draw[fill=black] (-2,-2) circle (.2ex);
\draw[fill=black] (-2,2) circle (.2ex);
\node at (0.2,0.2) {$0$};
\node at (-2.2,2.2) {$1$};
\node at (2.2,2.2) {$2$};
\node at (2.2,-2.2) {$3$};
\node at (-2.2,-2.2) {$4$};
\node at (-0.5,1.5) {$u_{\tilde 1}$};
\node at (-1.1,0.85) {$ u_1$};
\node at (1.1,0.5) {$u_{12}$};
\node at (0.1,-1.45) {$u_{13}$};

\draw (2,2)--(2,-2)--(-2,-2)--(-2,2)--(2,2);
\draw[thick] plot [smooth,tension=1.7] coordinates{(-2,2) (0,-0.3)(2,2)};
\draw[thick] plot [smooth,tension=1] coordinates{(-2,2) (-1.3,0.7)(0,0)};
\draw[thick] plot [smooth,tension=1] coordinates{(-2,2) (-0.7,1.3)(0,0)};
\draw[thick] plot [smooth,tension=1] coordinates{(-2,2) (-1,-1)(2,-2)};
\end{tikzpicture}
\caption{The polygon $PP_4$ and some $u$-variables. Note $u_{21}$ does not exist.}\label{d4}
\end{figure}

\medskip\noindent{\bf 3.2.4. Type $\boldsymbol{B_{n-1}}$.} 
The set $\Pi$ is the same as for $C_{n-1}$ \eqref{eq:PiC}. However, the compatibility degree $(\gamma||\omega)$ is equal to the number of crossings of one of the diagonals representing $\gamma$ with the diagonals representing $\omega$. Thus, $\big([1, \bar 1]||[2, \bar 3]\big) =2$ but $\big([2,\bar 3] || [1,\bar 1]\big) = 1$. For $B_2$, Definition~\ref{def:ID} gives the two equations
\begin{gather*}
1 = u_{[1 \bar 1]} + u_{[2 \bar 2]} u_{[3 \bar 3]} u_{[2 \bar 3]},
\\
1 = u_{[2 \bar 3]} + u_{[1 \bar 1]}^2 u_{[13]} u_{[1 \bar 2]}
\end{gather*}
and the three cyclic rotations of each. Note that $\M_{B_2}$ is isomorphic to $\M_{C_2}$ under a non-trivial re-indexing of the $u$-variables. However $\M_{B_n}$ and $\M_{C_n}$ are not isomorphic for $n >2$.

Type $B_{n-1}$ can be obtained from type $D_n$ by folding. Let~$\Gamma$ be the two-element group acting on $\tPi = \Pi(D_n)$ by sending $[i] \leftrightarrow [\tilde i]$ and fixing all other $(i,j)$. The~map $\nu\colon \tPi \to \Pi$ sends $[i]$ and~$[\tilde i]$ to $[i,\bar i]$, and sends $(i,j)$ to $[i, \bar j]$ if $i < j$ and to $[j, i]$ if $ i > j$. For example, for $B_3$, the images of the equations from \eqref{eq:D4} are
\begin{gather}
1 = u_{[1 \bar 2]} + u_{[3\bar 3]}^2 u_{[4 \bar 4]}^2 u_{[3 \bar 4]}^2 u_{[2 \bar 3]} u_{[2 \bar 4]} u_{[14]} u_{[13]},\nonumber
\\
1 =u_{[1 \bar 3]} + u_{[4 \bar 4]}^2 u_{[1 4]}u_{[2 4]} u_{[2 \bar4]} u_{[3 \bar 4]},\nonumber
\\
1 = u_{[1\bar 1]}+ u_{[2\bar 2]} u_{[3 \bar 3]} u_{[4 \bar 4]} u_{[2 \bar 3]} u_{[3 \bar 4]} u_{[2 \bar 4]}.
\label{eq:B3}
\end{gather}
The automorphism $\tau$ is inherited from the rotation of the $n$-gon in type $D_n$.

\medskip\noindent{\bf 3.2.5. Type $\boldsymbol {G_2}$.} 
In type $G_2$, we have $|\Pi|=8$ and we denote $u$-variables by $a_i$, $b_i$ for $i =1,2,3,4$. The~$u$-equations for type $G_2$ are
\begin{gather*}
1 = a_1 + a_2 b_2 a^2_3 b_3 a_4, \\
1 = b_1 + b_2 a_3^3 b_3^2 a^3_4 b_4
\end{gather*}
and the cyclic rotations under the group $\Z/4\Z$. Type $G_2$ can be obtained from $D_4$ by folding. (Though at present we only consider foldings of simply-laced diagrams, $G_2$ can also be obtained from $B_3$ by folding.)

\bigskip\noindent{\large\bf 3.3.}
Let $F$ be a face of the generalized associahedron, which we identify with a pairwise compatible subset $\{\gamma_1,\dots,\gamma_d\}$ of $\Pi$. We~define $\tM_D(F) \subset \tM_D$ to be the closed subscheme cut out by the ideal $(u_{\gamma_1},\dots,u_{\gamma_d})$. If~$F = \varnothing$, then $\tM_D(\varnothing) := \tM_D$. We~define $\M_D(F) \subset \tM_D(F)$ to be the open subscheme of $\tM_D(F)$ where all the variables $\{u_\gamma \,|\, \gamma \notin F\}$ are non-vanishing.

\begin{Proposition}\label{prop:MDstrat}
We have a natural stratification
\begin{gather}\label{eq:Mstrata}
\tM_D = \bigsqcup_{F} \M_D(F),
\end{gather}
where $F$ varies over all the faces of the generalized associahedron.
\end{Proposition}

\begin{proof}
When $\omega$ and $\gamma$ are incompatible, the coordinate $u_\omega$ is non-vanishing on $\M_D(\{\gamma\})$. This shows that the subschemes $\M_D(F)$ cover $\tM_D$.
\end{proof}

\medskip\noindent{\large\bf 3.4.}
Let us analyze $\tM_D(F)$ for $F = \{\gamma\}$. Setting $u_\gamma = 0$ in the equation
\begin{gather*}
u_\omega + \prod_{\tau} u_{\tau}^{(\tau||\omega)} = 1
\end{gather*}
we find that $u_\omega = 1$ for all $\omega$ incompatible with $\gamma$. Let~$\Pi(\gamma) \subset \Pi$ be the subset of $\kappa \in \Pi$ that are compatible with $\gamma$.
Setting $u_\omega=1$ in $R_\kappa$ for $\kappa \in \Pi(\gamma)$, we get
\begin{gather*}
R'_\kappa := u_\kappa + \prod_{\tau\in \Pi(\gamma)} u_{\tau}^{(\tau||\kappa)} - 1.
\end{gather*}
It follows that the coordinate ring of $\tM_D(F)$ has the following presentation
\begin{gather*}
\C[u_\kappa \,|\, \kappa\in \Pi(\gamma)]/(R'_\kappa).
\end{gather*}

\begin{Proposition}\label{prop:divisorsplit}
Suppose $\gamma = (t,j)$ and removing $j$ from $D$ disconnects $D$ into connected components $D_1,\dots,D_s$.
Then we have \begin{gather*}\tM_D(\{\gamma\}) = \tM_{D_1} \times \tM_{D_2} \times \cdots \times \tM_{D_s}\end{gather*} and
\begin{gather*}\M_D(\{\gamma\}) = \M_{D_1} \times \M_{D_2} \times \cdots \times \M_{D_s}.\end{gather*}
\end{Proposition}
\begin{proof}
Let $\gamma \in \Pi$. Since $\gamma$ defines a facet of the generalized associahedron for $B$, it follows that the collection of clusters containing $\gamma$ are connected by mutation, without mutating $\gamma$. It~follows that $\Pi(\gamma)$ is the cluster complex of a cluster algebra of finite type associated to the disjoint union of $D_1,D_2,\dots,D_s$.
\end{proof}
Note that removing a vertex from a finite type Dynkin diagram produces at most three components, so in Proposition~\ref{prop:divisorsplit} we have $s \leq 3$.

By applying Proposition~\ref{prop:divisorsplit} repeatedly, we have the following result.
\begin{Proposition}
Any $\tM_D(F)$ $($resp.\ $\M_D(F))$ is a direct product of $\tM_{D'}$ $($resp.\ $\M_{D'})$ as $D'$ varies over a finite set of Dynkin diagrams obtained by removing some vertices from $D$.
\end{Proposition}

\bigskip\noindent{\large\bf 3.5.}
Recall the bijection $\tau\colon \Pi \to \Pi$ of Section~2.11. We~have automorphisms $\tau\colon \M_D \to \M_D$ and $\tau\colon \tM_D \to \tM_D$ induced by $u_\gamma \mapsto u_{\tau \gamma}$. The~order of the automorphism $\tau$ is either $h+2$ or~$(h+2)/2$, where $h$ is the Coxeter number of $D$; see for example~\cite{ASS}. It~would be interesting to compute: (1) the group of automorphisms of the variety $\tM_D$, (2) the group of automorphisms of the variety $\M_D$, and (3) the group of automorphisms of $\M_D$ that send the positive part~$\M_{D,>0}$ (defined in Section~8.1) to itself.

In the case $D = A_{n-3}$, we have $\M_D = \M_{0,n}$ which has a natural action of $S_n$. The~group~$S_n$ acts transitively on the connected components of $\M_{0,n}(\R)$. The~positive part $\M_{D,>0}$ is one of the connected component of $\M_{0,n}(\R)$ and it is sent to itself by a dihedral subgroup of $S_n$ of order $2$, and the number of connected components of $\M_{0,n}(\R)$ is equal to $n!/2n$, see Section~7.2.

\bigskip\noindent{\large\bf 3.6.}
Let $\gamma \in \Pi$. Then as in Proposition~\ref{prop:divisorsplit}, we can uniquely associate Dynkin diagrams $D_1,\dots,D_{s(\gamma)}$ to $\gamma$, where $s \leq 3$.

\begin{Proposition}\label{prop:forget}
Suppose that $\gamma \in \Pi$ and $s(\gamma) = 1$. Then we have a natural morphism
\begin{gather}\label{eq:forget}
\M_{D} \longrightarrow \M_{D_1}.
\end{gather}
\end{Proposition}
The proof of Proposition~\ref{prop:forget} is delayed to Section~4.7. The~map of Proposition~\ref{prop:forget} corresponds to ``forgetting a marked point'' in the case of $\M_{A_{n-3}} = \M_{0,n}$. We~expect \eqref{eq:forget} to be a fibration, similar to the $\M_{0,n}$ case.

\bigskip\noindent{\large\bf 3.7.}
Let $D$ be a folding of $\tD$ and let $\Gamma$ and $\nu\colon \tPi \to \Pi$ be as in Section~2.14.

\begin{Proposition}
The quotient of $\C[u]/I_{\tD}$ by the ideal generated by the equations $u_{\tgamma} = u_{g \cdot \tgamma}$ for $\tgamma \in \tPi$ and $g \in \Gamma$ is canonically isomorphic to $\C[u]/I_D$.
\end{Proposition}
\begin{proof}
Let $\nu\colon \C\big[u_{\tgamma}\,|\, \tgamma \in \tPi\big] \to \C[u_{\gamma}\,|\, \gamma \in \Pi]$ be the ring homomorphism given by $\nu(u_{\tgamma}) \allowbreak = u_{\nu(\tgamma)}$. Then applying Proposition~\ref{prop:folding}(1),(2), we have $\nu(R_{\tgamma}) = R_{\nu(\tgamma)}$. The~result follows.
\end{proof}
Thus $\tM_{D}$ can be identified with a closed subscheme of $\tM_{\tD}$ and it is straightforward to see that~$\M_{D}$ is the intersection of $\tM_{D} \subset \tM_{\tD}$ with the open subset $\M_{\tD} \subset \tM_{\tD}$.

\bigskip\noindent{\large\bf 3.8.}
The significance of the following conjecture is unclear to us. We~have proved it by a direct, elementary calculation for $D= A_n$, $n \geq 2$.
\begin{Conjecture}
For $D$ not of type $A_1$, the ring $\C[u^{\pm 1}]/I_D$ is generated by $u_\gamma^{-1}$, for $\gamma \in \Pi$.
\end{Conjecture}

\section[MD as a quotient of a cluster variety]
{$\boldsymbol{\M_D}$ as a quotient of a cluster variety}\label{sec:quotient}

In this section we show that $\M_D$ can be obtained as a quotient of an open subspace $\Xo$ of the cluster variety $\X\big(\tB\big)$ by the action of the cluster automorphism torus $T$ considered in~\cite{LS}. An~important role is played by principal and universal coefficients, where $\tB = \tB^\prin$ or \mbox{$\tB=\tB^\univ$}. In~particular, the defining relations of $\M_D$ are obtained from the primitive exchange relations of $\X\big(\tB^\univ\big)$.

\bigskip\noindent{\large\bf 4.1.}
Let $\tB$ be a full rank extended exchange matrix. Let~$T = T\big(\tB\big)$ be the cluster automorphism group \cite{LS} of $\A\big(\tB\big)$: this is the group of algebra automorphisms $\phi\colon \A\big(\tB\big) \to \A\big(\tB\big)$ such that for each (mutable or frozen) cluster variable $x$, we have $\phi(x) = \zeta(x) x$ for $\zeta(x) \in \C^\ast$. Thus, $T$ acts on any cluster torus of the cluster variety $X\big(\tB\big)$ by scaling the coordinates. By \cite[Proposition~5.1]{LS}, we have
\begin{gather}\label{eq:auto}
T = \Hom\big(\Z^{n+m}/\tB \Z^n, \C^\ast\big).
\end{gather}
Since $\tB$ has full rank, the group $T$ is a (possibly disconnected) abelian algebraic group of dimension $m$. The~character group of $T$ is the lattice $\Z^{n+m}/\tB \Z^n$. By definition, the torus $T$ acts on each cluster variable by a character, and we denote the weight of the cluster variable $x \in \A\big(\tB\big)$ by $\wt(x) \in \Z^{n+m}/\tB \Z^n$.

\begin{Lemma}
Let $\tB$ have full rank. Then $\tB$ has really full rank if and only if $\Z^{n+m}/\tB \Z^n$ has no torsion, or equivalently, the group $T$ is connected, and thus a torus of dimension $m$.
\end{Lemma}
\begin{proof}
The rows of $\tB$ span $\Z^n$ if and only if $\tB \Q^n \cap \Z^{n+m} = \tB \Z^n$ if and only if $\Z^{n+m}/\tB \Z^n$ has no torsion.
\end{proof}

Let $X = X\big(\tB\big)$ be the cluster variety, which is a smooth affine algebraic variety. Let~$\Xo \subset X$ be the locus where all mutable cluster variables are non-vanishing. In~terms of rings, we have
\begin{gather*}
\Xo:= \Spec\big(\A[1/x \,|\, x \text{ is a mutable cluster variable}]\big).
\end{gather*}
Thus $\Xo$ is a smooth affine subvariety of the initial (or any) cluster subtorus of $X$, and it follows immediately from the definitions that the action of $T$ preserves $\Xo$, and furthermore the action of $T$ is free on $\Xo$. The~geometric invariant theory quotient
\begin{gather*}
\Xo \sslash T:= \Spec\big(\C\big[\Xo\big]^T\big)
\end{gather*}
is again a smooth affine algebraic variety, and furthermore, there is a bijection between closed points of $\Xo \sslash T$ and $T$-orbits on $\Xo$. We~thus simply denote $\Xo\sslash T$ by $\Xo/T$. 
Explicitly, the ring~$\C\big[\Xo\big]^T$ consists of all weight zero Laurent polynomials in cluster variables.

\bigskip\noindent{\large\bf 4.2.}
Let $\tau\colon \Pi \to \Pi$ be the bijection defined by $\tau(t,j) = (t-1,j)$ for $0 \leq t \leq r_j$. The~primitive exchange relations are of the form
\begin{gather*}
x_{\tau \gamma} x_{\gamma} = M + M',
\end{gather*}
where $M'$ only involves frozen variables. For each primitive exchange relation, we define the rational function
\begin{gather*}
f_\gamma := \frac{M}{x_{\tau \gamma} x_{\gamma}}.
\end{gather*}
By definition, $f_\gamma \in \C\big[\Xo\big]$, and it is easy to see that $f_\gamma$ is $T$-invariant. Thus $f_\gamma \in \C\big[\Xo\big]^T$.

\begin{Theorem}\label{thm:fullrank}
Suppose that $\tB$ is a full rank extended exchange matrix, acyclic and of finite type. Let~$X= X\big(\tB\big)$. Then the quotient $\Xo/T$ is a smooth affine variety isomorphic to $\M_D$, and the isomorphism $\C[\M_D] \cong \C\big[\Xo\big]^T$ is given by $u_\gamma \mapsto f_\gamma$.
\end{Theorem}

\begin{Theorem}\label{thm:principal}
Suppose that $\tB$ is a full rank extended exchange matrix, acyclic and of finite type. Let~$X= X^\prin\big(\tB\big)$ have principal coefficients. Then $\M_D \cong \Xo^\prin/T^\prin$ is isomorphic to the locus $X^\prin(1) \subset X^\prin$, where all initial mutable cluster variables have been set to $1$. The~coordinate ring $\C[\M_D]$ is isomorphic to the subring of $\C(y_1,y_2,\dots,y_n)$ generated by $F_\gamma^{\pm 1}(\y)$ and $y_i^{\pm 1}$.
\end{Theorem}

\begin{Theorem}\label{thm:univ}
Suppose that $\tB$ is a full rank extended exchange matrix, acyclic and of finite type. Let~$X= X^\univ\big(\tB\big)$ have universal coefficients. Then $\M_D \cong \Xo^\univ/T^\univ$ is isomorphic to the locus $X^\univ(1) \subset X^\univ$, where all mutable cluster variables have been set to $1$. The~isomorphism $\C[\M_D] \cong \C[X^\univ(1)]$ is given by $u_\gamma \mapsto z_\gamma$.
\end{Theorem}

\bigskip\noindent{\large\bf 4.3.}
The relations in the following corollary will be discussed in further detail in Section~10.1.
\begin{Corollary}\label{cor:ext}
The ideal $I_D$ has a natural set of generators of the form $U + U' - 1$, given by the images of all exchange relations of $X^\univ$. \end{Corollary}
The ideal $I_D$ also contains the $|\Pi|-n$ distinguished elements which are images of $1 - F^\univ_\gamma(\z)$.

\bigskip\noindent{\bf{\large 4.4.} Proof of Theorem~\ref{thm:univ}.} 
Recall that the mutable cluster variables of $\A^\univ$ are denoted~$x_\gamma$ and the frozen variables are denoted $z_\gamma$, where $\gamma \in \Pi$. Let~$X^\univ(1) \subset \Xo^\univ \subset X^\univ$ be the locus $\{x_\gamma = 1\}$, where all mutable cluster variables have been set to 1.

By Proposition~\ref{prop:YZ}, the primitive exchange relations are of the form
\begin{gather*}
x_{\tau \gamma} x_\gamma = z_\gamma S + \prod_{\omega \in \Pi } z_{\omega}^{(\omega||\gamma)},
\end{gather*}
where $S$ is a monomial in the mutable cluster variables. So,
\begin{gather*}
f_\gamma = \frac{z_\gamma S}{x_{\tau\gamma} x_\gamma }
\end{gather*}
and thus on $X^\univ(1)$ we have $(f_\gamma)|_{X^\univ(1)} = z_\gamma$ and the relation
\begin{gather}\label{eq:f}
f_\gamma + \prod_{\omega \in \Pi } f_{\omega}^{(\omega||\gamma)} = 1.
\end{gather}
We will now show that the multiplication map gives an isomorphism
\begin{gather*}
T^\univ \times X^\univ(1) \cong \Xo^\univ,
\end{gather*}
or equivalently, every $T^\univ$-orbit on $\Xo^\univ$ intersects $X^\univ(1)$ in exactly one point. The~character group of $T^\univ$ is naturally isomorphic to $\Z^{n+r}/\tB \Z^n$, which is a free abelian group of rank $r = |\Pi|$. Thus each cluster variable $x_\gamma$ has a weight (or degree) $\wt(x_\gamma) \in \Z^{n+r}/\tB \Z^n$ (see \eqref{eq:wtuniv} for the weight of initial and frozen variables). By Proposition~\ref{prop:wtbasis}, 
the set $\{\wt(x_\gamma) \,|\, \gamma \in \Pi\}$ form a basis of the lattice $\Z^{n+r}/\tB \Z^n$. Thus we have a projection $\Xo^\univ \mapsto T^\univ$ given by sending $x \in \Xo^\univ$ to the coordinates $(x_\gamma)_{\gamma \in \Pi}$, and the fiber of this projection is $X^\univ(1)$. This is an inverse to the multiplication map $T^\univ \times X^\univ(1) \to \Xo^\univ$, and we deduce that $T^\univ \times X^\univ(1) \cong \Xo^\univ$.

We conclude that $\C\big[\Xo^\univ\big]^{T^\univ} \!\!\cong \C[X^\univ(1)]$. Now, any $T^\univ$-invariant function in $\C\big[\Xo^\univ\big]$ is a linear combination of $T^\univ$-invariant Laurent monomials in mutable and frozen variables. Each such Laurent monomial restricts to a Laurent monomial in the $z_\gamma$-s on $\C[X^\univ(1)]$. It~follows that the functions $f_\gamma$ and their inverses generate $\C\big[\Xo^\univ\big]^{T^\univ}$, and by \eqref{eq:f} satisfy the same relations that $u_\gamma \in \C[\M_D]$ satisfy. Finally, we check that the generators $f_\gamma$ do not satisfy any further relations. Suppose we have a polynomial identity $p(f_\gamma) = 0$ inside $\C\big[\Xo^\univ\big]^{T^\univ}$. The~equality $p(f_\gamma) = 0$ is equivalent to an equality $q(x_\gamma,z_\gamma) = 0$ inside $\C\big[\Xo^\univ\big]$, where $q(x_\gamma,z_\gamma)$ is a Laurent polynomial. We~claim that the primitive exchange relations allow us to eliminate all the non-initial cluster variables, i.e.,
\begin{gather*}
q(x_\gamma,z_\gamma) = r(x_1,x_2,\dots,x_n,z_\gamma) \mod \mbox{ ideal generated by primitive exchange relations},
\end{gather*}
where $r(x_1,x_2,\dots,x_n,z_\gamma)$ is a Laurent polynomial and the ideal is taken inside $\C\big[\Xo^\univ\big]^{T^\univ}$. To see this, first note that $\deg(x_1),\dots,\deg(x_n)$ and $\deg(z_\gamma)$, $\gamma \in \Pi$ together span $\Z^{n+r}$, and thus we can always multiply $q(x_\gamma,z_\gamma)$ by a $T^\univ$-invariant monomial so that the denominator involves only initial $x_i$ and the $z_\gamma$. Next, we have
\begin{gather*}
x_\gamma = \frac{M+M'}{x_{\tau \gamma}} - x_\gamma R,
\end{gather*}
where $R = \frac{M+M'}{x_{\tau \gamma} x_\gamma} - 1$ is a primitive exchange relation (divided by $x_{\tau \gamma} x_\gamma$). This allows us (modulo the ideal) to replace $x_\gamma$ by an expression involving $x_{\tau \gamma}$ and $M+M'$. If~$\gamma = (t,j)$, the mutable cluster variables that appear in $M+M'$ are either of the form $(t-1,i)$ or of the form $(t,i)$, where $i \to j$ (see Proposition~\ref{prop:YZ}). It~follows that $x_\gamma$ will not appear again when this process is repeated. This proves our claim.

But $r(x_1,x_2,\dots,x_n,z_\gamma) = 0$ as an element of $\C\big[\Xo^\univ\big]^{T^\univ} \subset \C\big[\Xo^\univ\big] \subset \C(X^\univ)$ only if the polynomial $r$ is 0, since $x_1,x_2,\dots,x_n,z_\gamma$ are algebraically independent. We~conclude that $p(f_\gamma)$ lies in the ideal generated by primitive exchange relations. Thus the ideal of relations satisfied by the $f_\gamma$ is generated by \eqref{eq:f}. We~thus have an isomorphism of rings
\begin{gather*}
\C[\M_D] \longrightarrow \C\big[\Xo^\univ\big]^{T^\univ}, \qquad u_\gamma \longmapsto f_\gamma
\end{gather*}
and an isomorphism of varieties $X^\univ(1) \cong \M_D$. 

\bigskip\noindent{\bf {\large 4.5.} Proof of Theorem~\ref{thm:fullrank}.}
By the defining property of universal coefficients, we have a~homomorphism of rings $\phi\colon \A^\univ \to \A= \A\big(\tB\big)$ such that $\phi(x_\gamma^\univ) = x_\gamma$ and $\phi(z_\gamma)$ is a~Laurent monomial in the frozen variables $x_{n+1},\dots,x_{n+m}$ of $\A$. The~homomorphism $\phi$ may not be surjective, for example this would be the case if $\tB$ has rows equal to 0, or rows that are repeated.
However, the image $\A' := \phi(\A^\univ)$ is itself a cluster algebra: it is generated by~$\phi(x_\gamma^\univ)$ and the monomials $\phi(z_\gamma)$. The~monomials $\phi(z_\gamma)$ and their inverses generate a Laurent polynomial subring $S \subset \C\big[x_{n+1}^{\pm 1}, \dots, x_{n+m}^{\pm 1}\big]$ which is the coefficient ring of $\A'$. For any monomial~$M$ in~$x_{n+1},\dots,x_{n+m}$, we can find $t \in T$ such that $t \cdot M$ is a Laurent monomial in $S$. Thus, we have $\A^T = \A'^{T'}$. Since $\A$ has full rank, the quotient $X/T = X'/T'$ has dimension $n$, and $\A'$ also has full rank.

Replacing $\A$ by $\A'$, we now assume that $\phi\colon \A^\univ \to \A$ is surjective, and thus we have a~clo\-sed immersion $\psi\colon \Xo \to \Xo^\univ$. The~monomials $\phi(z_\gamma)$, $\gamma \in \Pi$ together define a sur\-jec\-tive linear map $C'\colon \Z^r \to \Z^m$. Extending by the identity in the first $n$ coordinates, we~get a~linear map $\Z^{n+r} \to \Z^{n+m}$, represented by a matrix $C$ satisfying $C \tB^\univ = \tB$. Suppose that $t \in T= \Hom\big(\Z^{n+m}/\tB \Z^n, \C^\ast\big)$. Then composing $t$ with $C$, we get an element $t' \in T^\univ =\Hom\big(\Z^{n+r}/\tB^\univ \Z^n,\C^\ast\big)$. Since $C$ is surjective, the induced map $\psi\colon T \to T^\univ$ is injective and thus the inclusion of a subgroup.

We need to show that $\phi\colon \C\big[\Xo^\univ\big]^{T^\univ} \to \C\big[\Xo\big]^T$ is an isomorphism. For surjectivity, suppose that $f \in \C\big[\Xo\big]^T$. Then we may assume that $f$ is a Laurent monomial in mutable and frozen variables. Let~$g \in \C\big[\Xo^\univ\big]$ be such that $\phi(g) = f$. It~is immediate that $g$ is invariant under~$T$, i.e., the weight $\wt(g) \in \Z^{n+r}/\tB^\univ \Z^n$ of $g$ satisfies $C \wt(g) = 0$ in $\Z^{n+m}/\tB \Z^n$. Thus, there exists $\u \in \tB^\univ \Z^n$ such that $C(\wt(g) + \u) = 0 \in \Z^{n+m}$. The~matrix $C$ is the identity in the first $n$-coordinates, so the first $n$ coordinates of $\wt(g)+\u$ is 0. Let~$M = \prod_{\gamma \in \Pi} z_\gamma^{-a_\gamma}$, where~$(a_\gamma)$ are the last $m$ coordinates of $\wt(g)$. Then by construction we have $\wt(gM)+\u = 0$, i.e., $\wt(gM) = 0 \in \Z^{n+r}/\tB^\univ \Z^n$. Furthermore, $\phi(M) = 1$ and thus $gM \in \C\big[\Xo^\univ\big]^{T^\univ}$ satisfies $\phi(gM) = f$, proving surjectivity.

For injectivity, suppose that $\phi(g) = 0$, where $g \in \C\big[\Xo^\univ\big]^{T^\univ}$ is nonzero. We~have already shown that $X^\univ/T^\univ$ is an irreducible affine variety in Section~4.4. The~affine variety $\Spec\big(\C\big[\Xo\big]^T\big)$ is thus identified with a subvariety of $X^\univ/T^\univ$ of lower dimension. But this is impossible, since $\dim(X/T) =n = \dim(X^\univ/T^\univ)$.

The isomorphism $\C[\M_D] \cong \C\big[\Xo\big]^T$ given by $u_\gamma \mapsto f_\gamma$ now follows from Section~4.4.

\bigskip\noindent{\bf {\large 4.6.} Proof of Theorem~\ref{thm:principal}.}
The group $\Z^{2n}/\tB^\prin \Z^n$ can be naturally identified with the subgroup $\Z^n = \Z^{[1,n]} \subset \Z^{[1,2n]}=\Z^{2n}$ consisting of vectors which vanish in the last $n$-coordinates. Under this identification, the torus $T^\prin$ has character lattice $\Z^n$, and the grading on $\A^\prin$ is given by \eqref{eq:prindeg}. By Theorem~\ref{thm:fullrank}, we have $\M_D \cong \Xo^\prin/T^\prin$. It~follows from $\wt(x_i) = e_i$ that $\Xo^\prin/T^\prin$ is identified with the subvariety $\Xo^\prin(1) \subset \Xo^\prin$, where all initial cluster variables are set to 1.

The function $f^\prin_\gamma$ on $\Xo^\prin(1)$ restricts to the rational function in $y_1,\dots,y_n$ given by (see \cite[Theorem~1.5]{YZ})
\begin{gather} \label{eq:fy}
f_{(t,j)}(\y) = \begin{cases}
\frac{\prod_{i \to j} F_{(t,i)}^{-a_{ij}}(\y) \prod_{j \to i} F_{(t-1,i)}^{-a_{ij}}(\y)}{F_{(t-1,j)}(\y) F_{(t,j)}(\y) }, &t \neq 0,
\\[1.5ex]
 \frac{y_j \prod_{i \to j} F_{(t,i)}^{-a_{ij}}(\y) \prod_{j \to i} F_{(t-1,i)}^{-a_{ij}}(\y)}{F_{(t-1,j)}(\y) F_{(t,j)}(\y) }, &t = 0
 \end{cases}
\end{gather}
for $\gamma = (t,j)$ with $0 \leq t \leq r_j$.
By the following result, $\C\big[\Xo^\prin(1)] \simeq \C[\M_D\big]$ is isomorphic to the subring of $\C(y_1,y_2,\dots,y_n)$ generated by $F_\gamma^{\pm 1}(\y)$ and $y_i^{\pm 1}$.

\begin{Proposition}\label{prop:hinvert}
The rational functions $\{f_\gamma(\y) \,|\, \gamma \in \Pi\}$ and $\{y_1,\dots,y_n\} \cup \{ F_\gamma(\y) \,|\, \gamma \in \Pi^+\}$ are related by an invertible monomial transformation.
\end{Proposition}
The proof of Proposition~\ref{prop:hinvert} is delayed until Section~6.4.

\bigskip\noindent{\bf {\large 4.7.} Proof of Proposition~\ref{prop:forget}.} 
Using $\tau$, let us assume that $\gamma = (0,j)$ so that $x_\gamma = x_j$ is an initial mutable cluster variable. Let~$\tB$ be full rank of type $D$ and let $\A\big(\tB_j\big)$ denote the cluster algebra of type $D_1$ that is obtained by freezing the variable $x_j$ in $\A\big(\tB\big)$. The~extended exchange matrix $\tB_j$ is obtained from $\tB$ by removing the $j$-th row and we have $\A\big(\tB\big)\big[x_j^{-1}\big] = \A\big(\tB_j\big)$. Thus $\Ao\big(\tB_j\big) \subset \Ao\big(\tB\big)$. The~action of the cluster automorphism group $T\big(\tB\big)$ extends to an action on~$\A\big(\tB_j\big)$ and we can identify $T\big(\tB\big)$ with a subgroup of $T\big(\tB_j\big)$. The~morphism $\M_{D} \to \M_{D_1}$ corresponds to the inclusion of rings $\Ao\big(\tB_j\big)^{T(\tB_j)} \subset \Ao\big(\tB_j\big)^{T(\tB)} \subset \Ao\big(\tB\big)^{T(\tB)}$.

\section[MD as an affine open in a projective toric variety]
{$\boldsymbol{\tM_D}$ as an affine open in a projective toric variety}

In this section, we show that the partial compactification $\tM_D$ is an affine open subspace of the projective toric variety $X_{\N(B^\vee)}$ associated to the cluster fan of $B^\vee$. The~stratification (Proposition~\ref{prop:MDstrat}) of $\tM_D$ is inherited from the natural stratification of $X_{\N(B^\vee)}$ by torus orbits. Our approach follows that of \cite{AHL}.

\bigskip\noindent{\large\bf 5.1.}
Let $\C(\y) = \C(y_1,\dots,y_n)$ denote the field of rational functions. Recall that for $\gamma \in \Pi$, we have defined $f_\gamma(\y) \in \C(\y)$ in \eqref{eq:fy}.
By the proof of Theorem~\ref{thm:principal}, $\C[\M_D]$ is isomorphic to the subring of $\C(\y)$ generated by $f_\gamma(\y)^{\pm 1}$. Define $R_B \subset \C(\y)$ to be the subring generated by $\{f_\gamma(\y) \,|\, \gamma \in \Pi\}$. Some examples of $f_\gamma(\y)$ are computed in Examples~\ref{ex:A3} and~\ref{ex:B3}.

\begin{Theorem}\label{thm:RDisom}
The coordinate ring $\C\big[\tM_D\big]$ is isomorphic to $R_B$. \end{Theorem}
\begin{proof}
There is a surjective ring homomorphism $\varphi'\colon \C[u] \to R_B$ given by $u_\gamma \mapsto f_\gamma(\y)$. We~already know that the kernel $K$ of $\varphi$ contains the ideal $I_D \subset \C[u]$. We~need to show that the homomorphism $\varphi\colon \C[u]/I_D \to R_B$ is an isomorphism. From Theorem~\ref{thm:principal}, we know this holds after inverting the $\{u_\gamma \,|\, \gamma \in \Pi\}$ and $\{f_\gamma \,|\, \gamma \in \Pi\}$.

By definition, $\tM_D(\{\gamma\})$ is cut out of $\tM_D$ by the ideal $(u_\gamma)$. Thus the ring $\C[u]/(I_D+(u_\gamma))$ is isomorphic to $\C\big[\tM_D(\{\gamma\})\big]$, and by Proposition~\ref{prop:divisorsplit}, we have $\tM_D(\{\gamma\}) = \tM_{D_1} \times \tM_{D_2} \times \cdots \allowbreak \times \tM_{D_s}$ for some Dynkin diagrams $D_i$. By induction on the rank of $D$, we have that $\varphi_i\colon \C[u]/\allowbreak(I_D+(u_\gamma)) \to R_B/(f_\gamma(\y))$ is an isomorphism. Applying Lemma \ref{lem:ringiso}, we conclude that $\varphi$ itself is an isomorphism.
\end{proof}

\bigskip\noindent{\large\bf 5.2.}
We give another description of $R_B \subset \C(\y)$. Let~$R(\y) = P(\y)/Q(\y) \in \C(\y)$ be a rational function such that $P(\y),Q(\y) \in \Z[\y]$ have positive integer coefficients. Then $\Trop(R(\y))$ is the piecewise-linear function on $\R^n$ given by the formal substitution
\begin{gather*}
y_i \mapsto Y_i, \qquad (+, \times, \div) \mapsto (\min, + , -).
\end{gather*}
For example, $\Trop\big(\big(3y_1^2y_2 + y_2^2\big)/(y_2+6y_3)\big) = \min(2Y_1+Y_2,2Y_2)-\min(Y_2,Y_3)$. Note that the coefficients are unimportant since, for example, $\Trop(2y) = \Trop(y+y) = \min(Y,Y) = Y$.

The domains of linearity of the piecewise-linear function $L(\Y) = \Trop(R(\y))$ define the structure of a complete fan on $\R^n$. A piecewise-linear function $L(\Y)\colon \R^n \to \R$ is called \emph{nonnegative}, denoted $L(\Y)\geq 0$, if it takes nonnegative values on $\R^n$.

\begin{Proposition}\label{prop:RD}
The ring $R_B$ is equal to the subring of $\C(\y)$ generated by rational functions~$R(\y)$ satisfying \begin{enumerate}\itemsep=0pt
\item[$(1)$]
$R(\y) = \prod_{i=1}^n y_i^{a_i} \prod_{\gamma \in \Pi} F_\gamma(\y)^{a_\gamma}$ is a Laurent monomial in $y_i$ and $F_\gamma(\y)$,
\item[$(2)$]
$\Trop(R(\y))$ is nonnegative.
\end{enumerate}
\end{Proposition}

\begin{proof}
Let $R(\y)$ be a Laurent monomial in $\{y_i, F_\gamma(\y)\}$. By Theorem~\ref{thm:FNewton}, the domains of linearity of the function $L(\Y) = \Trop(R(\y))$ is a coarsening of the negative of the cluster fan $-\N(B^\vee)$. Thus $L(\Y)$ is uniquely determined by $b_\gamma = L(-\g_\gamma)$ as $\gamma$ varies over $\Pi$, and $\g_\gamma$ denotes a $\g$-vector. As in the proof of Proposition~\ref{prop:hinvert}, we have $R(\y) = \prod_\gamma f_\gamma(\y)^{-b_{\tau\gamma}}$. The~condition $L(\Y) \geq 0$ is equivalent to $b_\gamma \geq 0$ for all $\gamma \in \Pi$. Thus the subring of rational functions $R(\y)$ satisfying (1) and (2) is exactly the subring $R_D$.
\end{proof}

\medskip\noindent{\large\bf 5.3.}
The Laurent polynomial ring $\C\big[y_1^{\pm 1}, \dots, y_n^{\pm 1}\big]$ is the coordinate ring of an $n$-dimensional torus $T_\y$. Recall that $F(\y) = \prod_\gamma F_\gamma(\y)$. The~following result is an application of \cite[Section~10]{AHL}.

\begin{Theorem}\label{thm:toric}
The affine scheme $\tM_D$ is isomorphic to the affine open $\{F(\y) \neq 0\}$ in the projective toric variety $X_{\N(B^\vee)}$ associated to the complete fan $\N\big(B^\vee\big)$. The~subvariety \mbox{$\M_D \subset \tM_D$} is identified with the intersection of $\{F(\y) \neq 0\}$ with the open torus orbit $T_\y$ in $X_{\N(B^\vee)}$.
\end{Theorem}
\begin{proof}
For any $\g \in \R^n$, the quantity $\Trop(F(\y))(\g)$ is equal to the minimum value that the linear function $\Y \mapsto \Y \cdot \g$ takes on the Newton polytope $P$ of $F(\y)$. Thus by Theorem~\ref{thm:FNewton}, the outer normal fan of $P$ is equal to $\N\big(B^\vee\big)$. Recall that a lattice polytope $Q$ is called \emph{very ample} if for sufficiently large integers $r>0$, every lattice point in $rQ$ is a sum
of $r$ (not necessarily distinct) lattice points in $Q$. For any lattice polytope $Q$, it is known that some integer dilation~$cQ$ is very ample. So let $c \in \Z_{>0}$ be such that $cP$ is very ample and let $\{\v_1,\dots,\v_k\} = cP \cap \Z^n$ be the set of all lattice points in $cP$. For $\v \in \Z^n$, let $\y^\v$ be the monomial with exponent vector~$\v$. Then $X_{\N(B^\vee)}$ can be explicitly realized as the closure of the set of points
\begin{gather*}
\Big\{\big[\y^{\v_1}:\cdots: \y^{\v_k}\big] \in \P^{k-1} \,|\, \y \in T_\y \Big\}
\end{gather*}
inside the projective space $\P^{k-1}$. The~polynomial $F(\y)^c$ can be identified with a hyperplane section of $X_{\N(B^\vee)}$ in this projective embedding, and the affine open $V:=\{F(\y) \neq 0\}$ is the complement of this hyperplane section. The~coordinate ring $\C[V]$ is generated by the functions $\y^{\v_i}/F(\y)^c$, $i =1,2,\dots,k$. Since $\Trop\big(\y^{\v_i}/F(\y)^c\big)$ is nonnegative, by Proposition~\ref{prop:RD}, we have $\C[V] \subset R_B$. It~is also not hard to see that $f_\gamma(\y) \in \C[V]$ (see~\cite[Section~10]{AHL}), and we have $\C[V] = R_B$ as subrings of $\C(\y)$. The~theorem now follows from Theorem~\ref{thm:RDisom}.
\end{proof}

\begin{question}
Is $P$, the Newton polytope of $F(\y) = \prod_\gamma F_\gamma(\y)$, very ample? Is $P$ normal?
\end{question}

\begin{question}\label{q:sat}
Is the polynomial $F(\y)$ saturated?
\end{question}

\begin{question}\label{q:sum}
Is every lattice point in $P$ a sum of lattice points in $P_\gamma$?
\end{question}

Fei \cite{Fei} has shown that $F_\gamma(\y)$ is saturated in the simply-laced case (and in more general situations). Thus Questions~\ref{q:sat} and~\ref{q:sum} are equivalent in that case.

\bigskip\noindent{\bf{\large 5.4.} Proof of Theorem~\ref{thm:geom}.} 
By Theorem~\ref{thm:clusterfan}, the fan $\N\big(B^\vee\big)$ is a smooth, simplicial, polytopal, complete fan. Thus $X_{\N(B^\vee)}$ is a smooth projective toric variety and the torus-orbit closure stratification of $X_{\N(B^\vee)}$ is simple normal-crossing.

\section[Properties of F-polynomials]
{Properties of $\boldsymbol F$-polynomials}\label{sec:F}

We establish some technical properties of $F_\gamma(\y)$ and $F^\univ_\gamma(\z)$, following the approach of \cite{BDMTY}. The~statements are first established in the case of simply-laced $D$; the multiply-laced case follows from folding. Another closely related approach is that of \cite{PPPP}, which would presumably avoid folding.

A key technical result is Theorem~\ref{thm:pair} which gives the values of the tropicalization $\Trop(f_\gamma(\y))$ on a (negated) $\g$-vector.

In this section, we will assume that $D$ is a finite type Dynkin diagram whose underlying tree has been given an orientation, and we let $B$ denote the corresponding exchange matrix. Recall that $D^\vee$ denotes the dual Dynkin diagram, and we let $B^\vee$ denote the exchange matrix of type~$D^\vee$, satisfying the condition: $B_{ij} > 0$ if and only if $B^\vee_{ij} > 0$. Recall that we write $i \to j$ if~$B_{ij} > 0$.

\bigskip\noindent{\large\bf 6.1.}
Let $B$ be the exchange matrix corresponding to the oriented Dynkin diagram $D$. Let~$\R^\Pi$ be the vector space with basis indexed by $\Pi$, and write $(p_\gamma)_{\gamma \in \Pi}$ for a typical vector in $\R^\Pi$. Define $\Pi^+ :=\{(s,i) \,|\, 1\leq s \leq r_i\} \subset \Pi$ and let $\c = (c_\gamma)_{\gamma \in \Pi^+}$ denote a typical vector in $\R^{\Pi^+}$. Following~\cite{BDMTY}, we consider the \emph{$\c$-deformed mesh relations}
\begin{gather}\label{eq:c}
p_{(t-1,j)} + p_{(t,j)} = c_{(t,j)} + \sum_{i \to j} |B_{ij}| p_{(t,i)}+ \sum_{j \to i} |B_{ij}| p_{(t-1,i)},
\end{gather}
where $(t,j) \in \Pi^+$. (Compare with \eqref{eq:exchange}, and note that if $i \to j$ then $B_{ij} >0$, but if $j \to i$ then $B_{ij} < 0$.) If $\c = 0$, we call \eqref{eq:c} the ${\mathbf 0}$-mesh relations.

For $\c=(c_\gamma) \in \R^{\Pi^+}$, we let $\E_\c \subset \R^{\Pi}$ denote the solutions to \eqref{eq:c}, and let $\U_\c := \E_\c \cap \R_{\geq 0}^{\Pi}$ denote the intersection of $\E_\c$ with the positive orthant. Let~$\pi\colon \R^{\Pi} \to \R^n$ denote the projection onto the coordinates $p_\gamma$, where $\gamma$ varies over $\{(r_i,i) \,|\, i=1,2,\dots,n\}$. (Up to the action of $\tau^{-1}$, this is the same as projection onto the initial cluster variables.)

\def\e{{\mathbf e}}
We use the notation $\U(D)_\c$ and $\E(D)_\c$ \big(resp.\ $\U\big(D^\vee\big)_\c$ and $\E\big(D^\vee\big)_\c$\big) to denote these objects for $B$ or $D$ \big(resp.\ $B^\vee$ or $D^\vee$\big). In~the following, $e_\gamma$ denotes the unit basis vector in $\R_{>0}^{\Pi^+}$.

\begin{Theorem}\label{thm:Newt} \
\begin{enumerate}\itemsep=0pt
\item[$1.$]
If $\c=(c_\gamma) \in \R_{>0}^{\Pi^+}$, then the normal fan of $\pi(\U(D)_{\c})$ is equal to $\N(B)$. If~$(c_\gamma) \in \R_{\geq 0}^{\Pi^+}$, then the normal fan of $\pi(\U(D)_{\c})$ is a coarsening of $\N(B)$.
\item[$2.$]
For $\gamma \in \Pi^+$, the polytope $\U\big(D^\vee\big)_{e_\gamma}$ is the Newton polytope of $F^\univ_\gamma(\z)$.
\item[$3.$]
For $\gamma \in \Pi^+$, the polytope $\pi\big(\U\big(D^\vee\big)_{e_\gamma}\big)$ is the Newton polytope of $F_\gamma(\y)$.
\end{enumerate}
\end{Theorem}

\begin{proof}[Proof of Theorem~\ref{thm:FNewton}]
By Theorem~\ref{thm:Newt}(3) the Newton polytope $P_\gamma$ of $F_\gamma(\y)$ is $\pi\big(\U\big(D^\vee\big)_{e_\gamma}\big)$. The~Newton polytope~$P$ of $\prod_\gamma F_\gamma(\y)$ is the Minkowski sum of the $P_\gamma$, and by Theorem~\ref{thm:Newt}(1), we conclude that $P$ is a generalized associahedron.
\end{proof}

We let $\g^\vee_\gamma$ denote the $\g$-vector for $B^\vee$ indexed by the element of $\Pi(B^\vee)$ corresponding to $\gamma$ under the bijection of Section~2.5.

\bigskip\noindent{\bf{\large 6.2.} Proof of Theorem~\ref{thm:Newt}.}
For $D$ simply-laced, we have $D = D^\vee$ and Theorem~\ref{thm:Newt} is proven in~\cite{BDMTY}. We~now prove it for multiply-laced $D$ via folding.

Let $D$ be a folding of $\tD$ with folding group $\Gamma$, and $\nu\colon \tPi \to \Pi$ the quotient map on cluster variables from Section~2.14. Define $\nu\colon \R^{\tPi} \to \R^{\Pi}$ by $\nu(e_\tgamma) = e_{\nu(\tgamma)}$, and $\nu^\vee\colon \R^{\tPi} \to \R^{\Pi}$ by $\nu^\vee(e_\tgamma) = \frac{1}{|\nu^{-1}(\nu(\tgamma))|}e_{\nu(\tgamma)}$. (The finite set $\nu^{-1}(\gamma)$ has cardinality one, two, or three.) Similarly, we have $\nu,\nu^\vee\colon \R^{\tPi^+} \to \R^{\Pi^+}$.

\begin{Lemma}\label{lem:fold}
If $(p_{\tgamma})_{\tgamma \in \tPi} \in \E\big(\tD\big)_{\tc}$ then $\nu^\vee(p_{\tgamma}) \in \E(D)_{\nu^\vee(\tc)}$ and $\nu(p_{\tgamma}) \in \E\big(D^\vee\big)_{\nu(\tc)}$.
\end{Lemma}

The $\g$-vectors for $\tD$ are solutions to the $0$-mesh relations in the following sense: for each $i =1,2,\dots,n$, the $i$-th coordinates of $\g_\gamma$ give a vector $\g^{(i)}$ that belongs to $\E_0$. This follows from~\cite[relation~(6.13)]{CA4}, noting that the sign-coherence conjecture \cite[Conjecture~6.13]{CA4} holds in our case.

The following follows from Lemma \ref{lem:fold} and Proposition~\ref{prop:folding}(3). (The appearance of $\nu^\vee$ seems to contradict Lemma \ref{lem:fold}, but it is actually correct: the $\nu^\vee$ in Lemma \ref{lem:fold} acts on $\R^{|\tPi|}$ while the~$\nu^\vee$ below acts only on $\R^{\tI}$.)
\begin{Proposition}\label{prop:gvee}
We have $\g^\vee_\gamma = \nu^\vee\Big( \sum_{\tgamma \in \nu^{-1}(\gamma)}\g_\tgamma\Big)$.
\end{Proposition}

We say that $(p_{\tgamma})_{\tgamma \in \tPi} \in \R^{\tPi}$ is $\Gamma$-invariant and write $(p_{\tgamma})_{\tgamma \in \tPi} \in \big(\R^{\tPi}\big)^\Gamma$ if for all $g \in \Gamma$, we have $p_{\tgamma} = p_{g \cdot \tgamma}$. Similarly, we define $\Gamma$-invariants $\tc \in \big(\R^{\tPi^+}\big)^\Gamma$. The~following result follows from Lemma \ref{lem:fold}.

\begin{Proposition}\label{prop:bijfold} \
\begin{enumerate}\itemsep=0pt
\item[$1.$]
Suppose that $\tc \in \big(\R^{\tPi^+}\big)^\Gamma$.
Then the linear map $\nu^\vee$ $($resp.\ $\nu)$ is a bijection between $\E\big(\tD\big)_{\tc} \cap \big(\R^{\tPi}\big)^\Gamma$ and $\E(D)_{\nu^\vee(\tc)}$ $\big($resp.\ $\E\big(D^\vee\big)_{\nu(\tc)}\big)$.
\item[$2.$]
Suppose that $\tc \in (\R_{\geq 0}^{\tPi^+})^\Gamma$.
Then the linear map $\nu^\vee$ $($resp.\ $\nu)$ is a bijection between $\U\big(\tD\big)_{\tc} \cap \big(\R^{\tPi}\big)^\Gamma$ and $\U(D)_{\nu^\vee(\tc)}$ $\big($resp.\ $\U\big(D^\vee\big)_{\nu(\tc)}\big)$.
\end{enumerate}
\end{Proposition}

\begin{proof}[Proof of Theorem~\ref{thm:Newt}]
In this proof we write $\N(D)$ for $\N(B)$ to avoid conflict of notation. Let~$\tD$ fold onto $D$. Let~$\c \in \R_{>0}^{\Pi^+}$ and pick $\tc \in \big(\R_{>0}^{\tPi^+}\big)^\Gamma$ satisfying $\c = \nu^\vee(\tc)$. By Proposition~\ref{prop:bijfold}(2), the map $\nu^\vee$ is a bijection between $\U\big(\tD\big)_{\tc} \cap \big(\R^{\tPi}\big)^\Gamma$ and $\U(D)_{\c}$. To prove Theorem~\ref{thm:Newt}(1) for $D$, it thus suffices to show that $\pi\big(\U\big(\tD\big)_{\tc} \cap \big(\R^{\tPi}\big)^\Gamma\big) = \pi\big(\U\big(\tD\big)_{\tc}\big) \cap \big(\R^{\tI}\big)^\Gamma \subset \big(\R^{\tI}\big)^\Gamma$ has normal fan $\N(D)$.
By Theorem~\ref{thm:Newt}(1) for $\tD$, the polytope $\pi\big(\U\big(\tD\big)_{\tc}\big)$ has normal fan~$\N\big(\tD\big)$, and by our choice of $\tc$, it is $\Gamma$-invariant. The~faces of $\pi\big(\U\big(\tD\big)_{\tc}\big)$ that intersect $\big(\R^{\tI}\big)^\Gamma$ are exactly those normal to the cones $\{\tgamma_1,\dots,\tgamma_a\}$ of $\N\big(\tD\big)$ consisting of $\Gamma$-invariant pairwise compatible collections. Combining with Proposition~\ref{prop:folding}(4), we deduce that the normal fan of $\pi\big(\U\big(\tD\big)_{\tc}\big) \cap \big(\R^{\tI}\big)^\Gamma$ is $\N(D)$. This proves the first statement of Theorem~\ref{thm:Newt}(1) for $D$, and the second statement is similar.

Now, let $\nu(\tgamma) = \gamma$. By Proposition~\ref{prop:folding}(3), the Newton polytope of $F_\gamma^\univ(\z)$ is the image of the Newton polytope of $F_\tgamma^\univ(\z)$ under the map $\nu$. By Proposition~\ref{prop:bijfold}(2) and Theorem~\ref{thm:Newt}(2) for $\tD$, the Newton polytope of $\prod_{\tgamma \in \nu^{-1}(\gamma)} F_\tgamma^\univ(\z)$ is equal to $\U\big(\tD\big)_{\sum_{\tgamma \in \nu^{-1}(\gamma)} e_{\tgamma}}$. Thus by Proposition~\ref{prop:bijfold}(2), the Newton polytope of $(F_\gamma^\univ(\z))^{|\nu^{-1}(\gamma)|}$ is equal to $\U\big(D^\vee\big)_{|\nu^{-1}(\gamma)| e_{\gamma}}$, and Theorem~\ref{thm:Newt}(2) for $D$ follows. Finally, Theorem~\ref{thm:Newt}(3) follows from Proposition~\ref{prop:folding}(3).
\end{proof}

\bigskip\noindent{\large\bf 6.3.}
The character group of $T^\univ$ is $\Z^{n+r}/\tB^\univ \Z^n$. Let~$\{e_1,\dots,e_n\} \cup \{e_\gamma \,|\, \gamma \in \Pi\}$ be basis vectors of $\Z^{n+r}$. We~have
\begin{gather}\label{eq:wtuniv}
\wt(x_i) = e_i\quad \text{for}\quad i = 1,2,\dots,n, \qquad \text{and} \qquad
\wt(z_\gamma) = e_\gamma\quad \text{for}\quad \gamma \in \Pi.
\end{gather}
For $\gamma \in \Pi^+$, we have
\begin{gather*}
\wt(x_\gamma) = \wt\big(F^\univ_\gamma\big) \mod \tB^\univ \Z^n + \sp({e_1,\dots,e_n}).
\end{gather*}
Note that all monomials in $F^\univ_\gamma$ have the same weight modulo $\tB^\univ \Z^n + \sp({e_1,\dots,e_n})$.

\begin{Proposition}\label{prop:wtbasis}
The sets \begin{gather*}
\big\{\wt\big(F^\univ_\gamma\big) \,|\, \gamma \in \Pi^+ \big\} \qquad \text{and} \qquad
\{\wt(x_\gamma) \,|\, \gamma \in \Pi\}
\end{gather*}
are bases of $\Z^{n+r}/\big(\tB^\univ \Z^n + \sp({e_1,\dots,e_n})\big)$ and $\Z^{n+r}/\tB^\univ \Z^n$ respectively.
\end{Proposition}
\begin{proof}
The first statement implies the second. By Theorem~\ref{thm:Newt}(2), for $\gamma \in \Pi^+$, we have $\wt\big(F^\univ_\gamma\big) \in \E(D^\vee)_{e_\gamma}$. The~equations \eqref{eq:c} define a linear map $L\colon \R^{\Pi} \to \R^{\Pi^+}$, sending $(p_\gamma)$ to~$(c_\gamma)$. Let~$B'$ be the last $r$ rows of $\tB^\univ$. By \cite{Reading}, see also \cite[Section~8]{BDMTY}, the matrix $B'$ has rows given by $-\g^\vee_\gamma$. By \cite[relation~(6.13)]{CA4}, the $\g$-vectors are solutions to the $0$-mesh relations, and thus the kernel of $L$ is exactly $B' \Z^n$. We~conclude that modulo $B'\Z^n$, the last $r$ entries of~$\wt\big(F^\univ_\gamma\big)$ is equal to the basis vector $e_\gamma$. Returning to the vector $\wt\big(F^\univ_\gamma\big) \in \Z^{n+r}$, we~obtain
\begin{gather*}
\wt\big(F^\univ_\gamma\big) = e_\gamma \mod \tB^\univ \Z^n + \sp({e_1,\dots,e_n}).
\end{gather*}
Thus
$\big\{\wt\big(F^\univ_\gamma\big) \,|\, \gamma \in \Pi^+\big\}$
form a basis of $\Z^{n+r}/\big(\tB^\univ \Z^n + \sp({e_1,\dots,e_n})\big)$.
\end{proof}

\bigskip\noindent{\large\bf 6.4.} 
Recall the definition of $f_\gamma(\y)$ from \eqref{eq:fy}. Let~$A_\gamma(\g):= \Trop(f_\gamma(\y))$, where we take $(g_1,\dots,g_n)$ as the tropicalization of $y_1,\dots,y_n$, for example $\Trop(1+y_1+y_1y_2) = \min(0,g_1,\allowbreak g_1+g_2)$.

\begin{Theorem}\label{thm:pair}
For $\gamma,\omega \in\Pi$, we have $A_\gamma\big({-}\g^\vee_\omega\big) = \delta_{\omega, \tau \gamma}$.
\end{Theorem}

\begin{proof}
First, assume that $D$ is simply-laced so that $\g_\omega = \g^\vee_\omega$. For $\gamma \in \Pi$, let $W_\gamma \in D^b( {\rm rep}\,Q^{\rm op})$ be the object indexed by $\gamma$ in the bounded derived category of representations of the quiver $Q^{\rm op}$ corresponding to the reversed orientation of $D$, see \cite[Section~3]{BDMTY}. For any $\g \in \R^n$, the quantity $\Trop(F_\gamma(\y))(\g)$ is equal to the minimum value that the linear function $\Y \mapsto \Y \cdot \g$ takes on the Newton polytope $P_\gamma$.

Now take $\gamma \in \Pi^+$. Let~$G$ be the $n \times |\Pi|$ matrix whose columns are $\g_\omega$. According to \cite[Proof of Theorem~1]{BDMTY}, the map $\Y \mapsto \Y \cdot (-G) + \v_{e_{\gamma}}$ is a diffeomorphism between $P_\gamma$ and $\U_{e_\gamma}$, where $\v_{e_{\gamma}} \in \E_{e_\gamma}$ is the integer vector given by $(\v_{e_\gamma})_{\omega} = \dim \Hom(W_\omega, W_{\tau \gamma})$. Here, $\Hom$ is taken within $D^b\big( {\rm rep}\,Q^{\rm op}\big)$. Furthermore, it follows from~\cite{BDMTY} that $\U_{e_\gamma}$ has nonempty intersection with every coordinate hyperplane. Thus,
\begin{gather*}
\Trop(F_\gamma(\y))(-\g_\omega) = -\dim \Hom(W_\omega, W_{\tau \gamma}).
\end{gather*}
Suppose $\tau \gamma = (t,i) \in \Pi^+$. Then we have an Auslander--Reiten triangle in $D^b\big( {\rm rep}\, Q^{\rm op}\big)$
\begin{gather*}
W_{(t-1,i)} \to E \to W_{(t,i)} \to W_{(t-1,i)} [1],
 \qquad \text{where} \quad
E = \bigoplus_{i \to j} W_{(t,j)} \oplus \bigoplus_{j \to i} W_{(t-1,j)}.
\end{gather*}
We have an exact sequence
\begin{gather*}
0 \to \Hom\big(W_\omega, W_{(t-1,i)}\big) \to \Hom(W_\omega, E) \to \Hom\big(W_\omega, W_{(t,i)}\big)
\\ \hphantom{0}
{}\to \Hom\big(W_\omega, W_{(t-1,i)} [1]\big) \to \cdots.
\end{gather*}
By the definition of Auslander--Reiten triangle, any map from $W_\omega$ to $W_{(t,i)}$ which is not an isomorphism factors through $E$. Thus
\begin{gather*}
\big(\dim \Hom\big(W_\omega, W_{(t-1,i)}\big) + \dim\Hom\big(W_\omega, W_{(t,i)}\big)\big) - \dim \Hom(W_\omega, E) =\delta_{\omega,(t,i)}
\end{gather*}
and this is exactly $A_\gamma(-\g_\omega)$. Now if $\gamma \in \Pi^+$ but $\tau \gamma$ is initial, then we have an Auslander--Reiten triangle of the form
\begin{gather*}
W_{(r_{i^*}-1,i^*)}[-1] \to E \to W_{(0,i)} \to W_{(r_{i^*}-1,i^*)}
\end{gather*}
and $\Hom\big(W_\omega,W_{(r_{i^*}-1,i^*)}[-1]\big)= \Ext^{-1}\big(W_\omega,W_{(r_{i^*}-1,i^*)}\big) = 0$ for all $\omega \in \Pi$, agreeing with our convention that $F_{(0,i)}(\y) = 1$, so again we have $A_\gamma(-\g_\omega) = \delta_{\omega, \tau \gamma}$. Finally, suppose that $\gamma = (0,i)$ itself is initial. Then we have an Auslander--Reiten triangle of the form
\begin{gather*}
W_{(r_{i^*}-1,i^*)} \to E \to W_{(0,i)}[1] \to W_{(r_{i^*}-1,i^*)} [1].
\end{gather*}
In this case, our formula for $f_\gamma(\y)$ includes a factor of $y_i$, and $\Trop(y_i)(\g_\omega)$ is simply the $i$-th coordinate of $\g_\omega$. Our claim then follows from the interpretation \cite[Section~6]{BDMTY} of $\g$-vectors as a change of basis between the summands of a tilting object and the indecomposable projectives $\big\{W_{(0,1)},\dots,W_{(0,n)}\big\}$ (see also \cite[Theorem~3.23(ii)]{PPPP}).

Now, suppose that $D$ is multiply-laced and let $\tD$ be the simply-laced diagram that folds to $D$. Note that for $\u \in \R^{|\tI|}$ and $\v \in \big(\R^{|\tI|}\big)^\Gamma$, we have $\u \cdot \v = \nu(\u) \cdot \nu^\vee(\v)$.
By Proposition~\ref{prop:gvee}, $\Trop(F_\gamma(\y))(-\g^\vee_\omega)$ is equal to the minimum value that the linear function $\Y \mapsto \Y \cdot \big({-}\sum_{\tomega \in \nu^{-1}(\omega)}\g_\tomega\big)$ takes on the Newton polytope $P_{\tgamma}$ (where $\tgamma \in \nu^{-1}(\gamma)$), and we thus have $\Trop(F_\gamma(\y))\big({-}\g^\vee_\omega\big) = \sum_{\tomega \in \nu^{-1}(\omega)} \Trop(F_\tgamma(\y))(-\g_\tomega)$. It~follows from the definitions that $f_\gamma(\y) = \nu(f_{\tgamma}(\y))$ for any $\tgamma \in \nu^{-1}(\gamma)$. The~equality $A_\gamma\big({-}\g^\vee_\omega\big) = \delta_{\omega, \tau \gamma}$ for $D$ thus follows from the same equality for~$\tD$.
\end{proof}

\begin{proof}[Proof of Proposition~\ref{prop:hinvert}]
Let $m(\y)$ be a Laurent monomial in $\{y_i, F_\gamma(\y)\}$, and denote by $G=G(\g):=\Trop(m(\y))$ the piecewise-linear function that is the tropicalization of $m(\y)$. (Recall that by convention the variables $\g$ are the tropicalizations of the variables $\y$.)
The~domains of linearity of the function $G$ is a coarsening of $-\N\big(B^\vee\big)$, so the function~$G$ is uniquely determined by the integer vector $\big(G\big({-}\g^\vee_\gamma\big) \,|\, \gamma \in \Pi\big) \in \Z^{|\Pi|}$. By Theorem~\ref{thm:pair}, any vector in $\Z^{|\Pi|}$ can arise in this way. It~follows that $m(\y)$ is uniquely determined by its tropicalization $G$ by the formula $m(\y) = \prod_\gamma f_\gamma(\y)^{G(-\g^\vee_{\tau \gamma})}$.
\end{proof}

\begin{Example}\label{ex:A3}
We illustrate Theorem~\ref{thm:pair} for the exchange matrix
\begin{gather*}
B = \begin{bmatrix}
0 & 1 & 0 \\
-1 & 0 & -1 \\
0 & 1 & 0
\end{bmatrix}
\end{gather*}
of type $A_3$. In~this case, we have $\Pi = \{0,1,2\} \times \{1,2,3\}$ and $i^* = 4-i$. We~tabulate $\g_\gamma$, $F_\gamma(\y)$, and $f_\gamma(\y)$ below:
\begin{center}
\setlength{\tabcolsep}{4.0pt}
\renewcommand{\arraystretch}{1.2}
\begin{tabular}{c|c|c|c}
\hline
$\gamma$ & $\g_\gamma$ & $F_\gamma(\y)$ &$f_\gamma(\y)$\\
\hline
$(0,1)$ & $(1,0,0)$ & 1 & $\dfrac{y_1(1+y_2)}{1+y_1+y_1y_2} \vphantom{\Bigg(}$\\
\hline
$(0,2)$ & $(0,1,0)$ & 1 & $\dfrac{y_2}{1+y_2}\vphantom{\Bigg(} $\\
\hline
$(0,3)$ & $(0,0,1)$ & 1 & $\dfrac{y_3(1+y_2)}{1+y_3+y_2y_3}\vphantom{\Bigg(} $\\
\hline
$(1,1)$ & $(-1,1,0)$ & $1+y_1$ & $\dfrac{1}{1+y_1} \vphantom{\Bigg(}$\\
\hline
$(1,2)$ & $(-1,1,-1)$ & $1+y_1+y_3+y_1y_3+y_1y_2y_3$ &$ \dfrac{(1+y_1)(1+y_3)}{1+y_1+y_3+y_1y_3+y_1y_2y_3}\vphantom{\Bigg(} $\\
\hline
$(1,3)$ & $(0,1,-1)$ & $1+y_3$ & $ \dfrac{1}{1+y_3}\vphantom{\Bigg(}$\\
\hline
$(2,1)$ & $(0,0,-1)$ & $1+y_3+y_2y_3$ & $\dfrac{1+y_1+y_3+y_1y_3+y_1y_2y_3}{(1+y_1)(1+y_3+y_2y_3)}\vphantom{\Bigg(}$ \\
\hline
$(2,2)$ & $(0,-1,0)$ & $1+y_2$ & $\dfrac{(1+y_1+y_1y_2)(1+y_3+y_2y_3)}{(1+y_2)(1+y_1+y_3+y_1y_3+y_1y_2y_3)}\vphantom{\Bigg(}$\\
\hline
$(2,3)$ & $(-1,0,0)$ & $1+y_1+y_1y_2$ & $\dfrac{1+y_1+y_3+y_1y_3+y_1y_2y_3}{(1+y_3)(1+y_1+y_1y_2)}\vphantom{\Bigg(}$\\
\hline
\end{tabular}
\end{center}

\noindent
Taking $\gamma = (1,2)$ as an example, we have
\begin{gather*}
A_\gamma(\g) = \min(0,g_1)+\min(0,g_3) - \min(0,g_1,g_3,g_1+g_3,g_1+g_2+g_3)
\end{gather*}
and one can verify that it takes value 0 on all negatives of $\g$-vectors except for $-\g_{(0,2)}$, where it takes value $1$.
\end{Example}

\begin{Example}\label{ex:B3}
Consider the following exchange matrix of type $B_3$:
\[
B=\begin{bmatrix} 0&1&0 \\ -1&0&1 \\ 0&-2&0 \end{bmatrix}\!.
\]
We have $\Pi = \{(t,j) \in [0,3] \times [1,3]\}$ and $\tau(t,j) = ((t-1) \mod 4,j)$. We~tabulate the $\g$-vectors, the $\g^\vee$-vectors, the $F$-polynomials, and the polynomials $f_\gamma(\y)$ in Table~\ref{fig:B3}.
\begin{table}[t!]
\setlength{\tabcolsep}{2.0pt}
\renewcommand{\arraystretch}{1.2}
\caption{$\g$-vectors, $F$-polynomials, and $f_\gamma(\y)$ in type $B_3$.}
\label{fig:B3}
\begin{tabular}{c|c|c|c|c}
\hline
$\gamma$ & $\g_\gamma$ & $\g^\vee_\gamma$ & $F_\gamma(\y)$ & $f_\gamma(\y)$\\
\hline
$(0,1)$ & $\begin{bmatrix}1\\[-1ex]0\\[-1ex]0 \end{bmatrix}$& $\begin{bmatrix}1\\[-1ex]0\\[-1ex]0 \end{bmatrix}$& 1& $\dfrac{ y_1(1 + y_2 + 2y_2y_3 + y_2y_3^2)}{1 + y_1 + y_1y_2 + 2y_1y_2y_3 + y_1y_2y_3^2}\vphantom{\begin{bmatrix}1\\[-.5ex]0\\[-.5ex]0 \end{bmatrix}}$\\
\hline
$(0,2)$ & $\begin{bmatrix}0\\[-1ex]1\\[-1ex]0 \end{bmatrix}$ & $\begin{bmatrix}0\\[-1ex]1\\[-1ex]0 \end{bmatrix}$ & 1& $\dfrac{y_2(1+y_3)^2}{1 + y_2 + 2y_2y_3 + y_2y_3^2}\vphantom{\begin{bmatrix}1\\[-.5ex]0\\[-.5ex]0 \end{bmatrix}}$\\
\hline
$(0,3)$ & $\begin{bmatrix}0\\[-1ex]0\\[-1ex]1 \end{bmatrix}$ & $\begin{bmatrix}0\\[-1ex]0\\[-1ex]1 \end{bmatrix}$ & 1 & $\dfrac{y_3}{1+y_3}\vphantom{\begin{bmatrix}1\\[-.5ex]0\\[-.5ex]0 \end{bmatrix}}$\\
\hline
$(1,1)$ & $\begin{bmatrix}-1\\[-1ex]1\\[-1ex]0 \end{bmatrix}$ & $\begin{bmatrix}-1\\[-1ex]1\\[-1ex]0 \end{bmatrix}$ & $1+y_1$&$\dfrac{1}{1+y_1}\vphantom{\begin{bmatrix}1\\[-.5ex]0\\[-.5ex]0 \end{bmatrix}}$\\
\hline
$(1,2)$ & $\begin{bmatrix}-1\\[-1ex]0\\[-1ex]2 \end{bmatrix}$ & $\begin{bmatrix}-1\\[-1ex]0\\[-1ex]1 \end{bmatrix}$ & $1+y_1+y_1y_2$ & $\dfrac{1+y_1}{1+y_1+y_1y_2}\vphantom{\begin{bmatrix}1\\[-.5ex]0\\[-.5ex]0 \end{bmatrix}}$\\
\hline
$(1,3)$ & $\begin{bmatrix}-1\\[-1ex]0\\[-1ex]1 \end{bmatrix}$ & $\begin{bmatrix}-2\\[-1ex]0\\[-1ex]1 \end{bmatrix}$ & $1+y_1+y_1y_2+y_1y_2y_3$ & $\dfrac{1+y_1+y_1y_2}{1+y_1+y_1y_2+y_1y_2y_3}\vphantom{\begin{bmatrix}1\\[-.5ex]0\\[-.5ex]0 \end{bmatrix}}$\\
\hline
$(2,1)$ & $\begin{bmatrix}0\\[-1ex]-1\\[-1ex]2 \end{bmatrix}$& $\begin{bmatrix}0\\[-1ex]-1\\[-1ex]1 \end{bmatrix}$&$1+y_2$ & $\dfrac{1+y_1+y_1y_2}{(1+y_1)(1+y_2)}\vphantom{\begin{bmatrix}1\\[-.5ex]0\\[-.5ex]0 \end{bmatrix}}$\\
\hline
$(2,2)$ & $\begin{bmatrix}-1\\[-1ex]-1\\[-1ex]2 \end{bmatrix}\vphantom{\begin{bmatrix}1\\[-.5ex]0\\[-.5ex]0 \end{bmatrix}}$ & $\begin{bmatrix}-1\\[-1ex]-1\\[-1ex]1 \end{bmatrix}$ &
 \begin{tabular}{@{}c@{}}$1 + y_1 + y_2+ 2y_1y_2 + y_1y_2^2 + $ \\ $ 2y_1y_2y_3 + 2y_1y_2^2y_3 + y_1y_2^2y_3^2$\end{tabular} & $\dfrac{(1+y_2)(1+y_1+y_1y_2+y_1y_2y_3)^2}{(1+y_1+y_1y_2)f_{(2,2)(\y)}}$
\\
\hline
$(2,3)$ & $\begin{bmatrix}0\\[-1ex]-1\\[-1ex]1 \end{bmatrix}$ & $\begin{bmatrix}0\\[-1ex]-2\\[-1ex]1 \end{bmatrix}$& $1+y_2+y_2y_3$ & $\dfrac{F_{(2,2)}(\y)}{(1+y_1+y_1y_2+y_1y_2y_3)(1+y_2+y_2y_3)}\vphantom{\begin{bmatrix}1\\[-.5ex]0\\[-.5ex]0 \end{bmatrix}}$\\
\hline
$(3,1)$ & $\begin{bmatrix}-1\\[-1ex]0\\[-1ex]0 \end{bmatrix}$& $\begin{bmatrix}-1\\[-1ex]0\\[-1ex]0 \end{bmatrix}$& $1 + y_1 + y_1y_2 + 2y_1y_2y_3 + y_1y_2y_3^2$ & $ \dfrac{F_{(2,2)}(\y)}{(1+y_2)f_{(3,1)}(\y)}\vphantom{\begin{bmatrix}1\\[-.5ex]0\\[-.5ex]0 \end{bmatrix}}$\\
\hline
$(3,2)$ & $\begin{bmatrix}0\\[-1ex]-1\\[-1ex]0 \end{bmatrix}$ & $\begin{bmatrix}0\\[-1ex]-1\\[-1ex]0 \end{bmatrix}$ & $1 + y_2 + 2y_2y_3 + y_2y_3^2$ & $\dfrac{(1+y_2+y_2y_3)^2 F_{(3,1)}(\y)}{F_{(2,2)}(\y)(1 + y_2 + 2y_2y_3 + y_2y_3^2)}\vphantom{\begin{bmatrix}1\\[-.5ex]0\\[-.5ex]0 \end{bmatrix}}$\\
\hline
$(3,3)$ & $\begin{bmatrix}0\\[-1ex]0\\[-1ex]-1 \end{bmatrix}$ & $\begin{bmatrix}0\\[-1ex]0\\[-1ex]-1 \end{bmatrix}$ & $1+y_3$ & $\dfrac{1 + y_2 + 2y_2y_3 + y_2y_3^2}{(1+y_3)(1+y_2+y_2y_3)}\vphantom{\begin{bmatrix}1\\[-.5ex]0\\[-.5ex]0 \end{bmatrix}}$
\\
\hline
\end{tabular}
\end{table}
\end{Example}

\section[Examples of MD as a configuration space]
{Examples of $\boldsymbol{\M_D}$ as a configuration space}

The space $\M_{A_{n-3}}$ can be identified with the configuration space of $n$ distinct points on $\P^1$. In~this section, we investigate similar descriptions of $\M_D$ in the cases $D=B_n$ and $D=C_n$. We~also consider the question of whether $\M_D$ is a hyperplane arrangement complement.

So far we have considered $\M_D$ as a complex algebraic variety. However, the equations \eqref{eq:I} make sense over the integers, and we may also consider $\M_D$ as a scheme over $\Z$. In~particular, in this section we will also consider $\M_D(\F_q)$, the set of $\F_q$-points of $\M_D$, where $\F_q$ is a finite field.

Throughout this section, we use the description of $\Pi$ from Section~3.2.

{\samepage\bigskip\noindent{\bf{\large 7.1.} Hyperplane arrangements.}
Let $H_1, H_2,\dots,H_r$ be (affine) hyperplanes in $\C^n$, with the assumption that the hyperplanes are defined over the integers. Let~$Z(\C) := \C^n - (H_1 \cup H_2 \cup \cdots \cup H_r)$. By our assumption $Z(\R)$ and $Z(\F_q)$ are also well-defined. The~following result is well-known \cite{OT, Sta}.

}
\begin{Theorem} \label{thm:hyp} \
\begin{enumerate}\itemsep=0pt
\item[$1.$] There exists a polynomial $\chi(t)$ so that $\chi(q) = \#Z(\F_q)$, where $q = p^m$ is a prime power with sufficiently large $p$.
\item[$2.$] The number of connected components $|\pi_0(Z(\R))|$ of the real hyperplane arrangement complement $Z(\R)$ is given by $(-1)^n \chi(-1)$.
\item[$3.$] The cohomology ring $H^*(Z(\C),\C)$ is generated by the classes of $\dlog f_i$, where $H_i =$ \mbox{$\{f_i =0\}$}, and we have $\sum_i \dim(H^i (Z(\C),\C)) \;t^i = (-t)^n \chi(-1/t)$. Thus the Euler characteristic of $Z(\C)$ is equal to $\chi(1)$.
\end{enumerate}
\end{Theorem}

\bigskip\noindent{\bf{\large 7.2.} Type $\boldsymbol{A_n}$.} 
Let $D = A_{n-3}$ with $n \geq 4$. Then $\Pi$ can be identified with the diagonals of a $n$-gon $P_n$. We~write $u_{ij}$ for the $u$-variable indexed by a diagonal $(i,j)$. Then the relations defining $\tM_{A_{n-3}}$ are given by \eqref{eq:Rij}.
These relations have appeared a number of times in the literature, for example see \cite{ABHY,Brown}. Let~$\M_{0,n}$ denote the configuration space of $n$ (distinct) points $(z_1,z_2,\dots,z_n)$ on $\P^1$. Then the identification
\begin{gather}\label{eq:uij}
u_{ij} = \frac{(z_i - z_{j+1})(z_{i+1}-z_j)}{(z_i - z_j)(z_{i+1}-z_{j+1})}
\end{gather}
of $u_{ij}$ with a cross ratio gives an isomorphism $\M_{A_{n-3}} \cong \M_{0,n}$. There is a well-studied Deligne--Knudsen--Mumford compactification $\bM_{0,n}$, and $\tM_{0,n}:=\tM_{A_{n-3}}$ is an affine variety that sits between $\M_{0,n}$ and $\bM_{0,n}$, that is, we have open inclusions $\M_{0,n} \subset \tM_{0,n} \subset \bM_{0,n}$.

Let $\Gr(2,n)$ denote the Grassmannian of 2-planes in $\C^n$. Let~$\oPi(2,n)\subset \Gr(2,n)$ denote the open subset where the adjacent cyclic minors $\Delta_{i,i+1}$ are non-vanishing. Then $\oPi(2,n)$ is a full rank cluster variety of type $A_{n-3}$, see \cite[Section~12.2]{CA2}. Let~$\oGr(2,n) \subset \oPi(2,n)$ be the open subset where \emph{all} Pl\"ucker coordinates $\Delta_{ij}$ are non-vanishing. This is the subset denoted $\Xo$ in Section~\ref{sec:quotient}. Then $\M_{0,n}$ can be identified with the quotient of $\oGr(2,n)$ by the diagonal torus $T$ sitting inside $\GL_n$ that acts on $\Gr(2,n)$. The~isomorphism $\M_{0,n} \cong \oGr(2,n)/T$ is an instance of Theorem~\ref{thm:fullrank} for $D = A_{n-3}$. In~the $\Gr(2,n)$ cluster algebra, we have the primitive exchange relation
\begin{gather*}
\Delta_{i,j} \Delta_{i+1, j+1}= \Delta_{i, j+1}\Delta_{i+1,j} + \Delta_{i,i+1} \Delta_{j,j+1},
\end{gather*}
where $\Delta_{i,i+1}$ and $\Delta_{j,j+1}$ are frozen variables. Thus \eqref{eq:uij}, or equivalently $u_{ij} = \frac{ \Delta_{i, j+1}\Delta_{i+1,j} }{\Delta_{i,j} \Delta_{i+1, j+1}}$, agrees with the formula for $u_\gamma$ in Theorem~\ref{thm:fullrank}.

The geometry and topology of $\M_{0,n}$ is very well-studied; see for example \cite{Brown}. We~recall some basic facts in the context of Theorem~\ref{thm:hyp}. Gauge-fixing $z_1$, $z_{n-1}$, $z_n$ to 0, 1, $\infty$, we have an identification
\begin{gather*}
\M_{0,n}(k) = \big\{(z_2,z_3,\dots,z_{n-2})\in k^{n-3} \,|\, \mbox{$z_i \neq z_j$ and $z_i \notin \{0,1\}$}\big\}
\end{gather*}
for $k$ a field. In~particular, $\M_{0,n}(k)$ is the complement in $k^{n-3}$ of the hyperplane arrangement with hyperplanes
\begin{gather*}
z_i - z_j = 0, \qquad z_i = 0, \qquad 1-z_i =0.
\end{gather*}
We may compute that
$ \#\M_{0,n}(\F_q) = (q-2)(q-3)\cdots (q-n+2).
$
By Theorem~\ref{thm:hyp}, the number of connected components of $\M_{0,n}(\R)$ is given by
\begin{gather*}
|\pi_0(\M_{0,n}(\R))| = |(-3)\cdot (-4) \cdots (-n+1)| = (n-1)!/2.
\end{gather*}

\bigskip\noindent{\bf{\large 7.3.} Type $\boldsymbol{B_n}$.}
By Theorem~\ref{thm:fullrank}, $\M_D$ can be identified with $\A\big(\tB\big)/T\big(\tB\big)$ for any full rank extended exchange matrix $\tB$ of type $D$. One such choice of $\tB$, and thus of $\A\big(\tB\big)$ is given in~\cite[Example 12.10]{CA2}. Let~$\C[X_{n+2}]$ be the ring generated by the Pl\"ucker coordinates of the Grassmannian $\Gr(2,n+2)$. Recall that $\Gamma$ is the two-element group whose non-trivial element maps $i \leftrightarrow \bar i$. Consider the following functions in $\C[X_{n+2}]$, labeled by $\Gamma$-orbits of sides and diagonals in the polygon $P_{2n+2}$ with vertices $\big\{1,2,\dots,n+1,\bar 1, \bar 2, \dots, \overline{n+1}\big\}$:
\begin{gather*}
\big[a, \bar a\big] \mapsto \Delta_{a \bar a} = \Delta_{a,n+2} \qquad (1 \leq a \leq n+1),
\\[.5ex]
\big\{[a,b],\big[\bar a, \bar b\big]\big\} \mapsto \Delta_{ab} \qquad (1 \leq a < b \leq n+1),
\\[.5ex]
\big\{\big[a,\bar b\big],\big[\bar a, b\big]\big\} \mapsto \Delta_{a \bar b} = \Delta_{a,n+2} \Delta_{b,n+2} - \Delta_{ab}\qquad
(1 \leq a < b \leq n+1).
\end{gather*}
Let $\oV_n$ be the space of $2 \times (n+2)$ matrices such that all the above functions are non-vanishing, and let $T_{n+1} \cong (\C^\times)^{n+1}$ act on $\oV_n$ by scaling the first $n+1$ columns. Then the action of $\SL_2 \times T_{n+1}$ on $\oV_n$ is free.
\begin{Proposition}
We have an isomorphism $\M_{B_n} \cong \SL_2 \backslash \oV_n/T_{n+1}$.
\end{Proposition}
Using the action of $\SL_2$ we can gauge-fix the last column of $M \in \oV_n$ to $[0,1]^T$, and using the action of $T_{n+1}$, we may gauge-fix the first entry of columns $1,2,\dots,n+1$ of $M$ to $1$.
Thus modulo the action of $\SL_2 \times T_{n+1}$, every point in $\oV_n$ can be written in the form\vspace{-.5ex}
\begin{gather*}
\begin{bmatrix}
1 & 1 & 1 & \cdots &1 & 0 \\
z_1 & z_2 & z_3 & \cdots & z_{n+1} & 1
\end{bmatrix}\!,
\end{gather*}
where $z_i \in \C$, and two such matrices with parameters $\z = (z_1,\dots,z_{n+1})$ and $\z'=(z'_1,\dots,z'_{n+1})$ are equivalent if $\z' -\z = c \one$, where $\one$ is the all $1$-s vector. We~may thus identify $\M_{B_n}$ with a subspace of $\C^{n+1}/\C= (z_1,\dots,z_{n+1})/\C\cdot \one$.
For these matrices, the cluster variables $\Delta_{a \bar a}$ are equal to $1$, and we have\vspace{-.5ex}
\begin{gather}
\Delta_{ab} = z_b-z_a \qquad (1 \leq a < b \leq n+1),\nonumber
\\
\Delta_{a \bar b} = 1 - z_b + z_a \qquad (1 \leq a < b \leq n+1).
\label{eq:Shi}
\end{gather}
We recognize the hyperplanes \eqref{eq:Shi} as the \emph{Shi arrangement} \cite{Shi}.
\begin{Proposition}
$\M_{B_n}$ is isomorphic to the complement in $\C^{n+1}/\C$ of the Shi arrangement.
\end{Proposition}
Among many well-known properties, we obtain the following as immediate consequences using Theorem~\ref{thm:hyp}: $(a)$ $\#\M_{B_n}(\F_q)= (q-n-1)^n $, $(b)$ the number of connected components $|\pi_0(\M_{B_n}(\R))| = (-1)^n(\#\M_{B_n}(\F_q)|_{q=-1})$ is equal to $(n+2)^n$, $(c)$ the cohomology $H^*(\M_{B_n},\C)$ is generated in degree one by $\dlog \Delta$, as $\Delta$ varies over the hyperplanes \eqref{eq:Shi}, and $(d)$ $|\chi(\M_{B_n})| \allowbreak = n^n$.

\bigskip\noindent{\bf{\large 7.4.} Type $\boldsymbol{C_n}$.} 
\def\Mat{{\rm Mat}}
\def\SO{{\operatorname{SO}}}
By applying Theorem~\ref{thm:fullrank} to \cite[Example 12.12]{CA2}, we obtain a description of~$\M_{C_n}$. Recall that~$\Gamma$ is the two-element group whose non-trivial element maps $i \leftrightarrow \bar i$. Consider the space $\Mat_{2,n+1}$ of $2 \times (n+1)$ matrices\vspace{-.5ex}
\begin{gather*}
\begin{bmatrix}
y_{11} & y_{12} & \cdots & y_{1,n+1} \\
y_{21} & y_{22} & \cdots & y_{2,n+1}
\end{bmatrix}
\end{gather*}
and the cluster variables\vspace{-.5ex}
\begin{gather}
\big\{[a,b],[\bar a, \bar b]\big\} \mapsto \Delta_{ab} = \frac{y_{1a}y_{2b}-y_{1b}y_{2a}}{2i} \qquad
(1 \leq a < b \leq n+1), \nonumber
\\
\big\{[a,\bar b],[\bar a, b]\big\} \mapsto \Delta_{a \bar b} = \frac{y_{1a}y_{2b}+y_{2a}y_{2b}}{2} \qquad
(1 \leq a \leq b \leq n+1)
\label{eq:Cncluster}
\end{gather}

\noindent
labeled by $\Gamma$-orbits of sides and diagonals in the polygon $P_{2n+2}$ with vertices $\{1,2,\dots,n+1,\bar 1, \bar 2, \dots, \overline{n+1}\}$. These functions generate the ring of invariant functions $\C[\Mat_{2,n+1}]^{S}$, isomorphic to a cluster algebra $\A\big(\tB\big)$ of type $C_n$ (when frozen variables are inverted). Here, $S = {\rm diag}(t,1/t)$ is the group of $2 \times 2$ diagonal matrices with determinant 1. The~cluster automorphism group is $T = (\C^\times)^{n+1} \times (\Z/2\Z)$, where $(\C^\times)^{n+1} \cong (\C^\times)^{n+1}/\langle (-1,-1,\dots,-1) \rangle$ acts on
$\Mat_{2,n+1}$ by rescaling columns (with the element $(-1,\dots,-1)$ acting trivially on $S \backslash \Mat_{2,n+1}$), and the non-trivial element of the group $(\Z/2\Z)$ acts by swapping the two rows.
\begin{Remark}
Note that in contrast to the $B_n$ case (\cite[Example 12.10]{CA2}), the $\tB$-matrix of \cite[Example 12.12]{CA2} is full rank but not really full rank. For example, for $n = 2$ we may choose an initial cluster so that we have
\begin{gather*}
\tB\quad \text{for type}\quad B_2= \begin{bmatrix} 0 & 1 \\
-2 & 0 \\
1 & 0 \\
1&-1\\
-1 & 0 \end{bmatrix} \qquad \text{and} \qquad
\tB\quad \text{for type}\quad C_2 = \begin{bmatrix} 0 & 2 \\
-1 & 0 \\
1 & 0 \\
1&-2\\
-1 & 0 \end{bmatrix}
\end{gather*}
respectively, where the rows are labelled by 13, $1\bar 1$, 12, $1 \bar 3$, 23. This explains why the cluster automorphism group $T$ in our discussion is disconnected.
\end{Remark}
On the locus where all cluster variables are non-vanishing, such $2 \times (n+1)$ matrices can be gauge-fixed, using $S$ and $(\C^\times)^{n+1} \subset T$ to the form:
\begin{gather}\label{eq:Zn}
\begin{bmatrix}
1 & 1 & \cdots & 1 &1 \\
z_1 & z_2 & \cdots & z_n & 1
\end{bmatrix}
\end{gather}
and the non-vanishing of the cluster variables is equivalent to the non-vanishing of the linear forms
\begin{gather}\label{eq:Chyper}
z_i - z_j, \qquad z_i + z_j, \qquad 1-z_i, \qquad 1+ z_i, \qquad z_i.
\end{gather}
Let $\Zo_n$ denote the space of matrices of the form \eqref{eq:Zn}, where the linear forms \eqref{eq:Chyper} are non-vanishing. There is still a free action of $\Z/2\Z$ on $\Zo_n$, acting by swapping the two rows, which induces $(z_1,\dots,z_n) \mapsto (1/z_1,\dots,1/z_n)$. By Theorem~\ref{thm:fullrank}, we obtain

\begin{Proposition}\label{prop:Cnconfig}
We have an isomorphism $\M_{C_n} \cong \Zo_n/(\Z/2\Z)$.
\end{Proposition}

This isomorphism is valid over the integers even though \eqref{eq:Cncluster} involves the scalars $1/2$ and~$1/2i$: this is because the scalars cancel in any $T$-invariant ratio of cluster variables.

\begin{Proposition}
We have
\begin{gather*}
\# \M_{C_n}(\F_q) = (q-n-1)(q-3)(q-5) \cdots (q-2n+1)
\end{gather*}
for ${\rm char}(q) > 2$ and
\begin{gather*}
|\pi_0( \M_{C_n}(\R))| = \frac{1}{2}\big( 2^n (n+1)! + 2^n n!\big) = 2^{n-1}(n+2) n!.
\end{gather*}
\end{Proposition}
Note that $|\pi_0( \M_{C_n}(\R))| = |(\# \M_{C_n}(\F_q))|_{q=-1}|$, agreeing with Theorem~\ref{thm:hyp}, even though we do not know whether $\M_{C_n}$ is isomorphic to a hyperplane arrangement complement.

\begin{proof}
For a field $k$, the points of $\M_{C_n}(k)$ in general come from $\Zo_n(\bar k)$, where $\bar k$ denotes the algebraic closure of $k$.

First, suppose that $k = \F_q$ with ${\rm char}(q) >2$. Let~$\bF_q$ denote the algebraic closure of $\F_q$. Then the Galois group is topologically generated by one generator $\sigma$, called the \emph{Frobenius auto\-morphism}. It~acts as the field automorphism $\sigma(x) = x^q$, and furthermore, we have $\sigma(x) = x$ if and only if $x \in \F_q$. The~map $\pi\colon \Zo_n(\bF_q) \to \M_{C_n}(\bF_q)$ commutes with the action of $\sigma$. Thus if~$\pi(z) = u \in \M_{C_n}(\F_q)$, we have
\begin{gather*}
u \in \M_{C_n}(\F_q) \Leftrightarrow \sigma(u) = u \Leftrightarrow \sigma(\pi(z)) = \pi(z) \Leftrightarrow \pi(\sigma(z)) = \pi(z),
\end{gather*}
and we have two possibilities: (1) $\sigma(z) = z$, or (2) $\sigma(z) = 1/z$. For case (1), we are just counting $\#\Zo_n(\F_q)$. Imposing the conditions \eqref{eq:Chyper}, we get
\begin{gather*}
\#\Zo_n(\F_q) = (q-3)(q-5) \cdots (q-2n-1).
\end{gather*}
For case (2), the equation $\sigma(z) = 1/z$ is equivalent to $z_i^{q+1} = 1$ for $i = 1,2,\dots,n$. There are $q+1$ solutions to the polynomial equation $x^{q+1}-1$ in $\bF_q$. Two of the solutions are $x = \pm 1$. The~other $q-1$ solutions lie in $\F_{q^2}$, since $x^{q+1}=1 \implies x^{q^2-1} = 1 \implies x^{q^2} =x$. Thus in case (2), we are counting $n$-tuples $(z_1,z_2,\dots,z_n) \in \F_{q^2}$ satisfying the condition $x^{q+1} = 1$ and the conditions imposed by \eqref{eq:Chyper}. The~conditions $z_i \neq \pm 1$ are automatically satisfied, so we get
\begin{gather*}
(q-1)(q-3)\cdots (q-2n+1).
\end{gather*}
In sum, taking into account the two-to-one covering $\Zo_n \to \M_{C_n}$, we have
\begin{align*}
\# \M_{C_n}(\F_q) &= \frac{1}{2} \left((2q-2n-2)(q-3)(q-5) \cdots (q-2n+1) \right) \\
& = (q-n-1)(q-3)(q-5) \cdots (q-2n+1).
\end{align*}
(Note that this point count is also the same as the hyperplane arrangement complement with hyperplanes $z_i+z_j$, $z_i - z_j$, $1+z_i$, $1-z_i$, though we do not know an explanation for this.)

Next let us consider $k = \R$, so $\bar k = \C$. In~this case, $\sigma$ is replaced by complex conjugation. So~the same argument says that we should consider the two cases (1) $\bar z = z$, or (2) $\bar z = 1/z$. For case (1), we are looking at $\Zo_n(\R)$ and by Theorem~\ref{thm:hyp} we get
\begin{gather*}
|\pi_0(\Zo_n(\R))| = 2^{n} (n+1)!
\end{gather*}
since $\#\Zo_n(\F_q) = (q-3)(q-5)\cdots (q-2n-1)$.

For case (2), we must have $z_i = \exp(i \theta_i)$ on the unit circle, where $\theta_i \in (\pi,\pi]$. The~conditions~\eqref{eq:Chyper} give
\begin{gather*}
\theta_i \neq \theta_j, \qquad \theta_i \neq 0, \qquad \theta_i \neq \pi, \qquad \theta_i + \theta_j \neq 0.
\end{gather*}
Thus the quantities $|\theta_1|,\dots, |\theta_n|$ lie in $(0,\pi)$ and satisfy $|\theta_i| \neq |\theta_j|$. There are $n!$ regions in $|\theta|$ space, and $2^n$ sign choices going from $|\theta_i|$ to $\theta_i$, giving
\begin{gather*}
\big|\pi_0\big(\Zo_n(S^1)\big)\big| = 2^n n!.
\end{gather*}
In sum, taking into account the two-to-one covering $\Zo_n \to \M_{C_n}$, we have
\begin{gather*}
|\pi_0( \M_{C_n}(\R))| = \frac{1}{2}\left( 2^n (n+1)! + 2^n n!\right) = 2^{n-1}(n+2) n!,
\end{gather*}
as claimed.
\end{proof}

The variety $\M_{C_n}$ can be identified with a subvariety of $\M_{A_{2n-1}}$. Sending $(z_1,\dots,z_n)$ to
\begin{gather}\label{eq:Cembed}
\begin{bmatrix}
1& 1 & \cdots & 1 &1 & 1 &\cdots &1&1
\\
z_1 & z_2 & \cdots & z_{n} & 1 & -z_1 & \cdots & -z_n & 1
\end{bmatrix}
\end{gather}
maps $\M_{C_n}$ into $\M_{A_{2n-1}}$. Note that the non-vanishing of \eqref{eq:Chyper} is equivalent to the non-vanishing of all minors in \eqref{eq:Cembed}, and that $(z_1,\dots,z_n)$ and $(1/z_1,\dots,1/z_n)$ represent the same point in~$\M_{A_{2n-1}}$.

\bigskip\noindent{\bf{\large 7.5.} Type $\boldsymbol{D_n}$.}
We do not know a simple description of $\M_D$ in this case. Indeed, the point counts we have obtained show that $\M_D$ cannot be a hyperplane arrangement complement. In~type $D_4$, numerical computations indicate that for $p \neq 2$, we have
\begin{gather}\label{eq:D4pointcount}
|\M_{D_4}(\F_p)|= \begin{cases}
206 - 231 p + 93 p^2 - 16 p^3 + p^4 & \mbox{if}\quad p = 2 \mod 3,
\\
208 - 231 p + 93 p^2 - 16 p^3 + p^4 & \mbox{if}\quad p = 1 \mod 3.
\end{cases}
\end{gather}
Substituting $p = -1$ in \eqref{eq:D4pointcount}, we get $547$. We~do not know for sure that $\M_{D_4}(\R)$ has $547$ connected components, but see Section~11.4.

For type $D_5$, numerical computations indicate that for $p \neq 2,3 $, we have
\begin{gather*}
|\M_{D_5}(\F_p)|=-2318 + 2644 p - 1156 p^2 + 244 p^3 - 25 p^4 + p^5+ (-36+5 p) \delta_3 (p)- \delta_4 (p),
\end{gather*}
where we define $\delta_3 (p)=0$ for $p = 2 \mod 3$ or $2$ for $p = 1 \mod 3$ and similarly $\delta_4 (p)=0$ for~$p=3 \mod 4$ or $2$ for $p=1 \mod 4$.

Substituting $p=-1$, we get $6388$. We~do not know whether $\M_{D_5}(\R)$ has $6388$ connected components, but see Section~11.4.

\bigskip\noindent{\bf{\large 7.6.} Type $\boldsymbol{G_2}$.} 
Numerical computations give
\begin{gather*}
|\M_{G_2}(\F_p)| = \begin{cases}
(p-4)^2 & \mbox{if}\quad p = 2 \mod 3,
\\
(p-4)^2 + 4& \mbox{if}\quad p = 1 \mod 3.
\end{cases}
\end{gather*}
Substituting $p = -1$, we get $25$, which we expect to be the number of connected components of~$\M_{G_2}(\R)$.

\section{Positive part}
In this section, we define the nonnegative subspace $\M_{D,\geq 0}$ of $\M_D(\R)$ and show that it is diffeomorphic to the generalized associahedron of $D^\vee$.

\bigskip\noindent{\large\bf 8.1.} 
We define the positive part $\M_{D,>0} \subset \M_D(\R)$ as the subspace
\begin{gather*}
\M_{D, >0} := \big\{x \in \M_D(\R) \,|\, u_\gamma(x) > 0 \text{ for all } \gamma \in \Pi\big\}
\end{gather*}
and the nonnegative part $\M_{D,\geq 0} \subset \tM_D(\R)$ by
\begin{gather*}
\M_{D,\geq 0} := \overline{\M_{D,>0}} \subset \tM_D(\R).
\end{gather*}
Intersecting with the stratification \eqref{eq:Mstrata}, we obtain
\begin{gather*}
\M_{D,\geq 0} = \bigsqcup_F \M_{F,>0}, \qquad \text{where} \quad \M_{F,>0} = \M_{D,>0} \cap \M_F(\R).
\end{gather*}

Let $P$ be an integer polytope and $\N(P)$ denote its normal fan. It~is well-known \cite[Chapter~4]{Ful} that the nonnegative part $X_{\N(P),\geq 0}$ of the projective toric variety $X_{\N(P)}$ is diffeomorphic to the polytope $P$. The~following result thus follows from Theorem~\ref{thm:toric}.
\begin{Theorem}\label{thm:pospart}
There is a face-preserving diffeomorphism between $\M_{D,\geq 0}$ and the generalized associahedron of $D^\vee$.
\end{Theorem}
The positive part $\M_{D,>0}$ is identical to the positive part $X_{\N(B^\vee),>0}$ of the ambient projective toric variety, and $\M_{D,>0}$ is equal to one of the connected components of the smooth manifold~$\M_D(\R)$.

\bigskip\noindent{\large\bf 8.2.}
The space $\M_{D}$ has a distinguished rational top-form $\Omega(\M_D)$, called the \emph{canonical form}, that can be described in a number of ways. Suppose $\tB$ is a full rank extended exchange mat\-rix of type~$D$. The~cluster algebra $\A\big(\tB\big)$ has a natural top-form $\Omega$ which in any cluster $(x_1,x_2,\dots,x_{n+m})$ can be written (up to sign):
\begin{gather*}
\Omega = \frac{{\rm d}x_1}{x_1} \wedge \cdots \wedge \frac{{\rm d}x_{n+m}}{x_{n+m}},
\end{gather*}
which is the natural top-form on the corresponding cluster torus $(\C^\times)^{n+m}$. The~cluster automorphism group $T\big(\tB\big)$ can be identified with a subgroup of $(\C^\times)^{n+m}$, and the quotient group $(\C^\times)^{n+m}/T\big(\tB\big)$ is again an algebraic torus $S = (\C^\times)^n$. The~inclusion $(\C^\times)^{n+m} \subset X\big(\tB\big)$ identifies $S$ birationally with $\M_D$ (note that neither $S$ nor $\M_D$ contains the other, but the two share a common dense open subset). The~torus $S$ has a natural top-form, and the canonical form~$\Omega(\M_D)$ is the image of this form under the birational isomorphism between $S$ and $\M_D$.

Another way to obtain the canonical form of $\M_D$ is via Theorem~\ref{thm:toric}: any toric variety $X_P$ has a canonical rational top-form $\Omega(X_P)$, which is simply the (extension of the) natural top-form of the dense algebraic torus in $X_P$. Restricting $\Omega\big(X_{\N(B^\vee)}\big)$ to $\M_D$ gives a top-form on $\M_D$ which equals $\Omega(\M_D)$.

The pair $\big(\tM_D, \M_{D,\geq 0}\big)$ is nearly a \emph{positive geometry} in the sense of \cite{ABL}, with $\Omega\big(\tM_D\big) = \Omega(\M_D)$ as canonical form: the residue of $\Omega(\M_D)$ along a divisor $\tM_D(F)$ is equal to the canonical form $\Omega\big(\tM_D(F)\big)$, which is simply the product of canonical forms corresponding to the factorization of Proposition~\ref{prop:divisorsplit}. This follows from the similar statement concerning $X_{\N(B^\vee)}$, proven in~\cite[Appendix G]{ABL}. However, $\tM_D$ is an affine variety rather than a projective variety, so it is not a positive geometry in the strict sense.

In the case $D=A_{n-3}$, we have $\M_D = \M_{0,n}$, and the canonical form can be written as
\begin{gather*}
\Omega\big((\M_{0,n})_{>0}\big) = \frac{{\rm d}z_2 \cdots {\rm d}z_{n-2}}{(z_2-z_1)(z_3-z_2)\cdots(z_1-z_n)},
\end{gather*}
where $(z_1,z_{n-1},z_n) = (0,1,\infty)$ as in Section~7.2, and denominator factors equal to $\infty$ are understood to be omitted.
This form is also called a cell-form in~\cite{BCS} and the condition that $\Omega((\M_{0,n})_{>0})$ only has poles along the boundary divisors of $\tM_{0,n}$ \big(and not elsewhere in $\bM_{0,n}$\big) is \cite[Proposition~2.7]{BCS}. Combining with the above discussion, we have
{\sloppy\begin{Proposition}
The pair $\big(\bM_{0,n},( \M_{0,n})_{\geq 0}\big)$ is a positive geometry with canonical form $\Omega\big((\M_{0,n})_{>0}\big) $.
\end{Proposition}

}

\section{Positive tropicalization}
\def\RR{{\mathcal R}}
\def\val{{\rm val}}

In this section, we consider the positive tropicalization of $\M_D$. We~use our results to resolve a~conjecture of Speyer and Williams \cite{SW} on positive tropicalizations of cluster algebras of finite type. We~refer the reader to \cite{SW} for background on positive tropicalizations.

Let $\RR = \bigcup_{n=1}^\infty \R\big(\big(t^{1/n}\big)\big) $ denote the field of Puiseux series over $\R$. We~define $\val\colon \RR \to \R \cup \{\infty\}$ by $\val(0) = \infty$ and $\val(x(t)) = r$ if the lowest term of $x(t)$ is equal to $\alpha t^r$, where $\alpha \in \R^\times$. We~define $\RR_{>0} \subset \RR$ to be the semifield consisting of Puiseux series $x(t)$ that are non-zero and such that coefficient of the lowest term is a positive real number.

A point $u(t) \in \M_D(\RR_{>0})$ is a collection $u(t) = \{u_\gamma(t) \in \RR_{>0}, \gamma \in \Pi\}$ of (positive) Puiseux series satisfying the relations from Definition~\ref{def:ID}. We~define the \emph{positive tropicalization} $\Trop_{>0} \M_D$ as the closure of valuations
\begin{gather*}
\Trop_{>0} \M_D := \overline{\big\{\val(u_\gamma(t))_{ \gamma \in \Pi}\in \R^\Pi \,|\, u(t) \in \M_D(\RR_{>0})\big\}} \subset \R^\Pi.
\end{gather*}
\begin{Lemma}
The subspace $\Trop_{>0} \M_D$ is a (not complete) polyhedral fan inside the linear space~$\R^\Pi$.
\end{Lemma}
\begin{proof}
We use the identification $\M_D \cong \Xo^\prin/T^\prin$ from Theorem~\ref{thm:principal}. According to Proposition~\ref{prop:hinvert}, there is an invertible monomial transformation between the functions $u_\gamma = f_\gamma$ and the set of functions $\{y_1,\dots,y_n\} \cup \{ F_\gamma(\y) \,|\, \gamma \in \Pi^+\}$. Since each $F_\gamma(\y)$ is a positive Laurent polynomial in the $y_i$-s, it follows that each $u_\gamma(\y)$ is a subtraction-free rational function in the $y_i$-s. Thus the map $(u_\gamma)_{\gamma \in \Pi} \mapsto (y_1,y_2,\dots,y_n)$ induces an isomorphism $\M_D(\RR_{>0}) \cong \RR_{>0}^n$. It~follows that we have a homeomorphism $\Trop_{>0}\M_D \cong \R^n$. The~embedding of $\Trop_{>0}\M_D$ in the (larger dimensional) linear space $\R^\Pi$ endows it with the structure of a polyhedral fan.
\end{proof}

\begin{Theorem}\label{thm:tropMD}
The fan $\Trop_{>0}\M_D$ is isomorphic to the cluster fan $\N(B^\vee)$.
\end{Theorem}
\begin{proof}
Under the isomorphism $\Trop_{>0}\M_D \cong \R^n$, the fan structure of $\Trop_{>0}\M_D$ gives a complete fan in $\R^n$ whose maximal cones are the common domains of linearity of the piecewise linear functions $\Trop(u_\gamma(\y))$. By Proposition~\ref{prop:hinvert}, we can equivalently take the common domains of linearity of the functions $\Trop(F_\gamma(\y))$, $\gamma \in \Pi^+$. It~is well-known (see for example \cite[Section~11.1]{ALS2}) that the resulting fan is the normal fan of the Newton polytope of the Laurent polynomial $\prod_{\gamma \in \Pi^+} F_\gamma(\y)$. By Theorem~\ref{thm:FNewton}, we deduce the isomorphism of fans $\Trop_{>0}\M_D \cong \N(B^\vee)$.
\end{proof}

Now let $\A\big(\tB\big)$ denote a cluster algebra of finite type, and $X\big(\tB\big)$ the corresponding cluster variety. Using the set of cluster variables $x_\gamma$, $\gamma \in \Pi$ and coefficient variables $x_{n+1},\dots,x_{n+m}$, we~have an embedding $X\big(\tB\big) \hookrightarrow \C^{|\Pi|+m}$. We~define $\Trop_{>0}X\big(\tB\big) \subset \R^{|\Pi|+m}$ as the closure of the image of $X\big(\tB\big)(\RR_{>0})$ under the map $\val\colon \RR_{>0} \to \R$. Note that the tropicalization $\Trop_{>0}X\big(\tB\big)$ depends only on the cluster algebra $\A = \A\big(\tB\big)$ and not on the choice of initial cluster.

The projection map $(x_\gamma,x_i) \mapsto (x_1,x_2,\dots,x_{n+m})$ from $\R^{|\Pi|+m}$ to $\R^{n+m}$ (onto the initial cluster variables) identifies $\Trop_{>0}X\big(\tB\big)$ with a complete polyhedral fan in $\R^{n+m}$.

\begin{Proposition}\label{prop:modL}
Suppose that $\tB$ is of full rank and has finite type $D$. Then $\Trop_{>0} X\big(\tB\big)$ modulo its lineality space $L$ is isomorphic to $\Trop_{>0}\M_D$.
\end{Proposition}
\begin{proof}
The translation action of the lineality space $L$ on $\R^{n+m}$ is simply the tropicalization of the action of the automorphism torus $T\big(\tB\big)$ \eqref{eq:auto}. Thus $\tB$ is of full rank if and only if~$L$ has dimension $m$ if and only if $T\big(\tB\big)$ is of dimension $m$. In~particular, $\Trop_{>0}X\big(\tB\big)/L$ is a~polyhedral fan of dimension $n$.

Via the isomorphism of Theorem~\ref{thm:fullrank}, each $u_\gamma = f_\gamma \in \C[\M_D]$ can be identified with a~monomial in the cluster and coefficient variables $x_\gamma$, $\gamma \in \Pi$ and $x_{n+1},\dots,x_{n+m}$. We~thus have a~linear projection map $p\colon \R^{|\Pi|+m} \to \R^{|\Pi|}$ (the tropicalization of the rational map sending $x$-s to $u$-s) mapping $\Trop_{>0} X\big(\tB\big)$ surjectively to $\Trop_{>0} \M_D$. The~fibers of $p$ are exactly the orbits of~the lineality space $L$ acting on $\Trop_{>0} X\big(\tB\big)$. It~follows that the fans $\Trop_{>0}X\big(\tB\big)/L$ and $\Trop_{>0} \M_D$ are isomorphic.
\end{proof}

Noting that the fans $\N\big(B^\vee\big)$ and $\N(B)$ are combinatorially isomorphic, we deduce from Theorem~\ref{thm:tropMD} and Proposition~\ref{prop:modL} the following conjecture of Speyer and Williams \cite[Conjecture~8.1]{SW}. The~principal coefficient case was established in~\cite{JLS} and we thank Christian Stump for drawing our attention to this work.
\begin{Corollary}
Suppose that $\tB$ is of full rank and has finite type. Then $\Trop_{>0}X\big(\tB\big)/L$ is combinatorially isomorphic to the complete fan $\N(B)$.
\end{Corollary}

\section[Extended and local u-equations]{Extended and local $\boldsymbol u$-equations}

In this section, we first study two additional sets of equations satisfied by $u$-variables, the {\it exten\-ded $u$-equations} and the {\it local $u$-equations}. In~other words, we give some further distinguished elements in the ideal $I_D$.

In type $A$, the extended $u$-equations were used by Brown \cite{Brown} to define what we call $\tM_{A_{n-3}}$; we see here that they can be interpreted as arising from all the exchange relations of the cluster algebra, rather than just the primitive exchange relations.

\bigskip\noindent{\bf{\large 10.1.} Extended $\boldsymbol u$-equations.} 
An \emph{extended $u$-equation} is an equation which holds in $\C[\M_D]$ of the form
\begin{gather}\label{eq:extu}
\prod_\gamma u_\gamma^{\alpha_\gamma} + \prod_\gamma u_\gamma^{\beta_\gamma} =1,
\end{gather}
where $\alpha_\gamma$, $\beta_\gamma$ are nonnegative integer parameters. All the primitive $u$-equations $R_\gamma$ in Definition~\ref{def:ID} are examples. Corollary~\ref{cor:ext} gives a class of extended $u$-equations for each $D$. It~would be desirable to have a uniform (instead of case-by-case) description of the extended $u$-equations coming from Corollary~\ref{cor:ext} similar to the description of the relations $R_\gamma$. This would follow from a solution to the following problem, which we believe is open. (Recall that the primitive exchange relations are described in Proposition~\ref{eq:exchange}.)

\begin{problem}
Give a uniform, root-system theoretic, description of {\it all} the exchange relations in a cluster algebra of finite type with universal coefficients.
\end{problem}

We now present explicitly extended $u$-equations for all the classical types $A$, $B$, $C$, $D$, and also for type $G_2$. In~types $A$ and $D$, the only extended $u$-equations we know come from Corollary~\ref{cor:ext}, and we conjecture that in simply-laced types these are the only ones. In~types $B$ and $C$, we~find more extended $u$-equations than those from Corollary~\ref{cor:ext}. Indeed, any extended $u$-equation for type $A_{2n-3}$ (resp.\ $D_n$) gives one for type $C_{n-1}$ (resp.\ $B_{n-1}$), but not all of these come from Corollary~\ref{cor:ext}. We~conjecture that all extended $u$-equations in multiply-laced type come from folding.

A similar analysis of extended $u$-equations for the types $E$ and $F$ can be found by a lengthy, but finite computation which we do not present here.

In the following, we will use the indexing of $\Pi$ from Section~3.2. For two disjoint subsets $I$ and $J$, define
\begin{gather}\label{eq:UIJ}
U_{I, J}:=\prod_{i \in I, j\in J} u_{i,j}.
\end{gather}
Note that $U_{I,J}$ and $U_{J,I}$ are not necessarily equal.

\medskip\noindent{\bf 10.1.1. Type $\boldsymbol A$.} 
Let $\{1,2,\dots, n\} = A \sqcup B \sqcup C \sqcup D$ be a decomposition of $[n]$ into cyclic intervals.
Then we have the extended $u$-equation
\begin{gather}\label{eq:exA}
U_{A,C}+U_{B,D}=1.
\end{gather}
Each equation depends on the choice of four cyclically ordered points $a$, $b$, $c$, $d$: $A=\{a{+}1, \dots, b\}$, $B=\{b{+}1, \dots, c\}$, $C=\{c{+}1, \dots, d\}$ and $D=\{d{+}1, \dots, a\}$, and thus there are $\binom{n}{4}$ equations in total, in bijection with the exchange relations of the type~$A$ cluster algebra. These equations arise from Corollary~\ref{cor:ext}. See \cite[Proposition~7.2]{YZ} for the exchange relations of the universal coefficient cluster algebra of type~$A$.

For $A=\{i\}$, $C=\{j\}$, the equation \eqref{eq:exA} becomes the primitive $u$-equations, $R_{ij}=0$.
As~discussed in~\cite{AHLT,Brown}, it is natural to interpret these $U$'s as cross-ratios of $n$ points on $\P^1$: denote $U_{A,C}=[a,b|c,d]$, and \eqref{eq:exA} becomes $[a, b | c, d]+[b, c | d, a]=1$. Together with the identity by definition $[a, b | c, e] [a, b | e, d]=[a, b | c, d]$, the equalities $[a, b | c, d]+[b, c | d, a]=1$ invariantly characterize cross-ratios of $n$ points, namely $[a, b | c, d]=\frac{\Delta_{a d}\Delta_{ b c}}{\Delta_{a c}\Delta_{b d}}$.

\medskip\noindent{\bf 10.1.2. Type $\boldsymbol C$.}
Extended $u$-equations for type $C_{n-1}$ arise via folding $A_{2n-3}$.
For a decomposition $\big\{1,2,\dots,n,\bar 1,\dots, \bar n\big\} \simeq [2n] = A\sqcup B\sqcup C\sqcup D$ into cyclic intervals, the image of \eqref{eq:exA} gives the extended $u$-equation for $C_{n{-}1}$ (which become the primitive ones for $A=\{i\}$, $C=\{j\}$). For example, if we choose $A=\{1\}$ and $C=\{n{+}1\}$, we have
\begin{gather*}
1-u_{[1 \bar 1]}=U_{\{2,\dots, n\}, \{\bar 2,\dots, \bar n\}}
=\prod_{i=2}^n u_{[i \bar i]} \prod_{2\leq i<j\leq n} u_{[i \bar j]}^2,
\end{gather*}
which is the primitive $u$-equation $R_{[1 \bar 1]} = 0$. For $C_2$, in addition to the $6$ primitive $u$-equations given in~\eqref{eq:C2}, we have $3$ more equations:
\begin{gather}\label{eq:exC2}
 u_{[1 \bar{1}]} u_{[1 \bar{2}]}+ u_{[2 \bar{3}]} u_{[3 \bar{3}]}=1,
 \end{gather}
and its cyclic rotations.

Let us count the number of extended $u$-equations for $C_{n-1}$ we have obtained. There are $\binom{2n}{4}$ equations \eqref{eq:exA} in type $A_{2n-3}$. Of those, $\binom{n}{2}$ are equal to its mirror image. For the remainder, both the equation and its mirror image map to the same equation in type $C_{n-1}$. Thus we have obtained $\frac{n(n{-}1)(2n^2-4n+3)}{6}$ extended $u$-equations for $C_{n-1}$. Note that this number is greater than the number of exchange relations of type $C_{n-1}$. For example, for $C_2$, there are $6$ exchange relations, but we have found $9$ extended $u$-equations.

\medskip\noindent{\bf 10.1.3. Type $\boldsymbol D$.} 
We now consider type $D_n$. We~use the notation \eqref{eq:UIJ} and also define (here $a\prec b$ means $a$ precedes $b$ in $I$)
\begin{gather*}
U_I:=\prod_{a\prec b \in I} u_{a,b} \prod_{i \in I} u_i,\qquad
\tilde{U}_I:=\prod_{a\prec b \in I} u_{a,b} \prod_{i \in I} u_{\tilde i}.
\end{gather*}
We now describe two types of extended $u$-equations. (These equalities were discovered emprically, but it should be relatively straightforward to prove them by induction.) First, similar to~\eqref{eq:exA}, for a cyclically ordered partition $A\sqcup B\sqcup C\sqcup D=\{1,2,\dots, n\}$, we have
\begin{gather}\label{eq:exD1}
U_{C,A}+U_{D,B} U_{A,B} U_{B,C} U_{B,D} U_B \tilde{U}_B=1.
\end{gather}
We allow $D$ to be empty here, in which case \eqref{eq:exD1} becomes $U_{C,A}+U_{A,B} U_{B,C} U_B \tilde{U}_B=1$. Second, for a cyclically ordered partition $A\sqcup B\sqcup C=\{1,2,\dots, n\}$, we have
\begin{gather}\label{eq:exD2}
U_A U_{A,B}+\tilde{U}_C U_{B,C}=1.
\end{gather}
Note that $B$ can be empty here, in which case we have $U_A+\tilde{U}_C=1$. In~the first and second type we have $4 {n \choose 4}+3 {n \choose 3}$ and $6 {n \choose 3}+2 {n \choose 2}$ equations, respectively, thus in total there are $\frac{n(n{-}1)(n^2+4n-6)}6$ extended $u$-equations. It~is not difficult to see that this is equal to the number of (unordered) pairs of exchangeable cluster variables. We~have $n^2$ cluster variables, $\frac{1}{4}(n^4+n^2-2n)$ pairs of compatible cluster variables, and $2\binom{n}{4}$ pairs of cluster variables where the compatibility degree is greater than one. The~remaining pairs of cluster variables are exchangeable.

It is straightforward to obtain the primitive $u$-equations for type $D_n$. In~\eqref{eq:exD1}, choosing $A=\{i\}$, $C=\{j\}$ (including the degenerate case with $D=\varnothing$, thus $j=i{-}1$), we have
\begin{gather*}
1-u_{j,i}=U_{D\cup \{i\}, B} U_{B, j\cup \{D\}} U_B \tilde {U}_B,
\end{gather*}
where $B=\{i{+}1, \dots, j{-}1\}$ and $D=\{j{+}1, \dots, i{-}1\}$. In~\eqref{eq:exD2}, take $B=\varnothing$, and choosing $A=i$ or $C=i$ we have
\begin{gather*}
1-u_i=\tilde{U}_{\{i{+}1, \dots, i{-}1\}},\qquad
1-{u}_{\tilde i}=U_{\{i{+}1, \dots, i{-}1\}}.
\end{gather*}
For example, for $n=4$, we have $52$ extended $u$-equations of the form \eqref{eq:exD1} and \eqref{eq:exD2}, including the $16$ primitive $u$-equations in \eqref{eq:D4}.

\medskip\noindent{\bf{10.1.4.} Type $\boldsymbol B$.}
Finally, we consider type $B_{n{-}1}$ by folding $D_n$. We~identify $u_i=u_{\tilde i}$, and the two types of equations become
\begin{gather*}
U_{C,A}+U_{D,B} U_{A,B} U_{B,C} U_{B,D} U_B^2=1, \qquad
U_A U_{A,B}+U_C U_{B,C}=1.
\end{gather*}
We have $\frac{n(n{-}1)(n^2+n-3)}{6}$ such extended $u$-equations in total: note that each equation \eqref{eq:exD1} has a distinct image in $B_{n{-}1}$ but two of the equations \eqref{eq:exD2} map to a single one in $B_{n{-}1}$. Again, we note that this number is greater than the number of exchange relations in type $B_{n{-}1}$.

The primitive $u$-equations can be recovered by setting $A=\{i\}$, $C=\{j\}$ (also $B=\varnothing$ in the second one). For example, for $n=4$, we obtain $34$ extended $u$-equations including the $12$ primitive ones of \eqref{eq:B3}. Let's write the additional $22$ equations as dihedral orbits of sizes 8, 8,~4,~2:
\begin{gather*}
8~{\rm eqs}\colon \quad u_{3,1} u_{4,1}+u_{1,2} u_{2,3} u_2^2=1,
\\
8~{\rm eqs}\colon \quad u_{1,2} u_1 u_2 u_{1,3} u_{2,3}+u_4 u_{3,4}=1,
\\
4~{\rm eqs}\colon \quad u_1 u_{1,2} u_{1,3}+ u_4 u_{2,4} u_{3,4}=1,
\\
2~{\rm eqs}\colon \quad u_{1,2} u_1 u_2+u_{3,4} u_3 u_4=1.
\end{gather*}
There are $22$ exchange relations in type $B_3$.

\medskip\noindent{\bf 10.1.5. Type $\boldsymbol G$.} 
As in Section~{3.2.5}, let us call the 8 $u$-variables $a_i$, $b_i$ for $i=1,\dots,4$ that can be thought of as labelling the edges of an octagon ($a_1,b_1,\dots,a_4,b_4)$. Folding $D_4$, we obtain 18 extended $u$-equations for~$G_2$, the primitive $u$-equations
\begin{gather*}
a_1 + a_2 a_3^2 a_4 b_2 b_3=1,\qquad b_1 + a_3^3 a_4^3 b_2 b_3^2 b_4=1, \qquad
{\rm + \, cyclic},
\end{gather*}
together with
\begin{gather*}
a_1 b_1 + a_3^2 a_4 b_2 b_3=1,\qquad a_2 b_1 + a_3 a_4^2 b_3 b_4=1,\qquad a_1 a_2 b_1 + a_3 a_4 b_3=1, \qquad
{\rm + \, cyclic}.
\end{gather*}

\bigskip\noindent{\bf{\large 10.2.} Local $\boldsymbol u$-equations.} 
The relations $R_\gamma$, and the extended $u$-equations are global in nature: they involve many $u_\gamma$ variables which are ``far away from each other". We~now describe a class of {\it local $u$-equations}. Using them one can show that all $u$-variables can be solved rationally in terms of $u$-variables of any acyclic seed; see also~\cite{AHLT}. This has implications for canonical forms; see \eqref{eq:canform}.

We first recall the $X$-coordinates for Fock and Goncharov's cluster $X$-variety. For an exchange relation $x x' = M + M'$, we have a cluster $X$-variable $X=M/M'$. Now, let us consider a primitive exchange relation in $X\big(\tB\big)$:
\begin{gather*}
x_{\tau \gamma} x_{\gamma} = M + M',
\end{gather*}
where $M'$ only involves frozen variables. We~recall that the isomorphism of Theorem~\ref{thm:fullrank} identifies the rational function $\frac{M}{x_{\tau \gamma} x_{\gamma}}$ with $u_\gamma$. Thus the cluster $X$-variable $X_\gamma:= M/M'$ can be identified with $u_\gamma/(1-u_\gamma) \in \C[\M_D]$ which is equal to a Laurent monomial in the $u_\omega$-s using the relations~$R_\gamma$. By \cite{FZY} or \cite[Proposition~3.9 or equation~(8.11)]{CA4} the variables $X_\gamma$ satisfy the relation
\begin{gather}\label{eq:localX}
X_{(t-1,j)} X_{(t,j)} = \prod_{i \to j} \big(1+X_{(t,i)}\big)^{-a_{ij}} \prod_{j \to i} \big(1+ X_{(t-1,i)}\big)^{-a_{ij}}
\end{gather}
or, equivalently, the variables $u_\gamma$ satisfy the relation
\begin{gather*}
\frac{u_{(t-1,j)} }{1-u_{(t-1,j)} } \frac{u_{(t,j)} }{1-u_{(t,j)} }= \prod_{i \to j} \big(1-u_{(t,i)}\big)^{a_{ij}} \prod_{j \to i} \big(1-u_{(t-1,i)}\big)^{a_{ij}}.
\end{gather*}

For the convenience of the reader, we give some examples of \eqref{eq:localX}, noting that the passage between $X$-variables and $u$-variables is completely compatible with folding.

\medskip\noindent{\bf 10.2.1. Type $\boldsymbol A$.}
For type $A_{n-3}$ we have $u$-variables $u_{i,j}$ for $1\leq i<j{-}1<n$, corresponding to the $n(n{-}3)/2$ diagonals of $n$-gon. We~have the same number of local $u$-equations, one for each ``skinny" quadrilateral:
\begin{gather}
\begin{tikzpicture}[{scale=0.7}]
\draw[thick] (0,0) circle (2);
\draw[thick] (0.894,1.788)
-- (0.894,-1.9) node[below] {$j$}
-- (-0.894,-1.9) node[below] {$j+1$}
-- (-0.894,1.85) node[above] {$i$}
-- (0.894,1.85) node[above] {$i+1$};
\draw[thick] (0.894,1.788) -- (-0.894,-1.788);
\draw[thick] (-0.894,1.788) -- (0.894,-1.788);
\end{tikzpicture}\nonumber
\\[1ex] \label{localA}
X_{i, j} X_{i{+}1, j{+}1}=(1+X_{i,j{+}1}) (1+ X_{i{+}1, j})
\end{gather}
or
\begin{gather*}
\frac{u_{i, j}}{1-u_{i,j}} \frac{u_{i{+}1, j{+}1}}{1-u_{i{+}1, j{+}1}}=\frac{1}{1-u_{i,j{+}1}} \frac{1}{1-u_{i{+}1, j}}.
\end{gather*}

\medskip\noindent{\bf 10.2.2. Type $\boldsymbol C$.}
By folding $A_{2n-3}$, we obtain local $u$-equations for type $C_{n-1}$. The~local $u$-equations take the form of~\eqref{localA}; for the special case of $j=\bar i$, it reads
\begin{gather*}
X_{[i, \bar i]} X_{[i{+}1, \overline{i+1}]}=\big(1+X_{[i, i{+}1]}\big)^2.
\end{gather*}

\medskip\noindent{\bf 10.2.3. Type $\boldsymbol D$.}
We use the notation for $u$-variables from Section~3.2.3. The~$n^2$ local $u$-equations for type $D_n$ read
\begin{gather*}
X_{i,j} X_{i{+}1, j{+}1}=(1+X_{i,j{+}1})(1+X_{i{+}1, j}),
\\
X_{i{-}1,i} X_{i, i{+}1}=(1+X_{i{-}1, i{+}1})(1+X_i) \big(1+X_{\tilde i}\big),
\\
X_i X_{\widetilde{i{+}1}}=1+X_{i, i{+}1},
\\
X_{\tilde i} X_{i{+}1}=1+X_{i,i{+}1}.
\end{gather*}

\medskip\noindent{\bf 10.2.4. Type $\boldsymbol B$.}
By folding type $D_n$, we obtain local $u$-equations for type $B_{n-1}$ (see Section~3.2.4). The~local $u$-equations are
\begin{gather*}
X_{[i,j]} X_{[i{+}1, j{+}1]}=\big(1+X_{[i,j{+}1]}\big)\big(1+X_{[i{+}1, j]}\big),
\\
X_{[i{-}1,\bar i]} X_{[i, \overline{i{+}1}]}=\big(1+X_{[i{-}1, \overline{i{+}1}]}\big)\big(1+X_{[i,\bar i]}\big)^2,
\\
X_{[i,\bar i]} X_{[i+1,\overline{i+1}]}=1+X_{[i, i{+}1]}
\end{gather*}
together with cyclic rotations.

\medskip\noindent{\bf 10.2.5. Type $\boldsymbol E$.}
Finally, for $E_n$ with $n=6,7,8$, we use the identification \eqref{eq:Pidef} to index~$u$ and~$X$ variables. When $\tB$ is bipartite, i.e., every vertex is either a source or a sink in the induced orientation of $D$, we have $r_i = h/2+1$ does not depend on $i$, where $h$ is the Coxeter number. The~Coxeter number is even in types $E_6$, $ E_7$, $E_8$, and $\Pi$ is identified with $m=7$, $10$, $16$ copies of~$I$ respectively. We~index the nodes of $E_n$ as shown below:
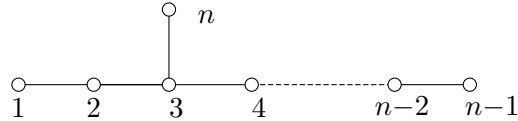
\begin{center} \begin{tikzpicture}[baseline={([yshift=-.5ex]current bounding box.center)},scale=1,every node/.style={scale=1}] \draw (-1,0) -- (-2,0); \draw (-2,0) -- (-3,0); \draw (-2,0) -- (-4,0);\draw (-2,0) -- (-2,1); \draw [dash pattern={on 2pt off 1pt } ](-1,0) -- (1,0); \draw (1,0) -- (2,0); \node[fill=white,circle,draw=black, inner sep=0pt,minimum size=5pt] at (-0.9,0) {}; \node[fill=white,circle,draw=black, inner sep=0pt,minimum size=5pt] at (-3,0) {}; \node[fill=white,circle,draw=black, inner sep=0pt,minimum size=5pt] at (-4,0) {}; \node[fill=white,circle,draw=black, inner sep=0pt,minimum size=5pt] at (-2,0) {}; \node[fill=white,circle,draw=black, inner sep=0pt,minimum size=5pt] at (-2,1) {}; \node[fill=white,circle,draw=black, inner sep=0pt,minimum size=5pt] at (1,0) {}; \node[fill=white,circle,draw=black, inner sep=0pt,minimum size=5pt] at (2,0) {}; \node at (-4,-0.3) {$1$}; \node at (-3, -0.3) {$2$}; \node at (-1.9,-0.3) {$3$}; \node at (-0.8,-0.3) {$4$}; \node at (1.1,-0.3) {${n{-}2}$}; \node at (2.3,-0.3) {$n{-}1$}; \node at (-1.5, 0.9) {$n$};
 \end{tikzpicture}
\end{center}
The local $u$-equations take the following form:
\begin{gather*}
X_{(t,i)}X_{(t+1,i)}=(1+X_{(t+1,i{+}1)})(1+X_{(t+1,i{-}1)})(1+X_{(t+1,n)})^{\delta_{i,3}}, \qquad \text{for~odd}\quad i<n,
\\
X_{(t,i)}X_{(t+1,i)}=(1+X_{(t,i)})(1+X_{(t,i{+}1)}), \qquad \text{for~even}\quad i<n,
\\
 X_{(t,n)}X_{(t+1,n)}=1+X_{(t,3)},
\end{gather*}
for $t=1,2,\dots, m$ in all these cases. We~have $n\times m$ equations in total.

\medskip\noindent{\bf 10.2.6. Types $\boldsymbol{F_4}$ and $\boldsymbol{G_2}$.} By folding $E_6$ we obtain local $u$-equations for $F_4$, and by folding~$D_4$ we obtain those for $G_2$.

\section{Connected components and sign patterns} \label{sec:conn}

The permutation group $S_n$ acts on the moduli space $\M_{0,n}(\R)$, permuting the $n$ points and permuting the connected components. The~presentation of $\M_{0,n}(\R)$ using $u_{ij}$ and the relations~\eqref{eq:Rij} depends on the choice of a dihedral ordering and the action of the symmetry group~$S_n$ is obscured. In~this section, we use the extended $u$-equations to investigate the connected components of $\M_D(\R)$, with the hope of uncovering an appropriate symmetry group for $\M_D$ in other types. While our results here are more speculative, we are able to construct new classes of $u$-equations.

A further motivation for studying connected components of $\M_D(\R)$ is the application to string amplitudes where it is important to consider canonical forms of different connected components of $\M_{0,n}(\R)$; see Section~12.2 for a brief discussion.

We also define an analogue of an oriented matroid for $\M_D$, called a ``consistent sign pattern", and it is conjectured that the number of consistent sign patterns is equal to the number of connected components.

\bigskip\noindent{\bf{\large 11.1.} Consistent sign patterns.}
A \emph{consistent sign pattern} for type $D$ is an element $(s_\gamma) \in \{+,-\}^\Pi$ such that for each extended $u$-equation \eqref{eq:extu} $\prod_\gamma u_\gamma^{\alpha_\gamma} + \prod_\gamma u_\gamma^{\beta_\gamma} =1$, we have that at least one of the signs $\prod_\gamma s_\gamma^{\alpha_\gamma}$ and $\prod_\gamma s_\gamma^{\beta_\gamma}$ is positive. In~other words, $(s_\gamma)$ are possible signs for some solution $(u_\gamma)$ of extended $u$-equations. We~could also call a consistent sign pattern a~``uniform oriented matroid" for the $u$-variables.

In~\cite{AHLT}, we made the following conjecture.

\begin{Conjecture} The number of connected components of $\M_D(\R)$ is given by the number of sign patterns of $u$-variables consistent with the extended $u$-equations for $D$.
\end{Conjecture}

\medskip\noindent{\bf{\large 11.2.} Type $\boldsymbol A$.} We consider $D = A_{n-3}$. The~space $\M_{0,n}(\R)$ has $(n{-}1)!/2$ connected components, corresponding to the dihedral orderings of $n$ points. In~the positive connected component $(\M_{0,n})_{> 0}$, all the cross ratios $u_{ij}$ are positive, indeed, we have $0 < u_{ij} < 1$. In~other connected components of $\M_{0,n}$, some of the $u_{ij}$-s are negative. We~find that the extended $u$-equations exclude those sign patterns for which both $[a, b | c, d]$ and $[b, c | d, a]$ are negative. Empirically, we find that precisely $(n{-}1)!/2$ consistent sign patterns are allowed by the extended $u$-equations~\eqref{eq:exA}, and this count agrees with the number of connected components of $\M_{0,n}(\R)$.

Let us now consider the problem of finding new $u$-variables for other components of $\M_{0,n}(\R)$, that is, we seek cross-ratios that are positive on that component. It~suffices to consider the orde\-ring that is obtained from the standard one by an adjacent transposition, e.g., $(1',2',3',\dots, n')=(2,1,3,\dots, n)$. Using the following identities for the cross ratio $[a, b | c, d]$ (see Section~10.1.1):
\begin{gather*}
[a, b| d, c]=\frac{1}{[a,b| c,d]},\qquad
[a, c| b,d]=-\frac{[a, b| c, d]}{[b,c | d,a]}\,
\end{gather*}
we find that the $u$-variables in this new ordering include
\begin{gather}
 u'_{1,3}=[n', 1'| 2', 3']=[n, 2 | 1, 3]=\frac{-[n, 1|2,3]}{1,2| 3, n]}=\frac{-[n,1|2,3]}{\prod_{i=3}^{n{-}1} [1,2 | i,i{+}1]}=\frac{-u_{1,3}}{u_{2,4} \cdots u_{2,n}}, \nonumber
 \\
 u'_{1, i}=[n', 1'| i'{-}1, i']=[n, 2| i-1, i]=[n,1| i{-}1,i][1,2| i{-}1,i]=u_{1,i}u_{2,i}, \nonumber
\\
 u'_{2, i}=[1', 2'| i'{-}1, i']=[2,1 | i{-}1, i]=\frac{1}{[1,2; i{-}1, i]}=\frac{1}{u_{2,i}}, \nonumber
\\
 u'_{3, i}=[2',3'| i'{-}1, i']=[1,3 |i{-}1, i]=[1,2 | i{-}1, i][2, 3 | i{-}1, i]=u_{2,i} u_{3,i},
\label{newu}
\end{gather}
and all other $u$'s are unchanged. These new $u'_{ij}$-s are positive in the connected component given by the ordering $(2,1,3,\dots, n)$, and furthermore they satisfy the extended $u$-equations (for this ordering). In~other words, the (invertible) signed monomial transformation \eqref{newu} sends the extended $u$-equations for $u_{ij}$ to a permutation of the extended $u$-equations for the $u'_{i,j}$, and thus exposing a hidden $S_n$-symmetry of these equations.

\bigskip\noindent{\bf{\large 11.3.} Type $\boldsymbol C$.}
From the analysis of Section~7.4, we know that $\M_{C_n}(\R)$ has $2^{n{-}1}n! (n+1)$ connected components. Computationally, we find that this agrees with the number of consistent sign patterns.

There are two types of components in ${\mathcal M}_{C_n}(\R)$ corresponding to two types of configurations of the $(2n{+}2)$-gon with labels $i$, $\bar{i}$ for $i=1,2, \dots, n{+}1$: ($A$) $2^{n{-}1} n!$ components for polygons with central symmetry, e.g., $1,2, \dots, n{+}1, \bar{1}, \bar{2}, \dots, \overline{n{+}1}$; ($B$) $2^{n{-}1} (n{+}1)!$ components for polygons with reflection symmetry avoiding vertices, e.g., the ordering $1, 2, \dots, n{+}1, \overline{n{+}1}, \dots, \bar{2}, \bar{1}$. Unlike type $A$, our investigations indicate that the compactification (arising from $u$-equations) of these two types of components have differing boundary combinatorics: for any component in ($A$), combinatorially it is a cyclohedron (the generalized associahedron of type $C$), while for any component in ($B$), combinatorially it is an associahedra. We~expect that this can be proven via a careful analysis of the extended $u$-equations, and here we illustrate it for the simplest example, $C_2=B_2$.

The extended $u$-equations are given by \eqref{eq:C2} and \eqref{eq:exC2}, plus cyclic rotations. The~positive part with all $u$-s positive corresponds to the ordering $1,2,3,\bar{1}, \bar{2}, \bar{3}$, which cuts out a hexagon. We~can see the other $3$ orderings in $(A)$ by making a signed monomial transformation of the $u$-variables. For example, for the ordering $2,1,3, \bar{2}, \bar{1}, \bar{3}$, we find that the $6$ new variables can be obtained by a monomial change of variables:
\begin{gather*}
\begin{split}
&u'_{[1 \bar{2}]}=\frac{1}{u_{[2 \bar{1}] }u_{[2 \bar{2}]}},\qquad
u'_{[2 \bar{3}]}=-\frac{u_{[1 3]}}{u_{[\bar{1} 2]} u_{[2 \bar{2}]} u_{[2 \bar{3}]}}, \qquad
u_{[3 \bar{1}]}=\frac{1}{u_{[2 \bar{2}]}u_{[2 \bar{3}]}},
\\
&u'_{[1 \bar{1}]}=u_{[2 \bar{2}]},\qquad
u'_{[2 \bar{2}]}=u_{[1 \bar{1}]} u^2_{[2 \bar{1}]} u_{[2 \bar{2}]},\qquad
u_{[3 \bar{3}]}=u_{[2 \bar{2}]} u^2_{[2 \bar{3}]}u_{[3 \bar{3}]}.
\end{split}
\end{gather*}
It is straightforward to check that these $6$ new variables satisfy identical extended $u$-equations for the ordering $2,1, 3, \bar{2}, \bar{1}, \bar{3}$. More generally, under this kind of transformation, similar to the type $A_{n-3}$ case, we find that for any ordering in $(A)$ the new variables satisfy identical extended $u$-equations, and the corresponding component is combinatorially a cyclohedron.

Let us now consider orderings in $(B)$. For example, for the ordering $1, 2, 3, \bar{3}, \bar{2}, \bar{1}$, we find that the $6$ new variables are given in terms of the old ones by
\begin{gather*}
u'_{[1 \bar{3}]}=-\frac{u_{[1 \bar{1}]}}{u_{[2 \bar{2}]} u^2_{[2 \bar{3}]} u_{[3 \bar{3}]}},\qquad u'_{[\bar{1} 2]}=\frac{1}{u_{[2 \bar{2}]}},\qquad
u'_{[\bar{2} 3]}=\frac{1}{u_{[3 \bar{3}]}},
\\
u'_{[1 \bar{2}]}=u_{[13]} u_{[2 \bar{3}]} u_{[3 \bar{3}]},\qquad
u'_{[2 \bar{3}]}=u_{[2 \bar{1}]} u_{[2 \bar{2}]} u_{[2 \bar{3}]},\qquad
u'_{[2 \bar{2}]}=\frac{1}{u_{[2 \bar{3}]}}.
\end{gather*}
The $9$ extended $u$-equations become the following ones for the $6$ new variables:
\begin{gather*}
u'_{[1 \bar{2}]} u'_{[1 \bar{3}]}+u'_{[2 \bar{2}]} u'_{[\bar{1} 2]}=1,\qquad
u'_{[1 \bar{3}]} u'_{[2 \bar{3}]}+ u'_{[2 \bar{2}]} u'_{[\bar{2} 3]}=1,\qquad
u'_{[2 \bar{2}]} (u'_{[1 \bar{2}]}+ u'_{[2 \bar{3}]}) =1,
\\[.5ex]
u'_{[2 \bar{3}]}+ u'_{[1\bar{2}]} u'_{[2 \bar{2}]} u'_{[\bar{2} 3]}=1,\qquad
u'_{[1 \bar{3}]}+ u'^2_{[2 \bar{2}]} u'_{[\bar{1} 2]} u'_{[ \bar{2} 3]}=1,\qquad
u'_{[1 \bar{2}]}+u'_{[2 \bar{2}]} u'_{[2 \bar{3}]} u'_{[\bar{1} 2]}=1,
\\[.5ex]
u'_{[\bar{1} 2]}+u'^2_{[1 \bar{2}]} u'_{[1 \bar{3}]}=1, \qquad
u'_{[2 \bar{3}]}+ u'_{[1 \bar{3}]} u'^2_{[2 \bar{3}]}=1, \qquad
u'_{[2 \bar{2}]} (1+ u'_{[1 \bar{2}]} u'_{[1 \bar{3}]} u'_{[2 \bar{3}]})=1.
\end{gather*}
Note that $u'_{[2 \bar{2}]}$ is special: from the third and the last equations, it is easy to see that $u'_{[ 2 \bar{2}]}$ cannot take the value 0, and thus $u'_{[2 \bar{2}]}=0$ does not correspond to a facet. The~other $5$ variables do correspond to facets, and from the equations we see that requiring all $u' \geq 0$ cuts out a pentagon instead. In~general, we expect that for any ordering in ($B$), such a transformation give equations of this type, where certain $u'$ cannot reach zero, and (an appropriate closure) of the component has the combinatorics of a (type $A$) associahedron.

\bigskip\noindent{\bf{\large 11.4.} Types $\boldsymbol{D_4}$ and $\boldsymbol{D_5}$.} 
We were unable to determine the number of connected components of $\M_{D_n}(\R)$. However, we can obtain a consistency check by comparing the number of consistent sign patterns with the point count over $\F_q$.

Recall from the point count \eqref{eq:D4pointcount} in type $D_4$, we predicted that $\M_D(\R)$ has $547$ connected components. By a direct computation, we checked that this is equal to the number of consistent sign patterns of $u$-variables with respect to the extended $u$-equations in Section~10.1.3. Similarly, the prediction of $6388$ in the $D_5$ case computationally agrees with the number of consistent sign patterns of $u$-variables.

\bigskip\noindent{\bf{\large 11.5.} Type $\boldsymbol{G_2}$.}
As we have discussed in Section~7.6, we expect that there are 25 different connected components for $G_2$. Furthermore, the point count $\M_{G_2}(\F_q)$ is not polynomial, and thus $\M_{G_2}$ is not a hyperplane arrangement complement. Computationally, we find that there are 25 consistent sign patterns for the 18 extended $u$-equations from Section~10.1.5.

We now investigate the connected components of $\M_{G_2}(\R)$, and note some new features. The~positive component has $0 < a_i, b_i < 1$ as usual. But suppose $b_1$ is made negative; to wit we put $b_1 = -b^\prime_1$ with $b^\prime_1 > 0$. As in our discussion for types $A$ and $C$, we rearrange all the extended $u$-equations to put them again in the form of (monomial$_1$) + (monomial$_2$) = 1. Quite nicely, the 36 exponent vectors of these monomials lie in an 8-dimensional cone, that is, all 36 vectors can be expressed as a positive linear combination of eight of them. The~8 generators can be associated with the new variables
\begin{gather*}
x_1 = b_4^{-1},\qquad x_2=a_1a_3 a_4^2 b_3 b_4,\qquad x_3 = b^\prime_1 a_3^{-3} a_4^{-3} b_2^{-1} b_3^{-2} b_4^{-1},
\\
x_4=a_2 a_3^2 a_4 b_2 b_3,\qquad x_5 =b_2^{-1},
\\
y_1=b_3, \quad
 y_2= a_4^{-1} b_3^{-1}, \quad y_3=a_3^{-1} b_3^{-1}
\end{gather*}
and using these variables, we get 18 equations (monomial)$_1$ + (monomial)$_2$ = 1, written in terms of the $x$'s and $y$'s, all with positive exponents. Eight of these are ``primitive" $u$-equations
\begin{gather*}
x_1 + x_3 x_4^2=1, \qquad
x_2 + x_4 x_5 y_1^2 y_2 y_3^2=1, \qquad
x_3 + x_1 x_5 y_1^4 y_2^3 y_3^2=1,
\\
x_4 + x_1 x_2 y_1^2 y_2^2 y_3=1, \qquad
x_5 + x_2^3 x_3=1,
\\[.5ex]
y_1 + y_1^4 x_2^3 x_3^2 x_4^3 y_2^3 y_3^2=1, \qquad
y_2 + y_2^2 x_2 x_3 x_4^2 y_1^2 y_3^2=1, \qquad
y_3 + y_3^2 x_2^2 x_3 x_4 y_1^2 y_2^2=1.
\end{gather*}
 As usual these equations tell us that if the $x,y \geq 0$, then we also have $x, y \leq 1$. But note an interesting feature of the equations for the $y_i$ that we also saw for $C_n$: both monomials contain a factor of $y_i$, and therefore we cannot set any of the $y$'s to zero. Thus the only boundaries of this connected component are associated with the $x_{i} \to 0$ for $i =1,\dots,5$; and we have found an unusual binary realization (in the sense of \cite{AHLT}) of a pentagon.

Suppose instead we now set $a_1$ to be negative, that is, $a_1 = -a^\prime_1$ with $a'_1 >0$. Repeating the same analysis something more interesting happens. The~set of 36 exponent vectors associated with the monomials of the extended $u$-equations span a cone with 12 generators. The~monomials associated with these 12 generators are
\begin{gather*}
x_1 = a_4^{-1}, \qquad\!
x_2=a_2^3 a_3^3 b_1 b_2^2 b_3, \qquad
x_3=b_3^{-1},\qquad \!
x_4=b_2^{-1},\qquad \!
x_5=a^\prime_1 a_2^{-1}a_3^{-2} a_4^{-1} b_2^{-1} b_3^{-1},
\\
x_6 =a_3^3 a_4^3 b_2 b_3^2 b_4, \qquad
x_7 = a^\prime_1 a_4^2 b_3 b_4 a_2^{-1} b_2^{-1}, \qquad
x_8=a_3^{-2} a_4^{-1} b_2^{-1} b_3^{-1},
\\
x_{9} = a^\prime_1 a_2^2 b_1 b_2 a_4^{-1} b_3^{-1}, \qquad
x_{10}=a_3 b_2, \qquad
x_{11} = a_3 b_3, \qquad
x_{12} = a_2^{-1} a_3^{-2} b_2^{-1} b_3^{-1}
\end{gather*}
and all the extended $u$-equations can be written (albeit not uniquely) as a sum of two monomials in the $x$'s with positive coefficients. Of course, twelve of the equations are of the form $x_i + \text{(monomial)} = 1$, so all the $x$'s are restricted to lie between 0 and 1. But the twelve exponent vectors in eight dimensions satisfy four relations which can be expressed in many equivalent ways, for instance
\begin{gather*}
x_2^2 x_3 x_5^2 x_8=x_1 x_4 x_9^2, \qquad x_6^2 x_4 x_5^2 x_8 = x_1 x_3 x_7^2, \qquad x_{10}^2 x_4 x_8 = x_1 x_3, \qquad x_{11}^2 x_3 x_8 = x_1 x_4.
\end{gather*}
So in the language of these $x$-variables, we have monomial equations with positive exponents, but also satisfying non-trivial constraint relations. Further study of this particular region reveals it to be a binary realization of a hexagon. It~is natural to conjecture that the phenomenon we have seen in this $G_2$ example is generic--when studying different connected components using $u$-equations, the set of exponent vectors will be a pointed cone, and that the equations will always force all variables to lie between $(0,1)$. But the variables will satisfy additional monomial constraints.

\section{Outlook}

We close with a few comments on open directions for future exploration.

\bigskip\noindent{\bf{\large 12.1.} Understanding real components of cluster configuration space.}
An open question immediately suggested by our investigations (see Section~\ref{sec:conn}) is understanding all the connected components of the real points of cluster configuration spaces. For $A_{n-3}$, there is a~beautiful picture, where the complete space is tiled by ``binary associahedra" corresponding to all $(n{-}1)!/2$ orderings of $n$ points on the projective line, and there is a similar complete picture of all the orderings for $C_n$ by folding. These examples are especially easy to understand since we have a ``linear model" for $\M_{0,n}$ in terms of a hyperplane arrangement. The~connected components can also be easily understood in these examples, directly studying the space of solutions of the $u$-equations with different sign patterns. In~general, $\M_D$ is not a hyperplane arrangement complement and it would thus be interesting to systematically study the question of connected components directly from the $u$-equations defining the space, as we have done in some examples in Section~\ref{sec:conn}.

It is natural to conjecture that some or all of the other real components of cluster configurations spaces (suitably compactified using the $u$-equations) are also positive geometries, and it would be interesting to determine their canonical forms.

In connection with determining canonical forms for general components, we state here without proof, a simple expression for the canonical forms of the positive component we have studied above, not in terms of cluster variables, but directly in terms of $u$-variables. Recall that the $u_\gamma$ are in bijection with all the cluster variables. Consider any {\it acyclic} cluster $(x_{\gamma_1}, \dots, x_{\gamma_n})$. Then, the canonical form is simply given by taking the wedge product
\begin{gather}\label{eq:canform}
\Omega = \bigwedge_{i=1}^n \frac{{\rm d} u_{\gamma_i}}{u_{\gamma_i} (1 - u_{\gamma_i})}.
\end{gather}
As we noted in Section~10.2, acyclic seeds have the following feature: all the $u_\gamma$ variables can be expressed rationally in terms of those in the initial seed. This idea can be extended to give canonical forms for other connected components of cluster configuration spaces. As a simple example, let us consider the description of the component discussed in Section~\ref{sec:conn} for the $G_2$ case, where $b_1 = - b^\prime_1 < 0$. We~can readily check that all of the $(x,y)$ variables can be rationally solved for in terms of either $(x_2,x_3)$ or in terms of $(x_3,x_4)$. The~canonical form is then given as
\begin{gather*}
\Omega = {\rm d log} \frac{x_2}{1-x_2} {\rm d log} \frac{x_3}{1-x_3} = {\rm d log} \frac{x_3}{1-x_3} {\rm d log} \frac{x_4}{1-x_4}.
\end{gather*}

\bigskip\noindent{\bf{\large 12.2.} Open and closed cluster string amplitudes.} 
The \emph{stringy canonical forms} of \cite{AHL} can be applied to the cluster configuration space $\M_D$, and we obtain the \emph{cluster string amplitude}.
For a cluster algebra $\A\big(\tB\big)$ of full rank and of type $D$, the cluster string integral is defined to~be
\begin{gather}\label{eq:clusterstring}
\I_D(s):= \int_{\M_{D,>0}} \Omega(\M_{D,>0}) \prod_{\gamma \in \Pi} x_\gamma^{\alpha' s_\gamma}\prod_{i=n+1}^{n+m} x_i^{\alpha' s_i},
\end{gather}
where $\{s_\gamma \,|\, \gamma \in \Pi\} \cup \{s_i \,|\, i \in [n+1,n+m]\}$
are parameters chosen so that the product $\prod_{\gamma \in \Pi} x_\gamma^{\alpha' s_\gamma}$
is $T$-invariant, and thus descends to a function on $\M_{D,>0}$. Choosing $\tB = \tB^\prin$, we~may use Theorem~\ref{thm:principal} to rewrite \eqref{eq:clusterstring} as
\begin{gather}\label{eq:Fintegral}
\I_D(\X,\{c\}) = \int_{\R_{>0}^n} \prod_i \frac{{\rm d}y_i}{y_i} y_i^{\alpha' X_i} \prod_{\gamma \in \Pi^+} F_\gamma(\y)^{-\alpha' c_\gamma},
\end{gather}
where $(\X,\{c\})$ are related linearly to $(s_i,s_\gamma)$. By \cite[Claim 2]{AHL}, \eqref{eq:Fintegral} converges when the point $\X$ belongs to the generalized associahedron $P(c)= \sum_{\gamma \in \Pi^+} c_\gamma P_\gamma$, where $P_\gamma$ is the Newton polytope of $F_\gamma(\y)$. By \cite[Claim 2]{AHL} the leading order of $\I_D(\X,\{c\})$ is the canonical function $\Omega(P(c))$ of~$P(c)$ evaluated at $X$:
\begin{gather*}
\lim_{\alpha' \to 0} (\alpha')^n \I_D(\X,\{c\}) = \underline{\Omega}(P(c))(\X).
\end{gather*}
(We refer the reader to \cite{ABL, AHL} for background on canonical functions and canonical forms.)
In~particular, the poles of $\I_D(\X,\{c\})$ as $\alpha' \to 0$, all of which are simple, correspond bijectively to the facets of the generalized associahedron of $D^\vee$. By \cite[Section~9]{AHL} and Theorem~\ref{thm:pair}, we may also rewrite
\begin{gather} \label{eq:IDu}
\I_D({\bf U}) = \int_{\M_{D,>0}} \Omega(\M_{D,>0}) \prod_{\gamma \in \Pi} u_\gamma^{\alpha'U_\gamma}
\end{gather}
and the convergence condition is the very simple condition $U_\gamma >0$. By Proposition~\ref{prop:hinvert}, the $u_\gamma$ and $\{y_i, F_\gamma\}$ are related by an invertible monomial transformation, and thus $\{U_\gamma \,|\, \gamma \in \Pi\}$ and $\{X_1,\dots,X_n\} \cup \{c_\gamma \,|\, \gamma \in \Pi^+\}$ are related by an invertible linear transformation. (The matrix of~this linear transformation has entries given by the integers $\Trop(F_\gamma(\y))\big({-}\g^\vee_\omega\big)$ that appeared in~Section~\ref{sec:F}.) We see from~\eqref{eq:IDu} that the $u$-variables $u_\gamma$ are reverse-engineered from the cluster string integral: they are those monomials in cluster variables making the domain of convergence explicit.

As explained in~\cite[Section~7]{AHL}, for generic exponents $X$, we expect that varying the cycle of integration (to something other than the cycle $\M_{D,>0}$) will span a space of integral functions of dimension equal to the absolute value of the Euler characteristic $|\chi(\M_D(\C))|$. Indeed, it is especially natural to integrate over any of the other real connected components of the cluster configuration space, directly generalizing the basis of all (tree-level) open string amplitudes associated with type $A$.

We can also define the analog of ``closed string" cluster amplitudes. The~simplest object we can define (as in~\cite{AHL}) is the ``mod square" of the open string integral
\begin{gather*}
\I_{D}^{\rm{closed}}\big(\big\{U,\bar{U}\big\}\big)=
\int_{\M_D(\C)} \Omega(\M_{D,>0}) \prod_{\gamma \in \Pi} u_\gamma^{\alpha' U_\gamma} \wedge
 \bar{\Omega}(\M_{D,>0}) \prod_{\gamma \in \Pi} \bar{u}_\gamma^{\alpha' \bar{U}_\gamma},
\end{gather*}
where in order for the integrand to be single-valued, we must have that the exponents $\bar{U}_\gamma$ differ from $U_\gamma$ at most by integers.

As we have remarked, it is plausible that real components of cluster configuration space other than the region associated with $u_\gamma \geq 0$ provide us with many different positive geometries $\M^{(i)}_{D}$, with associated canonical forms
$ \Omega\big(\M^{(i)}_{D}\big) $. In~this case we can extend the closed string integrals to be more generally labelled by pairs of these positive geometries,
\begin{gather*}
\I_{\M^{(i)}_{D},\M^{(j)}_{D}}^{\rm{closed}}\big(\big\{U,\bar{U}\big\}\big)=
\int_{\M_D(\C)} \Omega\big(\M^{(i)}_{D}\big) \prod_{\gamma \in \Pi} u_\gamma^{\alpha' U_\gamma} \wedge
 \bar{\Omega}\big(\M^{(j)}_{D}\big) \prod_{\gamma \in \Pi} \bar{u}_\gamma^{\alpha' \bar{U}_\gamma}.
\end{gather*}

It is clear that a complete understanding of the space of open and closed string integrals will go hand-in-hand with a similarly complete understanding of the space of all connected real components of the cluster configuration space.

\bigskip\noindent{\bf{\large 12.3.} Beyond finite type.}
Finally, the most obvious open question is whether the notions of cluster configuration space presented in this paper can naturally be extended beyond finite-type cluster algebras. It~is interesting to note that, as we have seen in \eqref{eq:IDu} above, in the finite type case, the introduction of the $u$-variables is naturally reverse engineered, starting from the definition of the cluster string amplitude, see also \cite{AHL}. This definition can be extended in~various ways to define natural ``compactifications" of the infinite-type configuration spaces, as~recently been explored for the case of Grassmannian cluster algebras \cite{ALS}. In~these examples, the reverse-engineering of $u$-variables does not work as it does in finite type: amongst other things the polytope capturing the combinatorics of the boundary structure in these cases is typically not simple. But there may be other choices of stringy integral that are more natural from the perspective of finding good $u$-variables and ``binary" realizations of general cluster configuration spaces.

\appendix

\section{A lemma in commutative algebra}

\begin{Lemma}\label{lem:ringiso}
Let $f\colon A \to B$ be a surjective homomorphism of Noetherian commutative rings with identity. Let~$S \subset A$ be the multiplicative set generated by elements $x_1,x_2,\dots,x_p$ such that
\begin{gather}\label{eq:nzd}
f(x_1),f(x_2),\dots,f(x_p)\quad \text{are not zero-divisors in $B$}.
\end{gather}
Suppose that
\begin{enumerate}\itemsep=0pt
\item[$(1)$]
the localized homomorphism $S^{-1} f\colon S^{-1}A \to S^{-1}B$ is an isomorphism, and
\item[$(2)$]
for each $i=1,2,\dots,p$ the induced homomorphism $f_i\colon A/(x_i) \to B/(f(x_i))$ is an isomorphism.
\end{enumerate}
Then $f$ is an isomorphism.
\end{Lemma}
\begin{proof}
Let $K$ denote the kernel of $f$. Let~$a \in K$ be a nonzero element. Suppose that $x_i a \neq 0$ in $A$ for all $i$. Then the image of $a$ in $S^{-1}A$ is nonzero and it is in the kernel of $S^{-1} f\colon S^{-1}A \allowbreak\to S^{-1}B$. This contradicts (1). Thus $M a = 0$ for some monomial $M$ in the $x_i$-s. Replacing $a$ by~$M'a$ for some other monomial $M'$, and using \eqref{eq:nzd}, we may assume that $a \in K$ and $x_ia=0$ for some $i =1,2,\dots,p$.

If $a \in (x_i)$, then by \eqref{eq:nzd}, we have $a = x_i a_1$ for a nonzero element $a_1 \in K$. Repeating, we~either find a nonzero element $a' \in K$ such that $a' \notin (x_i)$, or we have an ascending chain of ideals $(a) \subset (a_1) \subset (a_2) \subset \cdots$. In~the former case, the image of $a'$ in $A/(x_i)$ is nonzero and in the kernel of $f_i$, contradicting (2). Thus we are in the latter case. Since $x_i$ is not a~unit and $A$ is Noetherian, the chain of ideals stabilizes to a proper ideal $(a') = I \subsetneq A$, and we thus have $(a') = (a'')$, where $a'' = x_i a'$ and $x^n a' = 0$ for some $n > 0$. This is impossible: letting~$m$ be minimal such that $x_i^m a' = 0$ we find that $x_i^{m-1} a'' =0$ which implies $x_i^{m-1} a' = 0$, a~contradiction.
\end{proof}

\subsection*{Acknowledgements} We thank Mark Spradlin and Hugh Thomas for many discussions related to this work and for closely related collaborations. We~thank the anonymous referees for a number of corrections and helpful suggestions to the exposition. T.L.\ was supported
by NSF DMS-1464693, NSF DMS-1953852, and by a von Neumann Fellowship from the Institute for Advanced Study.
N.A-H.\ was supported by DOE grant DE-SC0009988. S.H.\ was supported in part by the National Natural Science Foundation of China under Grant No.\ 11935013, 11947301, 12047502, 12047503.

\pdfbookmark[1]{References}{ref}
\LastPageEnding

\end{document}